\documentclass[a4paper,10pt,twoside]{amsart}

\usepackage[affil-it]{authblk}
\usepackage{titling}
\usepackage[latin1]{inputenc}
\usepackage{calligra}
\usepackage[OT1]{fontenc}
\usepackage[english]{babel}
\usepackage{amsfonts}
\usepackage{amsmath}
\usepackage{amssymb}
\usepackage{amsthm}
\usepackage{faktor}
\usepackage{float}
\usepackage{enumerate}
\usepackage{color}
\usepackage{pdfpages}
\usepackage{esint}
\usepackage{hyperref}
\usepackage{cite}

\newcommand\testname{Abstract}
\makeatletter
\newenvironment{abs}{%
    \small
    \begin{center}%
        {\textsc \testname\vspace{-.2em}\vspace{\z@}}%
    \end{center}%
    \quote
    }
   {\endquote}
\makeatother

\newcommand{\tr}{{}^\mathrm{t} }
\newcommand{\Div}{\mathrm{div}}
\newcommand{\PP}{\mathbb{P}}
\newcommand{\ee}{\varepsilon}

\newcommand{\Aa}{\mathcal{A}}
\newcommand{\Cc}{\mathcal{C}}
\newcommand{\Bb}{\mathcal{B}}
\newcommand{\dd}{\mathrm{d}}

\newcommand{\MM}{\mathcal{M}}

\newcommand{\DD}{\mathfrak{D}}
\newcommand{\RR}{\mathbb{R}}

\newcommand{\NN}{\mathbb{N}}

\newcommand{\BB}{\dot{B}}
\newtheorem{theorem}{Theorem}[section]
\newtheorem{cor}{Corollary}[theorem]
\newtheorem{prop}[theorem]{Proposition}
\newtheorem{lemma}[theorem]{Lemma}
\newtheorem{definition}[theorem]{Definition}
\newtheorem{remark}[theorem]{Remark}

\title{\Large 
	\textbf{\uppercase{
		Global Weak Solutions for Boussinesq 
		System with Temperature dependent Viscosity and\\
		bounded Temperature }}}
	
\author{Francesco De Anna$\quad$}
\affil{	\textsc{Universit\'e  Bordeaux} \\ 
		\small\textsc{Institut de Math\'ematiques de Bordeaux}\\ 
		\small{F-33405 Talence Cedex, France}
		\\ \vspace{0cm}\\
		\small\textnormal{Francesco.Deanna@math.u-bordeaux1.fr}}
\date{September 22, 2015}

\begin{document}
\maketitle
\begin{abs}
   In this paper we obtain a result about the global existence of weak solutions for the 
   $d$-dimensional Bussinesq system, with viscosity dependent on 
   temperature. The initial temperature is just supposed to be bounded, while the initial velocity 
   belongs to some critical Besov Space, invariant to the scaling of this system. We suppose the viscosity 
   close enough to a positive constant, and the $L^\infty$ norm of their difference plus the Besov norm of 
   the horizontal component of the initial velocity is supposed to be exponentially small with respect to 
   the vertical component of the initial velocity. On Preliminaries and
   in the appendix we consider some $L^p L^q$ regularity Theorems for the heat kernel, which play 
   an important role in the main proof of this article.
\end{abs}

\section{Introduction}

\noindent 
The general Boussinesq system turns out from a first approximation of a coupling system related to the Navier-Stokes and the thermodynamic equations. In such approximation, if we consider the structural coefficients to be constant, as for example the viscosity, we obtain a system between two parabolic equations with linear second order operators. Nevertheless, several fluids cannot be modeled in this way, for instance if we want to study the plasma evolution. Hence it should be necessary to consider a class of quasilinear parabolic systems coming from the general Boussinesq one.
This paper is devoted to the global existence of solutions for the Cauchy problem related to one of these models, namely: 
\begin{equation}\label{Navier_Stokes_system}
	\begin{cases}
		\;\partial_t\theta + \Div\, (\theta u)=0									& \RR_+ \times\RR^d,\\
		\;\partial_t u + u\cdot \nabla u -\Div\, (\nu(\theta)\MM) +\nabla\Pi=0	& \RR_+ \times\RR^d,\\
		\;\Div\, u = 0																& \RR_+ \times\RR^d,\\
		\;(u,\,\theta)_{|t=0} = (\bar{u},\,\bar{\theta})							& \;\;\quad \quad\RR^d,\\
	\end{cases}
\end{equation}
where $\MM$ is defined by $\nabla u + \tr\nabla u$.
Here $\theta$, $u=(u^1,\dots,u^d)$ and $\Pi$ stand for the temperature, velocity field and pressure of the fluid respectively, depending on the time variable $t\in \RR_+=[0,+\infty)$ and on the space variables $x\in \RR^d$. We denote by $u^h:=(u^1,\dots,u^{d-1})$ the horizontal coordinates of the velocity field, while $u^d$ is the vertical coordinate. Furthermore $\nu(\cdot)$ stands for the viscosity coefficient, which is a smooth positive function on $\RR_+$.
Such system is useful as a model to describe many geophysical phenomena, like, for example, a composed obtained by mixing several incompressible immiscible fluids. Indeed the temperature fulfills a transport equation, while the velocity flow verifies a Navier-Stokes type equation which describes the fluids evolution. We consider here the case where the viscosity depends on the temperature, which allows to characterize the immiscibility hypotheses.

\subsubsection*{Some Developments in the Boussinesq System}
The general Boussinesq system, derived in \cite{MR0145789}, assumes the following form:
\begin{equation}\label{general_Boussinesq}
	\begin{cases}
		\;\partial_t\theta + \Div\, (\theta u)-\Delta \varphi (\theta)+ |D|^s\theta =0			
																					& \RR_+ \times\RR^d,\\
		\;\partial_t u + u\cdot \nabla u -\Div\, (\nu(\theta)\MM) +\nabla\Pi=F(\theta)
																					& \RR_+ \times\RR^d,\\
		\;\Div\, u = 0																& \RR_+ \times\RR^d,\\
		\;(u,\,\theta)_{|t=0} = (\bar{u},\,\bar{\theta})							& \;\;\quad \quad\RR^d,\\
	\end{cases}
\end{equation}
An exhaustive mathematical justification of such system as a model of stratified fluids (as atmosphere or oceans) is given by Danchin and He in \cite{MR3134744}. We present here a short (and of course incomplete) overview concerning some some well-posedness results. 

\noindent Provided by some technical hypotheses, in \cite{MR1642033} D\'iaz and Galiano establish the global existence of weak solution for system \ref{general_Boussinesq} when $s=0$. Moreover they achieve the uniqueness of such solutions in a two dimensional domain, assuming the viscosity $\nu$ to be constant.

\noindent In \cite{MR2305876} Hmidi and Keraani study system \eqref{general_Boussinesq} in a two dimensional setting, when the parameter $s$ is null, $\varphi(\theta)=\theta$ and $F(\theta)$ stands for a Buoyancy force, more precisely they considered $F(\theta)=\theta e_2$, with $e_2$ the classical element of the canonical basis of $\RR^2$. They prove the global existence of weak solutions when both the initial data belong to $L^2(\RR^2)$. Furthermore, they establish the uniqueness of such solutions under an extra regularity on the initial data, namely $H^r(\RR^2)$, for $r>0$.

\noindent In \cite{MR2822227} Wang and Zhang consider system \ref{general_Boussinesq} with Buoyancy force and constant viscosity, when the temperature $\theta$ satisfies
\begin{equation*}
	\partial_t\theta + \Div\, (\theta u)-\Div( k\nabla \theta) =0,
\end{equation*}
where $k$ stands for the thermal diffusivity, which also depends on the temperature. They prove existence and uniqueness of global solutions when the initial data belong to $H^r(\RR^2)$, for $r>0$.

\noindent In \cite{MR2227730} Chae considered system \eqref{general_Boussinesq} in two dimension, with constant viscosity and when $\varphi(\theta)$ is equal to $\theta$ or $0$. In this case the author establish the existence of smooth solutions.

\noindent System \eqref{general_Boussinesq} has also given interest in the Euler equation framework, when the viscosity $\nu$ is supposed to be null. In this direction, Hmidi, Keerani and Rousset \cite{MR2763332} develop the existence and uniqueness of a solutions when $s=1$, provided that the initial velocity belongs to $\BB_{\infty,1}^{1}\cap \dot{W}^{1,p}_x$  while the initial temperature lives in $\BB_{\infty,1}^{0}\cap L^p_x$.

\noindent In \cite{MR2290277} Abidi and Hmidi perform an existence and uniqueness result for system  \eqref{general_Boussinesq} in two dimension, when $\varphi\equiv 0$, $s=0$ and the force $F(\theta)=\theta e_2$. Here, the initial velocity is supposed in $L^2\cap \BB_{\infty,1}^{-1}$ and the temperature belongs to $\BB_{2,1}^0$.

\noindent In \cite{MR2520505} Paicu and Danchin consider the case of constant viscosity. Given a force $F(\theta)=\theta e_2$, imposing $s=2$ and $\phi=\theta$, the authors perform a global existence result for system \eqref{general_Boussinesq}, on the condition that the initial data are of Yudovich's type, namely the initial temperature is in $L^2_x\cap \BB_{p,1}^{-1}$, the initial velocity is in $L^2_x$ and the initial vorticity $\partial_1 \bar{u}_{2}-\partial_2 \bar{u}_{1}$ is bounded and belongs to some Lebesgue space $L^r_x$ with $r\geq 2$.

\noindent We mention that a no constant viscosity has also been treated in the study of the inhomogeneous incompressible Navier Stokes equation with variable viscosity
\begin{equation}\label{System2_intro}
	\begin{cases}
		\;\partial_t \rho + \Div\, (\rho u)=0										& \RR_+ \times\RR^d,\\
		\;\partial_t (\rho u) + \Div\{\rho u\otimes u\} -\Div\, (\eta(\rho)\MM) +\nabla\Pi=f
																					& \RR_+ \times\RR^d,\\
		\;\Div\, u = 0																& \RR_+ \times\RR^d,\\
		\;(u,\,\rho)_{|t=0} = (\bar{u},\,\bar{\rho})							& \;\;\quad \quad\RR^d.\\
	\end{cases}
\end{equation}
\noindent In \cite{MR2336833} Abidi and Paicu analyze the global well-posedness of \eqref{System2_intro} in certain critical Besov spaces provided that the initial velocity is small enough and the initial density is strictly close to a positive constant.

\noindent In \cite{2013arXiv1301.2371A} Abidi and Zhang establish the existence and uniqueness of global solutions for system \eqref{System2_intro}, on the condition that the initial velocity belongs to $H^{-2\delta}\cap H^1$, for some $\delta\in  (0,1/2)$, the initial density lives in $L^{2}_x\cap \dot{W}^{1,r}_x$, with $r\in (2,2/(1-2\delta)\,)$, and $\bar{\rho}-1$ belongs to $L^2_x$.

\noindent We finally mention that in \cite{2012arXiv1212.3918H} Huang and Paicu investigate the time decay behavior of weak solutions for \eqref{System2_intro} in a two dimensional setting.

\vspace{0.2cm}
\noindent In this paper we are going to study the global existence of solutions for the system \eqref{Navier_Stokes_system} concerning standard and natural conditions on the initial data: the initial temperature is only assumed to be bounded and the initial velocity field is supposed to belong to certain critical homogeneous Besov space. More precisely we consider
\begin{equation}\label{initial_data}
	\bar{\theta}\in L^\infty_x\quad\text{and}\quad \bar{u} \in \BB_{p,r}^{\frac{d}{p}-1}\quad\text{with}\quad r\in (1,\infty)\quad\text{and}\quad p\in (1,d).
\end{equation}
\begin{remark}
As the classical Navier-Stokes equation, system \eqref{Navier_Stokes_system} has also a scaling property, more precisely if $(\theta, u, \Pi)$ is a solution then, for 
all $\lambda>0$, 
\begin{equation*}
	(\theta(\lambda^2 t, \lambda\,x), \lambda\, u(\lambda^2 t, \lambda\,x), \lambda^2\, \Pi(\lambda^2 t, \lambda\,x))
\end{equation*}
is also solution of \eqref{Navier_Stokes_system}, with initial data $ (\bar{\theta}(\lambda\,x), \lambda \,\bar{u}(\lambda\,x))$. Hence it is natural to consider the initial data in a Banach space with a norm which is invariant under the previous scaling, as for instance $L^\infty_x\times \BB_{p,r}^{d/p-1}$. Let us remark that this initial data type allows $\theta$ to include discontinuities along an interface, an important physical case as a model that describes a mixture of fluids with different temperatures.
\end{remark}

\noindent From here on we suppose the viscosity $\nu$ to be a bounded smooth function, close enough to a positive constant $\mu$, which we assume to be $1$ for the sake of simplicity. Then, we assume the following small condition for the initial data to be fulfilled:
\begin{equation}\label{smallness_condition}
		\eta := 
			\big( 
				\| \nu - 1 \|_\infty + \|\bar{u}^h\|_{\dot{B}_{p,r}^{-1+\frac{d}{p}}}
			\big)
			\exp\Big\{ c_r \|\bar{u}^d\|_{\dot{B}_{p,r}^{-1+\frac{d}{p}}}^{4r} \Big\}
			\leq c_0.
\end{equation}
where $c_0$ and $c_r$ are two suitable positive constants. This sort of initial condition is not new in literature, for instance it appears in \cite{MR3056619}, where Huang, Paicu and Zhang  study of an incompressible inhomogeneous fluid in the whole space with viscosity dependent on the density, and moreover in \cite{2013arXiv1304.3235D}, where Danchin and Zhang examine the same fluid typology, in the half-space setting.

\vspace{0.2cm}
\noindent Before enunciating our main results, let us recall the meaning of weak solution for system \eqref{Navier_Stokes_system}:
\begin{definition}\label{definition_weak_global_solution}
	We call $(\theta ,\,u,\,\Pi)$ a global weak solution of 
	\eqref{Navier_Stokes_system} if
	\begin{itemize}
		\item[(i)]	for any test function $\varphi\in \DD(\RR_+\times \RR^d)$, the following identities are 
		well-defined and fulfilled:
		\begin{equation*}
		\int_{\RR_+} \int_{\RR^d}
		\{\theta \left( \partial_t \varphi + u \cdot \nabla \varphi \right)\}(t,x)\dd x\,\dd t + 
		\int_{\RR^d}
		\bar{\theta}(x)\varphi(0,x)\dd x = 0\text{,}
		\end{equation*}					
		\begin{equation*}
			\int_{\RR_+} \int_{\RR^d}
			\{u\cdot\nabla \varphi\}(t,x) \dd x\,\dd t = 0\text{,}	 
		\end{equation*}
		\item[(ii)]	for any vector valued function
					$\Phi=(\Phi_1,\dots,\Phi_d)\in \DD(\RR_+\times \RR^d)^d$  
					the following equality is well-defined and satisfied:
					\begin{equation*}
						\int_{\RR_+} \int_{\RR^d}\{
							u\cdot \partial_t\Phi - (u\cdot \nabla u)\cdot \Phi 
							- \nu (\theta) \MM \cdot \nabla \Phi 
							+ \Pi\, \Div\, \Phi\}(t,x) \dd x\, \dd t
							+ \int_{\RR^d}
								\bar{u}(x)\cdot \Phi(0,x) \dd x = 0\text{,}
					\end{equation*}
	\end{itemize}
\end{definition}

\subsubsection*{The smooth case} Some regularizing effects for the heat kernel, like the well-known $L^pL^q$-Maximal Regularity Theorem (see Theorem \ref{Maximal_regularity_theorem}), play an key role in our proof as well as an useful homogeneous Besov Spaces characterization (see Theorem \ref{Characterization_of_hom_Besov_spaces} and Corollary \ref{Cor_Characterization_of_hom_Besov_spaces}). Indeed, we can reformulate the momentum equation of \eqref{Navier_Stokes_system} in the following integral form:
\begin{equation}
	u(t) = e^{t\Delta}\bar{u}  +	\int_0^t e^{(t-s)\Delta} \big\{ -u\cdot \nabla u+\nabla\Pi\big\}(s)\dd s+
							\int_0^t \Div\, e^{(t-s)\Delta}\big\{ \big((\nu(\theta)-1\big)\MM\big\}(s)\dd s.
\end{equation} 
Thus, it is reasonable to assume the velocity $u$ having the same regularity of the heat kernel convoluted with the initial datum $\bar{u}$. The Maximal Regularity Theorem suggests us to look for a solution in a $L^{\bar{r}}_t L^q_x$ setting. Now, in the simpler case where $u$ just solves the heat equation with initial datum $\bar{u}$, having $\nabla u$ in some $L^{\bar{r}}_t L^q_x$ is equivalent to $\bar{u}\in \BB_{q,\bar{r}}^{d/q-1}$ on the condition $N/q-1=1-2/\bar{r}$. From the immersion $\BB_{p,r}^{d/p-1}\hookrightarrow \BB_{q,\bar{r}}^{d/q-1}$, for every $\bar{q}\geq p$ and $\bar{r}\geq r$, we deduce that this strategy requires $p\leq dr/(2r-1)$. Then, according to the above heuristics, our first result reads as follows:
\begin{theorem}\label{Main_Theorem}
	Let $r\in (1,\infty)$ and $p\in (1,dr/(2r-1))$. Suppose that the initial data $(\bar{\theta},\,\bar{u})$ belongs to $L^infty_x\times \BB_{p,r}^{d/p-1}$. 
	There exist two positive constants $c_0$, $c_r$ such that, if the smallness condition \eqref{smallness_condition} is fulfilled, then 
	there exists a global weak solution $(\theta,\,u,\,\Pi)$  of \eqref{Navier_Stokes_system}, in the sense of definition \ref{definition_weak_global_solution} such that
	\begin{equation*}
		u 			\in L^{2r}_tL^{\frac{dr}{r-1}}_x,\quad\quad
		\nabla u	\in L^{2r}_tL^{\frac{dr}{2r-1}}_x\cap 
						L^{r}_tL^{\frac{dr}{2(r-1)}}_x\quad\text{and}\quad
		\Pi\in L^{r}_t L^{\frac{dr}{2(r-1)}}_x.
	\end{equation*}
	Furthermore, the following inequalities are satisfied:	
	\begin{align*}
		& \| \nabla u^h \|_{L^{2r}_t L^{ \frac{dr}{2r-1}}_x}+ 
			\|\nabla u^h\|_{L^{r}_t L^{\frac{dr}{2(r-1)}}_x} +
			\| u^h \|_{L^{2r}_t L^{\frac{dr}{r-1}}_x} 
					\leq C_1\eta,\\
		&	\|\nabla u^d\|_{L^{2r}_t L^{\frac{dr}{2r-1}}_x}  + 
			\|\nabla u^d\|_{L^{r}_t L^{\frac{dr}{2(r-1)}}_x} +
			\| u^d \|_{L^{2r}_t L^{\frac{dr}{r-1}}_x} 
					\leq C_2\| \bar{u}^d \|_{\dot{B}_{p,r}^{\frac{d}{p}-1}} + C_3,\\
		&	\|\Pi\|_{L^{r}_t L^{\frac{dr}{2(r-1)}}_x} 
					\leq 
					C_4\eta,\quad 
			\| \theta \|_{L^\infty_{t,x}}\leq \| \bar{\theta} \|_{L^\infty_x}.
	\end{align*}
	for some positive constants $C_1$, $C_2$, $C_3$ and $C_4$.
\end{theorem}
\subsubsection*{The general case} As we have already pointed out, the choice of a $L^{\bar{r}}_t L^{q}_x$ functional setting requires the condition $p< dr/(2r-1)$. The remaining case 
$dr/(2r-1)\leq p<d$ can be handled by the addiction of a weight in time. Indeed, in the simpler case where $u$ just solves the heat equation with initial datum $\bar{u}$, having $u$ in some $\BB_{p_3,\bar{r}}^{d/p_3-1}$ for some $p_3\in (dr/(r-1),\infty)$ is equivalent to $t^{1/2( 1-d/p_3)-1/\bar{r})}u\in L^{\bar{r}}_t L^{p_3}_x$. In the same line having $\nabla \bar{u}$ in a suitable Besov space $\BB_{p_2,\bar{r}}^{d/p_2-1}$ is equivalent to have $t^{1/2( 2-d/p_3)-1/\bar{r})}u$ in $L^{\bar{r}}_t L^{p_2}_x$. 
Hence, reformulating the smallness condition \eqref{smallness_condition} by
\begin{equation}\label{smallness_condition_general_case}
		\eta := 
			\big( 
				\| \nu - 1 \|_\infty + \|\bar{u}^h\|_{\dot{B}_{p,r}^{-1+\frac{d}{p}}}
			\big)
			\exp\Big\{ c_r \|\bar{u}^d\|_{\dot{B}_{p,r}^{-1+\frac{d}{p}}}^{2r} \Big\}
			\leq c_0,
\end{equation}
with similar heuristics proposed in the first case, our second results reads as follows:
\begin{theorem}\label{Main_Theorem2}
	Let $p,\,r$ be two real numbers in $ (2d/3,\,d)$ and $(1,\infty)$ respectively, such that
	\begin{equation}\label{Main_Thm_r_restrictions}
	\frac{2}{3}\frac{d}{p}-\frac{d}{6p}<\frac{1}{2}-\frac{1}{2r},\quad
	\frac{1}{r}<\frac{1}{3}\big(\frac{d}{p}-1\big), \quad
	\frac{1}{r}<\frac{4}{3}-\frac{d}{p}.	
	\end{equation}
	Let us define $p_2:= 3pd/(2p+d)$ and $p_3:=3p^*/2 = 3pd/(2d-2p)$, so that $1/p=1/p_2 +1/p_3$ and 
	\begin{equation*}
		\alpha:=\frac{1}{2}\big(3-\frac{d}{p_1} \big) -\frac{1}{r},\quad
		\beta:=\frac{1}{2}\big(2-\frac{d}{p_2} \big) -\frac{1}{2r},\quad
		\gamma_1:=\frac{1}{2}\big(1-\frac{d}{p_3} \big)-\frac{1}{2r},\quad
		\gamma_2:=\frac{1}{2}\big(1-\frac{d}{p_3} \big).	
	\end{equation*}	 
	There exist two positive constants $c_0$ and $c_r$ such that, if the smallness condition \eqref{smallness_condition_general_case} is fulfilled, then 
	there exists a global weak solution $(\theta,u, \Pi)$  of \eqref{Navier_Stokes_system}, 
	in the sense of definition \ref{definition_weak_global_solution} such that
	\begin{equation*}
		t^{\gamma_1} u 	\in L^{2r}_tL^{p_3}_x,\quad
		t^{\gamma_2} u 	\in L^{\infty}_tL^{p_3}_x\quad
		t^{\beta}\nabla u	\in L^{2r}_tL^{p_2}_x\quad
		t^{\alpha}\Pi\in L^{r}_t L^{p^*}_x.
	\end{equation*}
	Furthermore, the following 
	inequalities are satisfied:	
	\begin{equation}\label{Main_Theorem2_inequalities}
	\begin{aligned}
		 &	\|t^{\alpha}\nabla u^h\|_{L^{2r}_t L^{p^*}_x}+
			\|t^{\beta}\nabla u^h\|_{L^{2r}_t L^{p_2}_x}+
			\|t^{\gamma_1} u^h \|_{L^{2r}_t L^{p_3}_x}+ 
			\|t^{\gamma_2} u^h \|_{L^{\infty}_t L^{p_3}_x}			
			\leq C_1\eta,\\&
			\|t^{\alpha}\nabla u^d\|_{L^{2r}_t L^{p^*}_x}+
			\|t^{\beta}\nabla u^d\|_{L^{2r}_t L^{p_2}_x}+ 
			\|t^{\gamma_1} u^d \|_{L^{2r}_t L^{p_3}_x}+ 
			\|t^{\gamma_2} u^d \|_{L^{\infty}_t L^{p_3}_x}		
			\leq C_2\| \bar{u}^d \|_{\BB_{p,r}^{\frac{d}{p}-1}} + C_3\\&
			\|t^\alpha\Pi\|_{L^{r}_t L^{p^*}_x} 
					\leq 
					C_4\eta,\quad 
			\| \theta \|_{L^\infty_{t,x}}\leq \| \bar{\theta} \|_{L^\infty_x}.
	\end{aligned}
	\end{equation}
	for some positive constants $C_1$, $C_2$ and $C_3$.
\end{theorem}
\begin{remark}
	We remark that the conditions on $p$ and $r$ in Theorem \ref{Main_Theorem2} are not restrictive. Indeed, we can always embed $\BB_{p,r}^{d/p-1}$ into $\BB_{q,r}^{d/q-1}$  with 
	$q\geq p$ which satisfies $q\in (2d/3, d)$ (see Theorem \ref{Theorem_embedding_Besov}). Moreover $\BB_{p,r}^{d/p-1}$ is embedded in $\BB_{p,\tilde{r}}^{d/p-1}$, with $\tilde{r}\geq r$, 
	then there is no lost of generality assuming the inequalities \eqref{Main_Thm_r_restrictions}.
\end{remark}

\noindent Let us briefly describe the organization of this paper. In the second section we recall some technical Lemmas concerning the regularizing effects for the heat kernel, as the Maximal regularity Theorem, which will play an important role in the main proofs. We also mention some results regarding the characterization of the homogeneous Besov Spaces. In the third section we prove the existence of solutions for \eqref{Navier_Stokes_system}, with stronger conditions on the initial data with respect to the ones of Theorem \ref{Main_Theorem}. In the fourth section we regularize our initial data by the dyadic partition, and, using the results of the third section with a compactness argument, we conclude the proof of Theorem \ref{Main_Theorem}. In the fifth and sixth sections we perform to the proof of Theorem \ref{Main_Theorem2}, proceeding with a similar structure of the third and fourth sections. 

\begin{remark}
\noindent In order to obtain the uniqueness about the solution of \eqref{Navier_Stokes_system}, the more suitable strategy is to reformulate our system by Lagrangian coordinates, following for example \cite{MR3056619}, \cite{2013arXiv1304.3235D} and \cite{MR3017294}. The existence of such coordinates may be achieved supposing the velocity field with Lipschitz space condition, more precisely claiming $u$ belongs to $L^1_{loc}(\RR_+;{Lip}_x)$, or equivalently 
$\nabla u \in L^{1}_{loc}(\RR_+;L^{\infty}_x)$. If we want to obtain this condition without controlling two derivatives of $u$ (in the same line of the existence part) and then without using Sobolev embedding, we need to bound terms like
\begin{equation}\label{problem_uniqueness_introduction1}
	\int_0^t \Delta e^{(t-s)\Delta}\big\{ \big((\nu(\theta)-1\big)\nabla u\big\}(s)\dd s
\end{equation}
in some $L^{s}(0,T;L^\infty_x)$ space, with $s>1$. Unfortunately this is not allowed by the Maximal Regularity Theorem \ref{Maximal_regularity_theorem} for the heat kernel, because of the critical exponents of this spaces. Then, we need to impose an extra regularity for the initial temperature, as $\nabla \bar{\theta}\in L^{l_1}_x$, for an opportune $l_1$, in order to obtain $\nabla \theta$ in $L^1_{loc}(\RR_+;{L^{l_1}_x})$ and then to split \eqref{problem_uniqueness_introduction1} into
\begin{equation}
	\int_0^t \Div\, e^{(t-s)\Delta} \big\{\nu'(\theta)\nabla \theta\cdot \nabla u\big\} (s)\dd s + 
	\int_0^t \Div\, e^{(t-s)\Delta} \big\{\big( \nu(\theta)-1\big)\nabla^2 u\big\} (s)\dd s.
\end{equation}
Hence we need to control the norm of $\nabla^2 u$ in some $L^{r_1}(0,T;L^{l_2}_x)$, with $r_1>1$ and also $l_2>d$ in order to fulfill the Morrey Theorem's hypotheses. It is necessary to do that starting from the approximate systems of the third section, however the only way to control two derivatives of the approximates solutions with some inequalities independent by the indexes $n\in \NN$ and $\varepsilon>0$ (present in the extra term of the perturbed transport equation) is to impose 
$\nabla\bar{\theta}\in L^{l_1}_x$ with $l_1>d$. We conjecture that this is not the optimal condition for the initial data in order to obtain the uniqueness, indeed, inspired by \cite{MR2336833}, we claim that, supposing $\nabla \bar{\theta}\in L^d_x$ and $\bar{u}\in$ {\small $\dot{B}_{p,1}^{-1+\frac{d}{p}}$}, it is possible to prove the uniqueness with the velocity field into the space
\begin{equation*}
	L^\infty_t \dot{B}_{p,1}^{-1+\frac{d}{p}}\cap L^1_t \dot{B}_{p,1}^{1+\frac{d}{p}}. 
\end{equation*}
However this needs to change the structure of the existence part, more precisely to change the functional space where we are looking for a solution. Since in our Theorem we suppose only the initial temperature to be bounded, then we have decided to devote this paper only to the existence part of a global weak solution for system \eqref{Navier_Stokes_system}.
\end{remark}

\section{Preliminaries}

\noindent The purpose of this section is to present some lemmas concerning the regularizing effects for the heat kernel, which will be useful for the next sections. At first step let us recall the well-known Hardy-Littlewood-Sobolev inequality, whose proof is available in  \cite{MR2768550}, Theorem $1.7$.

\begin{theorem}[Hardy-Littlewood-Sobolev inequality]\label{HLS_Theorem}
	Let $f$ belongs to $L^p_x$, with $1< p <\infty $, $\alpha\in ]0,d[$ and suppose $r\in ]0,\infty[$ satisfies
	$1/p+\alpha/d= 1+1/r$. Then $|\cdot |^{-\alpha}* f$ belongs to $ L^r_x$ and there exists a positive constant $C$ such that
	$
		\left\|
			|\cdot |^{-\alpha}* f
		\right\| _{L^r_x}
		\leq 
		C
		\left\| 
			f
		\right\|_{L^p_x}
	$.
\end{theorem}

\noindent
From this Theorem we can infer the following corollary.

\begin{cor}\label{Hardy-Litlewood-Sobolev-Inequality}
	Let $f$ belongs to $L^p_x$, with $1< p <d $ and let $(\sqrt{-\Delta})^{-1}$ be the Riesz potential, 
	defined by $(\sqrt{-\Delta})^{-1}f(\xi) := \mathcal{F}^{-1}( \hat{f}(\xi)/|\xi|)$. Then $(\sqrt{-\Delta})^{-1}f$ belongs to $L^{dp/(d-p)}_x$
	and there exists a positive constant $C$ such that $\|(\sqrt{-\Delta})^{-1}f\|_{L^{pd/(d-p)}_x}\leq C\|f \|_{L^p_x}$.
\end{cor}
\begin{proof}
	From the equality $	(\sqrt{-\Delta})^{-1}f(x)=c(|\cdot |^{-d+1}* f)(x)$, for almost every $x\in\RR^d$ and for an appropriate constante $c$, the theorem is a direct consequence of 
	Theorem \ref{HLS_Theorem}, considering $\alpha =d-1$. 
\end{proof}

\noindent
One of the key ingredient used in the proof of Theorem \eqref{Main_Theorem} is the maximal regularity Theorem for the heat kernel. We recall here the statement (see \cite{MR1938147}, theorem 7.3).

\begin{theorem}[Maximal $L^p(L^q)$ regularity for the heat kernel]\label{Maximal_regularity_theorem}
	Let $T\in ]0,\infty]$, $1<p,q<\infty$ and $f\in L^p(0,T;L^q_x)$. Let the operator $A$ be 
	defined by
	\begin{equation*}
		Af(t,\cdot):=\int_0^t \Delta e^{(t-s)\Delta}f(s,\cdot)\dd s\text{.}
	\end{equation*}
	Then $A$ is a bounded operator from $L^p(0,T;L^q_x)$ to $L^p(0,T;L^q_x)$.
\end{theorem}
\noindent
If instead of $\Delta$ on the definition of the operator $A$ we consider $\nabla$ (the operator $B$ of Lemmas \ref{Lemma1} and \ref{Lemma2} ) or even without derivatives (the operator $C$ of Lemma \ref{Lemma3}), then we can obtain similar results with respect to the maximal regularity Theorem, using a direct computation. We present here the proofs. At first step let us recall two useful identities: 
\begin{remark}\label{remark2.1} 
	Let us denote by $K$ the heat kernel, defined by $K(t,x)=e^{-|x|^2/(4t)}/(2\pi t)^{d/2}$, then 
	$\|K(t,\cdot)\|_{L^q_x}=\|K(1,\cdot)\|_{L^q_x}/t^{d/(2q')}$, for all $1\leq q <\infty$. Moreover considering the gradient of the heat kernel, 
	$\Omega(t,x):=\nabla K(t,x)=-xK(t,x)/(2t)$, we  have $\|\Omega(t,\cdot)\|_{L^q_x}=\|\Omega(1,\cdot)\|_{L^q_x}/|t|^{d/(2q')+1/2}$.
\end{remark}
\noindent
Let us denote by $R:= \tr (R_1,\dots,R_d)$, where $R_j$ is the Riesz transform over $\RR^d$, defined by
\begin{equation*}
		R_jf:= \mathcal{F}^{-1}\left(-i \frac{\xi_j}{|\xi|}\hat{f}\right)\text{.}
\end{equation*}
we recall that $R_j$ is a bounded operator from $L^q_x$ to itself, for every $1<q<\infty$ (for more details we refer to \cite{MR1938147}).

\begin{lemma}\label{Lemma1}
	Let $T\in\, ]0,\infty]$ and $f\in L^r(0,T;L^p_x)$, with $1<p<d$ and $1<r<\infty$. 
	Let the operator $B$ be defined by
	\begin{equation*}
		\mathcal{B}f(t,\cdot)\doteq \int_0^t \nabla e^{(t-s)\Delta}f(s,\cdot)\dd s\text{,}
	\end{equation*}
	Then $\mathcal{B}$ is a bounded operator from $L^r(0,T;L^p_x)$ to $L^r(0,T;L^{\frac{dp}{d-p}}_x)$.
\end{lemma}
\begin{proof}
	From corollary \ref{Hardy-Litlewood-Sobolev-Inequality} we have that, for almost every $s\in (0,T)$,
	\begin{equation*}
		\big(\sqrt{-\Delta}\big)^{-1}f(s) \in L^{\frac{dp}{d-p}}_x.
	\end{equation*}
	Then, reformulating $\mathcal{B}$ by
	\begin{equation*}
		\mathcal{B}f(t,\cdot)= -\int_0^t \Delta e^{(t-s)\Delta}
		R\big(\sqrt{-\Delta}\big)^{-1}f(s,\cdot)\dd s\text{,}
	\end{equation*}
	we deduce, by theorem \ref{Maximal_regularity_theorem}, that 
	$\mathcal{B}f\in L^r( 0,T ;L^{\frac{dp}{d-p}}_x)$ and
	\begin{equation*}
		\left\|\,\mathcal{B}f\,\right\|_{L^r( 0,T ;L^{\frac{dp}{d-p}}_x)}
		\leq C_1 \big\|\,R(\sqrt{-\Delta})^{-1}f\,\big\|_{L^r( 0,T ;L^{\frac{dp}{d-p}}_x)}
		\leq C_2 \left\|\,f\,\right\|_{L^r( 0,T ;L^p_x)}
		\text{,}
	\end{equation*}
	for opportune positive constant $C_1$ and $C_2$.
\end{proof}

\begin{lemma}\label{Lemma2}
	Let $T\in\, ]0,\infty]$ and $f\in L^r(0,T;L^{p}_x)$, with $1<r<\infty$ and $p\in [1,\frac{dr}{r-1}]$. 
	Let the operator $\mathcal{B}$ be defined as in Lemma \ref{Lemma1}.
	Then, we have that $\mathcal{B}$ is a bounded operator from 
	$L^r(0,T;L^{p}_x)$ with values to $L^{2r}(0,T;L^{q}_x)$, where $1/q:=1/p-(r-1)/(dr)$.
\end{lemma}
\begin{proof}
	Observe that, for every $t\in \RR_+$,
	\begin{equation*}
		\big\|\int_0^t \nabla e^{(t-s)\Delta}f(s)\dd s\,\big\|_{L^q_x}
		\leq \int_0^t \|\,\Omega(t-s,\cdot)*f(s,\cdot)\,\|_{L^q_x}\dd s
		\leq \int_0^t \|\,\Omega(t-s)\|_{L^{\tilde{q}}_x}\|\,f(s)\,\|_{L^{p}_x}\dd s,
	\end{equation*}
	with $1/\tilde{q}+1/p=1/q+1$ or equivalently $\tilde{q}'=dr/(r-1)$. Recalling Remark \ref{remark2.1}, we obtain
	\begin{equation*}
	\big\|\int_0^t \nabla e^{(t-s)\Delta}f(s)\dd s\,\big\|_{L^q_x}
	\leq 
	C
	\int_0^t \frac{\quad\|\,f(s)\,\|_{L^{p}_x}}{\quad|t-s|^{\frac{2r-1}{2r}}}\dd s
	\leq
	C
	\int_{\RR} \frac{\quad\|\,f(s)\,\|_{L^{p}_x}}
	{\quad|t-s|^{\frac{2r-1}{2r}}}1_{(0,T)}(s)\dd s.
	\end{equation*}
	Since by Theorem \ref{HLS_Theorem}
	\begin{equation*}
		|\cdot|^{-\frac{2r-1}{2r}}*\|f(\cdot)1_{(0,T)}(\cdot)\|_{L^p_x}\in L^{2r}_t,	
	\end{equation*}
	then there exists $\tilde{C}>0$ such that
	\begin{equation*}
		\|\,\mathcal{B}f\,\|_{L^{2r}(0,T;L^{q}_x)}
		\leq C
		\big\|\;|\cdot|^{-\frac{2r-1}{2r}}*\|f(\cdot)1_{(0,T)}(\cdot)\|_{L^p_x}\big\|_{L^{2r}_t}
		\leq
		\tilde{C} \|\,f\,\|_{L^{r}(0,T;L^{p}_x)}
	\end{equation*}
\end{proof}

\begin{lemma}\label{Lemma3}
	Let $T\in\, ]0,\infty]$, $r\in (1,\infty)$ and $p\in (1,\frac{dr}{2r-1})$. 
	Let the operator $\mathcal{C}$ be defined by
	\begin{equation*}
		\mathcal{C}f(t,\cdot)\doteq \int_0^t e^{(t-s)\Delta}f(s,\cdot)\dd s\text{,}
	\end{equation*}
	Then, $\mathcal{C}$ is a bounded operator from 
	$L^r(0,T;L^{p}_x)$ with values to $L^{2r}(0,T;L^{q}_x)$, where $1/q:=1/p-(2r-1)/dr$.
\end{lemma}
\begin{proof}
	For every $t\in \RR_+$, notice that
	\begin{equation*}
		\big\|\int_0^t e^{(t-s)\Delta}f(s)\dd s\,\big\|_{L^q_x}
		\leq \int_0^t \|\,K(t-s,\cdot)*f(s,\cdot)\,\|_{L^q_x}\dd s
		\leq \int_0^t \|\,K(t-s)\|_{L^{\tilde{q}}_x}\|\,f(s)\,\|_{L^{p}_x}\dd s,
	\end{equation*}
	with $1/\tilde{q}+1/p=1/q+1$, that is $\tilde{q}'=dr/(2r-1)$. Recalling Remark \ref{remark2.1}, we get
	\begin{equation*}
	\big\|\int_0^t e^{(t-s)\Delta}f(s)\dd s\,\big\|_{L^q_x}
	\leq 
	\int_0^t \frac{\quad\|\,f(s)\,\|_{L^{p}_x}}{\quad|t-s|^{\frac{2r-1}{2r}}}\dd s
	\leq
	\int_{\RR} \frac{\quad\|\,f(s)\,\|_{L^{p}_x}}
	{\quad|t-s|^{\frac{2r-1}{2r}}}1_{(0,T)}(s)\dd s.
	\end{equation*}
	Since by Theorem \ref{HLS_Theorem}
	\begin{equation*}
		|\cdot|^{-\frac{2r-1}{2r}}*\|f(\cdot)1_{(0,T)}(\cdot)\|_{L^p_x}\in L^{2r}_t,	
	\end{equation*}
	then there exists $\tilde{C}>0$ such that
	\begin{equation*}
		\|\, \mathcal{C}f \,\|_{L^{2r}(0,T;L^{q}_x)}
		\leq 
		\big\|\;|\cdot|^{-\frac{2r-1}{2r}}*\|f(\cdot)1_{(0,T)}(\cdot)\|_{L^p_x}\big\|_{L^{2r}_t} 
		\leq
		\tilde{C} \|\,f\,\|_{L^{r}(0,T;L^{p}_x)}.
	\end{equation*}
\end{proof}
\noindent For the definition and the main properties of homogeneous Besov Spaces we refer to \cite{MR2768550}. However let us briefly recall two results which characterize such spaces in relation to the heat kernel. 
\begin{theorem}[Characterization of Homogeneous Besov Spaces]\label{Characterization_of_hom_Besov_spaces}
	Let $s$ be a negative real number and $(p,r)\in [1,\infty]^2$. $u$ belongs to $\dot{B}_{p,r}^s(\RR^d)$ 
	if and only if $e^{t\Delta}u$ belongs to $L^p_x$ for almost every $t\in \RR_+$ and 
	\begin{equation*}
		t^{-\frac{s}{2}}\left\|e^{t\Delta}u\right\|_{L^p_x}\in L^r\Big(\RR_+;\frac{\dd t}{t}\Big).
	\end{equation*}
	Moreover, there exists a positive constant $C$ such that
	\begin{equation*}
		\frac{1}{C}\left\| u\right\|_{\dot{B}_{p,r}^s(\RR^d)} \leq
		\left\|\left\|t^{-\frac{s}{2}}e^{t\Delta}u\right\|_{L^p_x}
		\right\|_{L^r(\RR_+;\frac{\dd t}{t})}\leq
		C \left\| u\right\|_{\dot{B}_{p,r}^s(\RR^d)}\text{.}
	\end{equation*}
\end{theorem}
\noindent
Then, imposing the index $s$ equal to $-\frac{2}{r}$, the following Corollary is satisfied: 
\begin{cor}\label{Cor_Characterization_of_hom_Besov_spaces}
	Let $p\in [1,\infty]$ and $r\in [1,\infty)$. $u$ belongs to $\dot{B}_{p,r}^{-\frac{2}{r}}(\RR^d)$ 
	if and only if $e^{t\Delta}u\in L^r_t L^p_x$.
	Moreover, there exists a positive constant $C$ such that
	\begin{equation*}
		\frac{1}{C}\left\| u\right\|_{\dot{B}_{p,r}^{-\frac{1}{2r}}(\RR^d)} \leq
		\left\|e^{t\Delta}u
		\right\|_{L^r_tL^p_x} \leq
		C \left\| u\right\|_{\dot{B}_{p,r}^{-\frac{1}{2r}}(\RR^d)}\text{.}
	\end{equation*}
\end{cor}
\noindent At last, let us state the following Theorem concerning embedding features of Besov spaces, which proof is in 
\cite{MR2768550} Proposition $2.20$.
\begin{theorem}\label{Theorem_embedding_Besov}
	Let $1\leq p_1\leq p_2 \leq \infty$ and $1\leq r_1\leq r_2\leq \infty$. Then for any real 
	number $s$, the space $\dot{B}_{p_1,r_1}^{s}(\RR^d)$ is continuously embedded in 
	$\dot{B}_{p_2,r_2}^{s-d\big(\frac{1}{p_1}-\frac{1}{p_2} \big) }(\RR^d)$. 
\end{theorem}

\section{Existence of solutions for smooth initial dates}
\noindent In this section, by Proposition \ref{Proposition_solutions_smooth_data} and Theorem \ref{Theorem_solutions_smooth_dates}, we prove the existence of weak solutions for system \eqref{Navier_Stokes_system}, assuming more regularity for the initial data. The proofs proceed in the same line of \cite{2013arXiv1304.3235D} and \cite{MR2484939}, however the novelty is to consider also an extra-term $-\ee \Delta$, with $\ee>0$, in the transport equation. 
This perturbation is motivated by the necessity to control the norm of the gradient of the approximate temperatures, even without a space-Lipschitz condition on the approximate velocity field. Obviously this control depends on $\ee$. Hence we consider the following approximation of \eqref{Navier_Stokes_system}.
\begin{equation}\label{Navier_Stokes_system_eps}
\begin{cases}
	\;\partial_t\theta + \Div\, (\theta u) - \ee \Delta u=0					& \RR_+ \times\RR^d,\\
	\;\partial_t u + u\cdot \nabla u -\Div\, (\nu(\theta)\MM) +\nabla\Pi=0	& \RR_+ \times\RR^d,\\
	\;\Div\, u = 0																& \RR_+ \times\RR^d,\\
	\;(u,\,\theta)_{|t=0} = (\bar{u},\,\bar{\theta})							& \;\;\quad \quad\RR^d,\\
\end{cases}
\end{equation}
\begin{remark}\label{rmk_reformulation_momentum_eq}
Since $\Div\,u=0$, we observe that the momentum equation of system \eqref{Navier_Stokes_system_eps} can be reformulated as follows
\begin{equation*}
	\begin{cases}
		\partial_t u^h -\Delta u^h +\nabla^h	\Pi   	= - u^d\,\partial_d u^h - u^h\cdot\nabla u^h + \Div \big\{(\nu (\theta)-1)\MM^h	\big\}	& \RR_+ \times\RR^d,\\
		\partial_t u^d -\Delta u^d +\partial_d	\Pi	= - \nabla^h u^d \cdot u^h  + u^d\Div^h u^h      + \Div	\big\{(\nu (\theta)-1)\MM^d	\big\}	& \RR_+ \times\RR^d,\\
		\end{cases}
\end{equation*}
where $\MM^h:= \nabla u^h + \tr \nabla^h u\quad\text{and}\quad \MM^d:= \partial^d u + \nabla u^d$.  
\end{remark}

\noindent
Firstly, let us prove the existence of weak solutions for system \eqref{Navier_Stokes_system_eps}.
\begin{prop}\label{Proposition_solutions_smooth_data} 
	Let $1<r<\infty$ and $p\in (1,dr/(2r-1))$. Suppose that $\bar{\theta}$ belongs to $L^\infty_x $ and $\bar{u}$ belongs to 
	$\BB_{p,r}^{d/p-1}\cap \BB_{p,r}^{d/p-1+\ee} $ with $\ee<\min\{1/(2r), 1-1/r, 2(d/p -2 + 1/r)\}$. 
	If \eqref{smallness_condition} holds, then there exists a global weak solution $(\theta, u, \Pi)$  
	of \eqref{Navier_Stokes_system_eps} such that
	\begin{equation*}
		u			\in	L^{2r}_tL^{\frac{dr}{r-1}}_x, \quad\quad
		\nabla u	\in	L^{2r}_tL^{\frac{dr}{2r-1}}_x	
					\cap L^{r}_tL^\frac{dr}{2(r-1)}_x,
					\quad\text{and}\quad 
		\Pi\in L^{r}_t L^{\frac{dr}{2(r-1)}}_x.										
	\end{equation*}
	Furthermore, the following inequalities are satisfied:	
	\begin{equation}\label{inequalities_statement_prop_smooth_dates}
	\begin{aligned}
		& 	\| \nabla u^h \|_{L^{r}_t L^{ \frac{dr}{2(r-1)}}_x}+
			\| \nabla u^h \|_{L^{2r}_t L^{ \frac{dr}{2r-1}}_x}+ 
			\| u^h\|_{L^{2r}_t L^{\frac{dr}{r-1}}_x} 
					\leq C_1\eta\text{,}\\
		&	\| \nabla u^d \|_{L^{r}_t L^{ \frac{dr}{2(r-1)}}_x}+
			\|\nabla u^d\|_{L^{2r}_t L^{\frac{dr}{2r-1}}_x}  + 
			\| u^d \|_{L^{2r}_t L^{\frac{dr}{r-1}}_x} 
					\leq C_2\| \bar{u}^d \|_{\dot{B}_{p,r}^{\frac{d}{p}-1}} + C_3\\
		&\|\Pi\|_{L^{r}_t L^{\frac{dr}{2(r-1)}}_x} 
					\leq 
					C_4\eta,\quad
					\|\theta\|_{L^\infty_{t,x}}\leq \|\bar{\theta}\|_{L^{\infty}_x}.
	\end{aligned}
	\end{equation}
	for some suitable positive constants $C_1$, $C_2$, $C_3$ and $C_4$ which are independent by $n$ and $\ee$.
\end{prop}
\begin{proof}
	First, recalling remark \ref{rmk_reformulation_momentum_eq}, we approximate system \eqref{Navier_Stokes_system_eps} by a sequence of linear systems:
	we impose $(\theta_0,u_0,\Pi_0)=(0,0,0)$ and we consider
	\begin{equation}\label{Transport_equation_navier_stokes_approximate}
		\begin{cases}
			\partial_t\theta_{n+1} -\ee\Delta \theta_{n+1}+ \Div (\theta_{n+1} u_n)=0			& \RR_+ \times\RR^d,\\
			\theta_{n|t=0} = \bar{\theta}														& \quad\,\;\, \quad\RR^d,\\
		\end{cases}
	\end{equation}
	\begin{equation}\label{Navier_Stokes_system_approximate}
		\begin{cases}
			\partial_t u_{n+1} -\Delta u_{n+1} +\nabla\Pi_{n+1}=g_{n+1}+\Div \big\{ (\nu (\theta_{n+1})-1)\MM_n \big\} 		& \RR_+ \times\RR^d,\\
			\Div\, u_{n+1} = 0																								& \RR_+ \times\RR^d,\\
			u_{n+1|t=0} = \bar{u}																							&\;\;\quad \quad \RR^d,\\
		\end{cases}
	\end{equation}
	for all $n\in \NN$, where $g_{n+1}$ is a $d$-dimensional vector field, defined by 
	\begin{equation}\label{def_gn}
		g_{n+1}:=-
		\left(
		\begin{matrix}
			u_{n}^d\,\partial_d u^h_{n+1} +u_n^h\cdot\nabla u_n^h\\
			\nabla^h u_n^d \cdot u_{n+1}^h - u_n^d\Div^h u_{n+1}^h
		\end{matrix}
		\right) =:
		\left(
		\begin{matrix}
			g_{n+1}^h\\
			g_{n+1}^d
		\end{matrix}
		\right).
	\end{equation}
	Moreover we denote by $\MM_n^h:= \nabla u_n^h + \tr \nabla^h u_n$ and by $\MM^d:= \partial_d u_n + \nabla u^d_n$. For all $n\in \NN$, the global existence of a 
	weak solution $(\theta_{n+1}, u_{n+1}, \Pi_{n+1})$ of \eqref{Transport_equation_navier_stokes_approximate} and \eqref{Navier_Stokes_system_approximate}  is proved by induction, 
	using Theorem \ref{Theorem_stokes_system_with_linear_perturbation}. Thanks to such results, we have that
	$u_{n+1}$ belongs to $L^{2r}_t L^{dr/(r-1)}_x$, $\nabla u_{n+1}$ belongs to $L^{2r}_t L^{dr/(2r-1)}_x\cap L^{r}_tL^{dr/(2r-2)}_x$,
	$\theta_{n+1}$ to $L^{\infty}_{t,x}$ and $\Pi_{n+1}$ to $L^{r}_t L_x^{dr/(2r-2)}$. 
	
	\vspace{0.2cm}
	\textsc{Step 1: Estimates not dependent on $\ee$.} First, the Maximal Principle for parabolic equation implies, $\| \theta_n\|_{L^\infty_{t,x}}\leq \|\bar{\theta}\|_{L^\infty_x}$, 
	for any positive integer $n$. We want to prove that
	\begin{equation}\label{horizantal_inequality_apprx_system}
	\begin{alignedat}{2}
		& 	\| \nabla u_n^h \|_{L^{r}_t L^{ \frac{dr}{2(r-1)}}_x}+
			\| \nabla u_n^h \|_{L^{2r}_t L^{ \frac{dr}{2r-1}}_x}+ 
			\| u_n^h \|_{L^{2r}_t L^{\frac{dr}{r-1}}_x} 
		&&\leq C_1\eta \text{,}\\
		&	\|\nabla u_n^d\|_{L^{r}_t L^{\frac{dr}{2(r-1)}}_x}+
			\|\nabla u_n^d\|_{L^{2r}_t L^{\frac{dr}{2r-1}}_x}  + 
			\| u_n^d \|_{L^{2r}_t L^{\frac{dr}{r-1}}_x} 
		&&\leq C_2\| \bar{u}^d \|_{\dot{B}_{p,r}^{\frac{d}{p}-1}} + C_3,
	\end{alignedat}
	\end{equation} 
	for any $n\in\NN$ and for some suitable positive constants $C_1$, $C_2$ and $C_3$. 	First we will show by induction that, if $\eta$ is small enough then 
	\begin{equation}\label{indcution_one}
	\begin{aligned}
		& 	\| \nabla u_n^h \|_{L^{2r}_t L^{ \frac{dr}{2r-1}}_x}+ 
			\| u_n^h \|_{L^{2r}_t L^{\frac{dr}{r-1}}_x} 
					\leq \bar{C}_1\tilde{\eta}\text{,}\\
		&	\|\nabla u_n^d\|_{L^{2r}_t L^{\frac{dr}{2r-1}}_x}  + 
			\| u_n^d \|_{L^{2r}_t L^{\frac{dr}{r-1}}_x} 
					\leq \bar{C}_2\| \bar{u}^d \|_{\dot{B}_{p,r}^{\frac{d}{p}-1}} + \bar{C}_3
	\end{aligned}
	\end{equation}
	for all $n\in \NN$ and for some appropriate positive constant $\bar{C}_1$, $\bar{C}_2$, $\bar{C}_3$, where $\tilde{\eta}\leq \eta$ is defined by
	\begin{equation}\label{def_tildeeta}
		\tilde{\eta}:=\big(
						\|\nu-1\|_\infty + \| \bar{u}^h \|_{\dot{B}_{p,r}^{-1+\frac{d}{p}}}
					\big)
					\exp
					\big\{ 
						\frac{c_r}{2}\|\bar{u}^d\|^{4r}_{\dot{B}_{p,r}^{-1+\frac{d}{p}}} 
					\big\}.
	\end{equation}
	Let $\lambda$ be a positive real number, and let $u_{n+1,\lambda}$, $\nabla u_{n+1, \lambda}$ and $\Pi_{n+1, \lambda}$ be defined by
	\begin{equation}\label{def_ulambda}
		(u_{n+1, \lambda},\,\nabla u_{n+1, \lambda},\,\Pi_{n+1, \lambda} )(t):= h_{n,\lambda }(0,t)(u_{n+1},\,\nabla u_{n+1},\,\Pi_{n+1} )(t),
	\end{equation}
	where, for all $0 \leq s<t<\infty$, 
	\begin{equation}\label{def_h}
		h_{n, \lambda}(s,t):=  
		\exp
		\big\{ 
			-\lambda \int_s^t \|u_n^d(\tau)\|^{2r}_{L_x^{\frac{dr}{r-1}}}\dd \tau\,
			-\lambda \int_s^t \|\nabla u_n^d(\tau)\|^{2r}_{L_x^{\frac{dr}{2r-1}}}\dd \tau\,
		\big\}.
	\end{equation}	
	Writing $u_{n+1}$ by the Mild formulation, we get
	\begin{equation}\label{def_u_n+1}
	\begin{aligned}
		u_{n+1}(t) = 
		\underbrace{
		e^{t \Delta}\bar{u}}_{u_{L}}+
		\underbrace{
		\int_{0}^t e^{(t-s)\Delta}\PP g_{n+1}(s)\dd s}_{F^1_{n+1}(t)}+
		\underbrace{
		\int_{0}^t \nabla e^{(t-s)\Delta} R\cdot R\cdot \{(\nu (\theta_{n+1})-1)\MM_n\}(s)\dd s}_{F^2_{n+1}(t)} + \\ +
		\underbrace{
		\int_{0}^t \Div\, e^{(t-s)\Delta} \{(\nu (\theta_{n+1})-1)\MM_n\}(s)\dd s}_{F^3_{n+1}(t)},
	\end{aligned}
	\end{equation}
	where $R:=\nabla/\sqrt{-\Delta}$ is the Riesz transform ($R\cdot := \Div/\sqrt{-\Delta}$) and $\PP := I+R\,R\cdot$ is the Leray projection operator, which are 
	bounded operators from $L^q_x$ to $L^q_x$ for any $q\in(1,\infty)$. Thus
	\begin{equation}\label{def_un+1lambda}
		u_{n+1,\lambda}(t)= 
		\underbrace{
		h_{n,\lambda}(0,t)u_L(t)													}_{u_{L,\lambda}(t)}+
		\underbrace{
		\int_{0}^th_{n,\lambda}(s,t) e^{(t-s)\Delta}\PP g_{n+1,\lambda}(s)\dd s		}_{F^{1}_{n+1,\lambda}(t)} +
		\underbrace{
		h_{n,\lambda}(0,t)F_2(t)													}_{F^{2}_{n+1, \lambda}(t)} 
		+ 
		\underbrace{		
		h_{n,\lambda}(0,t)F_3(t)													}_{F^{3}_{n+1,\lambda}(t)}, 
	\end{equation}
	where $g_{n+1,\lambda}(t)=g_{n+1}(t)h_{n,\lambda}(0,t)$. First, we want to estimate $\nabla u_{n+1,\lambda}^h$ in $L^{2r}_t L^{ dr/(2r-1)}_x$ and $ u_{n+1,\lambda}^h$ in 
	$L^{2r}_t L^{ dr/(r-1)}_x$. We begin observing that, by Corollary \ref{Cor_Characterization_of_hom_Besov_spaces} and Theorem \ref{Theorem_embedding_Besov},
	\begin{equation}\label{estitmate_1a}
	\begin{aligned}
		\|			u_{L,\lambda}^h	\|_{L^{2r}_t L^{\frac{dr}{r-1	}}_x}	+
		\| \nabla 	u_{L,\lambda}^h \|_{L^{2r}_t L^{\frac{dr}{2r-1	}}_x} 		
		\leq 
		\|			u_L^h				\|_{L^{2r}_t L^{\frac{dr}{r-1}}_x}+
		\| \nabla 	u_L^h				\|_{L^{2r}_t L^{\frac{dr}{2r-1}}_x}
		\leq 
		C\|\,\bar{u}^h\,\|_{\BB_{p,r}^{\frac{d}{p}-1}},
	\end{aligned}
	\end{equation}
	for a suitable positive constant $C$. Furhtermore, by the definition of $g_{n+1}$ and by Lemma \ref{Lemma_A.1}, Lemma \ref{Lemma_A.2}, Lemma \ref{Lemma2} and 
	Lemma \ref{Lemma3}, we obtain
	\begin{equation}\label{estitmate_1b}
	\begin{aligned}
		\| 			F^{1,h}_{n+1,\lambda} \|_{L^{2r}_t L^{\frac{dr}{r-1}}_x } &+
		\|\nabla 	F^{1,h}_{n+1,\lambda} \|_{L^{2r}_t L^{\frac{dr}{2r-1}}_x}\leq 
		C\big\{
			\frac{1}{\,\lambda^{\frac{1}{4r}}} \|\, u_n^d\,\|^{\frac{1}{2}}_{L^{2r}_t L^{\frac{dr}{r-1}}_x}
			\|\,\nabla u_{n+1,\lambda}^h\,\|_{L^{2r}_t L^{\frac{dr}{2r-1}}_x}+ \\&+
			\|\,u_n^h\,\|_{L^{2r}_t L^{\frac{dr}{r-1}}_x}\|\,\nabla u_n^h\,\|_{L^{2r}_t L^{\frac{dr}{2r-1}}_x} +
			\frac{1}{\,\lambda^{\frac{1}{4r}}} 
			\|\,\nabla u_n^d\,\|^{\frac{1}{2}}_{L^{2r}_t L^{\frac{dr}{2r-1}}_x}
			\|\,u_{n+1,\lambda}^h\,\|_{L^{2r}_t L^{\frac{dr}{r-1}}_x}
		\big\}.
	\end{aligned}	
	\end{equation}
	Furthermore, By Corollary \ref{Hardy-Litlewood-Sobolev-Inequality} and Theorem \ref{Maximal_regularity_theorem} we also obtain 
	\begin{equation}\label{estitmate_1c}
	\begin{aligned}
		\| F^{2,h}_{n+1,\lambda} + F^{3,h}_{n+1,\lambda} &\|_{L^{2r}_t L^{\frac{dr}{r-1}}_x} 
		=
		\|
			\int_{0}^t \Delta e^{(t-s)\Delta}\PP R\cdot (\sqrt{-\Delta})^{-1}\{(\nu (\theta_{n+1})-1)
			\MM_{n}\}(s)\dd s\,
		\|_{L^{2r}_t L^{\frac{dr}{r-1}}_x}\\
		&\leq
		C\|\, 
			(\sqrt{-\Delta})^{-1}(\nu (\theta_{n+1})-1)
			\MM_{n} 
		\|_{L^{2r}_t L^{\frac{dr}{r-1}}_x}\leq
		C\|\,\nu-1\,\|_{\infty} \|\,\nabla u_{n}\,\|_{L^{2r}_t L^{\frac{dr}{2r-1}}_x}.
	\end{aligned}
	\end{equation}
	Similarly, recalling Theorem \ref{Maximal_regularity_theorem}, we deduce that
	\begin{equation}\label{estitmate_1d}
	\begin{aligned}
		\| \nabla F^{2,h}_{n+1,\lambda}+ \nabla F^{3,h}_{n+1,\lambda} \|_{L^{2r}_tL^{\frac{dr}{2r-1}}_x}
		&=
		\|\,
			\int_0^t \Delta e^{(t-s)\Delta}R \,\PP\, R\cdot 
			\{(\nu(\theta_{n+1})-1)\MM_{n}\}(s)\dd s\,
		\|_{L^{2r}_tL^{\frac{dr}{2r-1}}_x}\\
		&\leq
		\|\{(\nu(\theta_{n+1})-1)\MM_{n}\}
		\|_{L^{2r}_tL^{\frac{dr}{2r-1}}_x}\leq
		C\|\nu -1\|_\infty\|\,\nabla u_{n}\,\|_{L^{2r}_tL^{\frac{dr}{2r-1}}_x}.
	\end{aligned}
	\end{equation}
	Summarizing \eqref{estitmate_1a}, \eqref{estitmate_1b}, \eqref{estitmate_1c} and \eqref{estitmate_1d}, we deduce that there exists a positive constant $C$ such that, 
	for all $n\in \NN$
	\begin{equation}\label{estimate_1}
	\begin{aligned}
		 \|\nabla u_{n+1,\lambda}^h \|_{L^{2r}_t L^{ \frac{dr}{2r-1}}_x}+
		 \| u_{n+1,\lambda}^h \|_{L^{2r}_t L^{\frac{dr}{r-1}}_x}
		 \leq 
		 C\big\{ \|\,\bar{u}^h\,\|_{\dot{B}_{p,r}^{\frac{d}{p}-1}} +
		 \frac{1}{\,\lambda^{\frac{1}{4r}}}
		\big( 
			\|\, u_n^d\,\|^{\frac{1}{2}}_{L^{2r}_t L^{\frac{dr}{r-1}}_x}
			\|\,\nabla u_{n+1,\lambda}^h\,\|_{L^{2r}_t L^{\frac{dr}{2r-1}}_x} +\\ +
			\|\,\nabla u_n^d\,\|^{\frac{1}{2}}_{L^{2r}_t L^{\frac{dr}{2r-1}}_x}
			\|\,u_{n+1,\lambda}^h\,\|_{L^{2r}_t L^{\frac{dr}{r-1}}_x}
		\big)		
		+\|\,u_n^h\,\|_{L^{2r}_t L^{\frac{dr}{r-1}}_x}\|\,\nabla u_n^h\,\|_{L^{2r}_t L^{\frac{dr}{2r-1}}_x}
		+\|\nu -1\|_\infty\|\,\nabla u_{n}\,\|_{L^{2r}_tL^{\frac{dr}{2r-1}}_x}\big\}.
	\end{aligned}
	\end{equation}
	Recalling the induction hypotheses \eqref{indcution_one}, we fix a positive $\lambda$ such that
	\begin{equation}\label{inequality_lambda1}
		C \frac{1}{\,\lambda^{\frac{1}{4r}}} 
		\Big(\bar{C}_2\| \bar{u}^d \|_{\dot{B}_{p,r}^{-1+\frac{d}{p}}} + \bar{C}_3\Big)^{\frac{1}{2}}=
		\frac{1}{4}\quad 
		(
		 \text{ namely } \lambda =  (4C)^{4r}(\bar{C}_2\| \bar{u}^d \|_{\dot{B}_{p,r}^{\frac{d}{p}-1}} + 
		\bar{C}_3)^{2r}\,),
	\end{equation}
	so that we can absorb all the terms on the right-hands side with index $n+1$ by the left-hand side, obtaining
	\begin{equation}\label{first_estimate}
	\begin{aligned}
		 \| \nabla u_{n+1,\lambda}^h \|_{L^{2r}_t L^{ \frac{dr}{2r-1}}_x}&+
		 \| u_{n+1,\lambda}^h \|_{L^{2r}_t L^{\frac{dr}{r-1}}_x}
		 \leq \\
		 &\leq
		 2C
		 (
			 \| \bar{u}^h \|_{\dot{B}_{p,r}^{\frac{d}{p}-1}}+ 
			 \bar{C}_1^2\tilde{\eta}^2 + 
			\|\nu-1\|_{L^\infty_x}( \bar{C}_1\tilde{\eta} +\bar{C}_2\| \bar{u}^d \|_{\dot{B}_{p,r}^{\frac{d}{p}-1}} + \bar{C}_3 )
		),
	\end{aligned}
	\end{equation}
	thanks to the induction hypotheses \eqref{indcution_one}. 
	Now we reformulate \eqref{first_estimate} without the index $\lambda$ on the left-hand side:
	\begin{align*}
		 \| \nabla u_{n+1}^h& \|_{L^{2r}_t L^{ \frac{dr}{2r-1}}_x} +
		 \| u_{n+1}^h \|_{L^{2r}_t L^{\frac{dr}{r-1}}_x}
		 \leq
		 \sup_{t\in\RR_+}h_{n,\lambda}(0,t)^{-1}
		\big(\,
		 	 \| \nabla u_{n+1,\lambda}^h \|_{L^{2r}_t L^{ \frac{dr}{2r-1}}_x}+
			 \| u_{n+1,\lambda}^h \|_{L^{2r}_t L^{\frac{dr}{r-1}}_x}\,
		\big)\\
		&\leq 
		\exp
		\big\{ \lambda 
			(\,
				\bar{C}_2\| \bar{u}^d \|_{\dot{B}_{p,r}^{-1+\frac{d}{p}}}+\bar{C}_3
			)^{2r} 
		\big\}
		\big(\,
		 	 \| \nabla u_{n+1,\lambda}^h \|_{L^{2r}_t L^{ \frac{dr}{2r-1}}_x}+
			 \| u_{n+1,\lambda}^h \|_{L^{2r}_t L^{\frac{dr}{r-1}}_x}
		\big),		
	\end{align*}
	thanks to the second inequality of \eqref{indcution_one}. Hence, recalling \eqref{inequality_lambda1} and \eqref{first_estimate}, we obtain the following inequality
	\begin{align*}
		\| \nabla u_{n+1}^h \|_{L^{2r}_t L^{ \frac{dr}{2r-1}}_x} +
		\| u_{n+1}^h \|_{L^{2r}_t L^{\frac{dr}{r-1}}_x}
		\leq
		&\exp
		\big\{  
		  2^{4r-1}(4C)^{4r}
			(\,
				\bar{C}_2^{4r}\| \bar{u}^d \|_{\dot{B}_{p,r}^{\frac{d}{p}-1}}^{4r}+\bar{C}_3^{4r}
			)
		\big\}{\scriptstyle \times} \\ {\scriptstyle \times}
		 &
		 2C
		 (
			 \| \bar{u}^h \|_{\dot{B}_{p,r}^{\frac{d}{p}-1}}+ 
			 \bar{C}_1^2\tilde{\eta}^2 + 
			\|\nu-1\|_{L^\infty_x}( \bar{C}_1\tilde{\eta} +\bar{C}_2\| \bar{u}^d \|_{\dot{B}_{p,r}^{\frac{d}{p}-1}} + \bar{C}_3 )
		).
	\end{align*}
	Assuming that $c_r$ of \eqref{smallness_condition} fulfills $c_r\geq 1$ and $c_r/4\geq 2^{4r-1}(4C)^{4r} \bar{C}_2^{4r}$, we get that the right-hand side of the previous 
	inequality is bounded by
	\begin{align*}
		2C \exp
		\big\{  
			2^{4r-1}(4C)^{4r}
			C^{4r}_3 +
			\frac{c_r}{4}\| \bar{u}^d \|^{4r} _{\dot{B}_{p,r}^{\frac{d}{p}-1}}
		\big\}
		(
			 \| \bar{u}^h \|_{\dot{B}_{p,r}^{\frac{d}{p}-1}}+ 
			 \bar{C}_1^2\tilde{\eta}^2 + 
			\|\nu-1\|_{L^\infty_x}( \bar{C}_1\tilde{\eta} +\bar{C}_2\| \bar{u}^d \|_{\dot{B}_{p,r}^{\frac{d}{p}-1}} + \bar{C}_3 )
		)\\
		\leq 
		2C \exp
		\left\{  
			2^{4r-1}(4C)^{4r}
			C^{4r}_3
		\right\}
		( 1+(\bar{C}_1^2+\bar{C}_1)\tilde{\eta} + \bar{C}_2 + \bar{C}_3 )\tilde{\eta},
	\end{align*}
	where we have used 
	$ 	\|\nu-1\|_{\infty}\|\bar{u}^d\|_{\dot{B}_{p,r}^{d/p-1}}
		\leq 
		\|\nu-1\|_{\infty}
		\exp
		\{
			\|\bar{u}^d\|_{\dot{B}_{p,r}^{d/p-1}}^{4r}/(4r)
		\}
	$.
	Imposing $\bar{C_1}$ big enough and $\eta$ small enough in order to have
	\begin{equation*}
		\exp
		\big\{  
			2^{4r-1}(4C)^{4r}
			\bar{C}^{4r} _3
		\big\}2C(1+\bar{C}_2+\bar{C}_3) < \frac{\bar{C}_1}{2}
		\quad\text{and}\quad
		\exp
		\big\{  
			2^{4r-1}(4C)^{4r}
			\bar{C}^{4r} _3
		\big\}
		(\bar{C}_1+1)\tilde{\eta}\leq \frac{1}{2},
	\end{equation*}
	we finally obtain that the first equation of \eqref{indcution_one} is true for any $n\in\NN$. Now we deal with the second equation of \eqref{indcution_one} and we still proceed 
	by induction. Recalling \eqref{def_u_n+1}and proceeding in a similarly way as done in the previous estimates, the following inequality is satisfied:
	\begin{align*}
		\|\nabla u_{n+1}^d &\|_{L^{2r}_t L^{\frac{dr}{2r-1}}_x}  +
		\| u_{n+1}^d \|_{L^{2r}_t L^{\frac{dr}{r-1}}_x}\leq
		 C\big\{ \|\bar{u}^d\|_{\dot{B}_{p,r}^{-1+\frac{1}{p}}} +
		\|g_{n+1}\|_{L^{r}_tL^{\frac{dr}{3r-2}}_x}+
		\|\nu-1\|_{\infty}\|\nabla u_n \|_{L^{2r}_tL^{\frac{dr}{2r-1}}_x}
		\big\},
	\end{align*}
	for a suitable positive constant $C$. Hence, by the definition \eqref{def_gn} of $g_{n+1}$, we deduce that
	\begin{align*}
		&\|\nabla u_{n+1}^d\|_{L^{2r}_t L^{\frac{dr}{2r-1}}_x}+\| u_{n+1}^d \|_{L^{2r}_t L^{\frac{dr}{r-1}}_x}\leq
		C\big\{ 
				\|\bar{u}^d\|_{\BB_{p,r}^{\frac{d}{p}-1}}+
				\|u_n^h\|_{L^{2r}_t L^{\frac{dr}{r-1}}_x}\|\nabla u_n^h\|_{L^{2r}_t L^{\frac{dr}{2r-1}}_x}+
				\|u_{n+1}^h\|_{L^{2r}_t L^{\frac{dr}{r-1}}_x}\\&{\scriptstyle \times}
				\|\nabla u_{n}^d\|_{L^{2r}_t L^{\frac{dr}{2r-1}}_x}
				+\|u_{n}^d\|_{L^{2r}_t L^{\frac{dr}{r-1}}_x}\|\nabla u_{n+1}^h\|_{L^{2r}_t L^{\frac{dr}{2r-1}}_x}+
				\|\nu-1\|_{\infty}
				\big(
					\|\nabla u_n^h \|_{L^{2r}_tL^{\frac{dr}{2r-1}}_x} + 
					\|\nabla u_n^d \|_{L^{2r}_tL^{\frac{dr}{2r-1}}_x}
				\big)
			\big\},		
	\end{align*}
	so that, thanks to the induction hypotheses and the previous estimates, we bound the right hand-side by
	\begin{equation*}
		(\, 
			C + \bar{C}_1\bar{C}_2\tilde{\eta} + \|\nu-1\|_{\infty}\bar{C}_2\,
		)\|\bar{u}^d\|_{\dot{B}_{p,r}^{\frac{d}{p}-1}} +
		( 
			\bar{C}_1\bar{C}_3 + \bar{C}_1^2\tilde{\eta} + 
			\|\nu-1\|_{\infty} 
			(
				\bar{C}_1 +\bar{C}_2 
			) 
		)\tilde{\eta}.
	\end{equation*}
	Finally, imposing $C<\bar{C}_2$ and $\eta$ small enough in order to fulfill $C + (\,\bar{C}_1\bar{C}_2\,+\,\bar{C}_2\,) \eta \leq \bar{C}_2$ and moreover 
	$( \bar{C}_1\bar{C}_3 + \bar{C}_1^2\eta + \eta( \bar{C}_1 +\bar{C}_2 \big) )\eta \leq \bar{C}_3$, then the second inequality of \eqref{indcution_one} is 
	satisfied for any $n\in\NN$. Now, let us prove by induction that there exist three positive constants $\tilde{C}_1$, $\tilde{C}_2$ and $\tilde{C}_3$, such that 
	\begin{equation}\label{Finalestimate2r-2}
		\|\nabla u_{n+1}^h\|_{L^{r}_t L^{\frac{dr}{2(r-1)}}_x} \leq \tilde{C}_1 \eta\quad
		\text{and}\quad
		\|\nabla u_{n+1}^d\|_{L^{r}_t L^{\frac{dr}{2(r-1)}}_x} \leq 
		\tilde{C}_2\|\bar{u}^d\|_{\dot{B}_{p,r}^{\frac{d}{p}-1}}+\tilde{C}_3, 
	\end{equation}
	for any positive integer $n$. Recalling the mild formulation \eqref{def_u_n+1} of $u_{n+1}$, Lemma \eqref{Cor_Characterization_of_hom_Besov_spaces},
	Corollary \eqref{Cor_Characterization_of_hom_Besov_spaces} and Theorem \eqref{Theorem_embedding_Besov}, it turns out that
	\begin{equation}\label{uLh2r-1}
		\|\nabla u_L^h 			\|_{L^{r}_t L^{\frac{dr}{2(r-1)}}_x} + 
		\|\nabla F^{1,h}_{n+1}	\|_{L^{r}_t L^{\frac{dr}{2(r-1)}}_x} \leq 
		C\big(
			\|\bar{u}^h\|_{\dot{B}_{p,r}^{\frac{d}{p}-1}}+
			\|g_{n+1}\|_{L^{r}_t L^{\frac{dr}{3r-2}}_x}
		\big),
	\end{equation}
	while Theorem \ref{Maximal_regularity_theorem} implies
	\begin{equation*}\label{F22r-1}
		\| \nabla F^{2,h}_{n+1} + \nabla F^{3,h}_{n+1} \|_{L^{r}_t L^{\frac{dr}{2(r-1)}}_x} \leq \|\nu-1\|_{\infty}\|\nabla u_n\|_{L^r_tL^{\frac{dr}{2(r-1)}}_x}.
	\end{equation*}
	By the definition of $g_{n+1}$ \eqref{def_gn}, its $L^r_t L^{dr/(3r-2)}_x$-norm is bounded by
	\begin{equation*}
		\|u_n^d \|_{L^{2r}_tL^{\frac{dr}{r-1}}_x} \|\nabla u_{n+1}^h \|_{L^{2r}_tL^{\frac{dr}{2r-1}}_x}+
		\|u_n^h \|_{L^{2r}_tL^{\frac{dr}{r-1}}_x} \|\nabla u_{n}^h \|_{L^{2r}_tL^{\frac{dr}{2r-1}}_x}+
		\|u_{n+1}^h \|_{L^{2r}_tL^{\frac{dr}{r-1}}_x} \|\nabla u_{n}^d \|_{L^{2r}_tL^{\frac{dr}{2r-1}}_x}
	\end{equation*}
	Hence, thanks to the uniform estimates given by \eqref{indcution_one}, we obtain
	\begin{equation}\label{estimate_g}
		\|g_{n+1}\|_{L^{r}_t L^{\frac{dr}{3r-2}}_x}
		\leq 
		\big( \bar{C}_1\tilde{\eta}+\bar{C}_3 \big)\bar{C}_1\tilde{\eta}+ \bar{C}_1 \bar{C}_2 
		\|\bar{u}^d\|_{\dot{B}_{p,r}^{-1+\frac{d}{p}}}\tilde{\eta}\leq 
		\big( 
			\bar{C}_1\tilde{\eta}+\bar{C}_3+\bar{C}_2
		\big)\bar{C}_1 \eta,
	\end{equation}
	Furthermore, by the induction hypotheses \eqref{Finalestimate2r-2}, we remark that
	\begin{equation}\label{recall_2}
		\|\nu-1\|_{\infty}\|\nabla u_n\|_{L^r_t L^{\frac{dr}{2(r-1)}}_x} \leq 
		\|\nu-1\|_{\infty}\tilde{C}_1 \eta
		+\tilde{C}_2\tilde{\eta}+
		\|\nu-1\|_{\infty}\tilde{C}_3.
	\end{equation}
	Those, summarizing \eqref{uLh2r-1}, \eqref{F22r-1}, \eqref{estimate_g} and \eqref{recall_2}, we finally obtain
	\begin{equation}\label{induction_2r-1a}
	\|\nabla u_{n+1}^h\|_{L^{r}_t L^{\frac{dr}{2(r-1)}}_x} \leq 
		C\big\{
			\|\bar{u}^h\|_{\dot{B}_{p,r}^{-1+\frac{d}{p}}}+
			\big( 
				\bar{C}_1\tilde{\eta}+\bar{C}_3+ \bar{C}_2
			\big)\bar{C}_1 \eta+
			\|\nu-1\|_{\infty}\tilde{C}_1 \eta+
			\tilde{C}_2\tilde{\eta}+
			\|\nu-1\|_{\infty}\tilde{C}_3
		\big\},
	\end{equation}
	hence, imposing $\tilde{C}_1> C\big(1 + \bar{C}_1\bar{C}_3+\bar{C}_1 \bar{C}_2 +\tilde{C}_2+\tilde{C}_3\big)$ 
	and assuming $\eta$ small enough, we get that the first inequality of \eqref{Finalestimate2r-2} is true for any positive integer $n$.
	Now, proceeding as to prove \eqref{induction_2r-1a}, we get
	\begin{equation*}
		\|\nabla u_{n+1}^d\|_{L^{r}_t L^{\frac{dr}{2(r-1)}}_x} \leq
		C\big\{
			\|\bar{u}^d\|_{\dot{B}_{p,r}^{-1+\frac{d}{p}}}
			+
			\big( 
			\bar{C}_1\tilde{\eta}+\bar{C}_3+\bar{C}_2
			\big)\bar{C}_1 \eta
			+
			\|\nu-1\|_{\infty}\tilde{C}_1 \eta
			+\tilde{C}_2\tilde{\eta}+
			\|\nu-1\|_{\infty}\tilde{C}_3.
		\big\}
	\end{equation*}
	Hence, imposing $\tilde{C}_2> C$, $\tilde{C}_3>0$ such that 
	$C\{( \bar{C}_1\tilde{\eta}+\bar{C}_3+\bar{C}_2)\bar{C}_1 \eta+\|\nu-1\|_{\infty}\tilde{C}_1 \eta +\tilde{C}_2\tilde{\eta}\}< \tilde{C}_3$ and assuming 
	$\eta$ small enough, we finally establish that also the second inequality of \eqref{Finalestimate2r-2} is true for any $n\in\NN$. To conclude this first step,
	denoting $C_1:=\bar{C}_1+\tilde{C}_1$, $ C_2:= \bar{C}_2 + \tilde{C}_2$, $C_3:=\bar{C}_3 + \tilde{C}_3$ and summarizing 
	\eqref{indcution_one} and \eqref{Finalestimate2r-2}, we finally obtain \eqref{horizantal_inequality_apprx_system}. To conclude this first step we observe that
	$\Pi_{n+1}$ is determined by
	\begin{equation}\label{def_Pi_n}
			\Pi_{n+1} := 
			-\left(-\Delta\right)^{-\frac{1}{2}}R\cdot g_{n+1}
			-R\cdot R\cdot \{(\nu(\theta_{n+1})-1)\nabla u_n \}, 
	\end{equation} 
	so that, thanks to Corollary \ref{Hardy-Litlewood-Sobolev-Inequality} and \eqref{estimate_g}, we deduce that 
	\begin{equation}\label{estimate_Pin+1}
		\| \Pi_{n+1} \|_{L^r_t L^{\frac{Nr}{2(r-1)}}} 
		\leq 
		C(
		\| g_{n+1} 		\|_{L^r_t L^{\frac{Nr}{3r-2)}}_x}+
		\|\nu -1\|_{L^\infty_x}\| \nabla u_n 	\|_{L^r_t L^{\frac{Nr}{2(r-1)}}_x} )
		\leq C_4 \eta,
	\end{equation}
	for any $n\in\NN$ and for a suitable positive constant $C_4$.

	\vspace{0.2cm}
	\textsc{Step 2: $\ee$-Dependent Estimates.} As second step, we are going to establish some $\ee$-dependent estimates which are useful for the third step, where we will
	prove that $(\theta^n,\, u^n,\, \Pi^n)$ is a Cauchy sequence in a suitable space. Defining $\bar{r}:= 2r/(2-\ee r)>r$,  then
	we still have  $ p<d\bar{r}/(2\bar{r}-1)=2dr/((4+\ee)r-2)$, since $\ee$ is bounded by $2(d/p -2+ 1/r)$. Since 
	$\BB_{p,r}^{d/p-1}\hookrightarrow \BB_{p,\bar{r}}^{d/p-1}$, then there exists a positive constant $C$ such that
	\begin{equation*}
		\bar{\eta}:= \big( 
				\| \nu - 1 \|_\infty + \|\bar{u}^h\|_{\dot{B}_{p,\bar{r}}^{\frac{d}{p}-1}}
			\big)
			\exp\Big\{ c_r \|\bar{u}^d\|_{\dot{B}_{p,\bar{r}}^{\frac{d}{p}-1}}^{4r} \Big\}
			\leq C\eta.
	\end{equation*}
	Hence, arguing exactly as to prove \eqref{horizantal_inequality_apprx_system} with $\bar{r}$ instead of $r$, we get also
	\begin{equation}\label{inequality_apprx2}
	\begin{alignedat}{2}
		& 	\| \nabla u_n^h \|_{L^{2\bar{r}}_t L^{ \frac{d\bar{r}}{2\bar{r}-1}}_x}+ 
			\| u_n^h \|_{L^{2\bar{r}}_t L^{\frac{d\bar{r}}{\bar{r}-1}}_x} 
		&&\leq C_1\bar{\eta} \text{,}\\
		&	\|\nabla u_n^d\|_{L^{2\bar{r}}_t L^{\frac{dr}{2\bar{r}-1}}_x}  + 
			\| u_n^d \|_{L^{2\bar{r}}_t L^{\frac{d\bar{r}}{\bar{r}-1}}_x} 
		&&\leq C_2\| \bar{u}^d \|_{\dot{B}_{p,\bar{r}}^{\frac{d}{p}-1}} + C_3.
	\end{alignedat}
	\end{equation}
	
	First, we want to demonstrate by induction that there exists a positive constant $\bar{C}_5$ such that
	\begin{equation}\label{induction_ee}
		\| 			u_n \|_{ L^{	2\bar{r}					}_t L^{ \frac{d\bar{r}}{\bar{r}(1-\ee)-1}	}_x} +
		\| \nabla 	u_n \|_{ L^{	2\bar{r}					}_t L^{ \frac{d\bar{r}}{(2-\ee)\bar{r}-1}	}_x}
		\leq 
		\bar{C}_5 \| \bar{u} \|_{ \BB_{p,\bar{r}}^{\frac{d}{p}-1+\ee }}
	\end{equation}
	Let us remark that such spaces are well defined, since $\bar{r}(1-\ee)-1>0$ (from $\ee<1-1/r<1-1/\bar{r}$). 
	Recalling the mild formulation of $u_{n+1}$ \eqref{def_u_n+1}, Corollary \ref{Cor_Characterization_of_hom_Besov_spaces}  and Theorem \ref{Theorem_embedding_Besov} 
	yield
	\begin{equation*}
		\| 			u_L \|_{ 	L^{2\bar{r}						}_t L^{ \frac{d\bar{r}}{\bar{r}(1-\ee)-1}	}_x} + 
		\| \nabla 	u_L \|_{ L^{	2\bar{r}					}_t L^{ \frac{d\bar{r}}{(2-\ee)\bar{r}-1}	}_x} 
		\leq 
		C \| \bar{u} \|_{ \BB_{p,\bar{r}}^{\frac{d}{p}-1+\ee } },
	\end{equation*}
	for a suitable positive constant $C$. Moreover, thanks to Lemma \ref{Lemma2} and Lemma \ref{Lemma3}, we get 
	\begin{equation*}
		\|		 	F_{n+1}^1 		\|_{ L^{	2\bar{r}					}_t L^{ \frac{d\bar{r}}{\bar{r}(1-\ee)-1}	}_x}  	+
		\| \nabla 	F_{n+1}^1		\|_{ L^{	2\bar{r}					}_t L^{ \frac{d\bar{r}}{(2-\ee)\bar{r}-1}	}_x} 
		\leq 
		C
		\| g_{n+1} \|_{L^{\bar{r}}_t L^{\frac{d\bar{r}}{(3-\ee)\bar{r}-2}}_x}.
	\end{equation*}
	From the definition of $g_{n+1}$ \eqref{def_gn} and the estimates \eqref{inequality_apprx2}, we get 
	\begin{equation*}
	\begin{aligned}
		\| g_{n+1} \|_{L^{\bar{r}}_t L^{\frac{d\bar{r}}{(3-\ee)\bar{r}-2}}_x} 
		\leq 
		C_1\bar{\eta} 
		\big(\, 
			\| u_n^d 				 \|_{ 	L^{2\bar{r}						}_t L^{ \frac{d\bar{r}}{\bar{r}(1-\ee)-1}	}_x} & + 
			\| u_{n+1}^d 			 \|_{ 	L^{2\bar{r}						}_t L^{ \frac{d\bar{r}}{\bar{r}(1-\ee)-1}	}_x}  + \\&+
			\| \nabla u_n^d		 	 \|_{ L^{	2\bar{r}					}_t L^{ \frac{d\bar{r}}{(2-\ee)\bar{r}-1}	}_x}  +
			\| \nabla u_{n+1}^d  	 \|_{ L^{	2\bar{r}					}_t L^{ \frac{d\bar{r}}{(2-\ee)\bar{r}-1}	}_x}
		\big),
	\end{aligned}
	\end{equation*}
	so that, by the induction hypotheses \eqref{induction_ee}, we have the following bound
	\begin{equation*}
		\| g_{n+1} \|_{L^{\bar{r}}_t L^{\frac{d\bar{r}}{(3-\ee)\bar{r}-2}}_x}
		\leq 
		C_1\bar{\eta} 
		\big(\, 
			\| u_{n+1} 			\|_{ 	L^{2\bar{r}						}_t L^{ \frac{d\bar{r}}{\bar{r}(1-\ee)-1}	}_x}  +
			\| \nabla u_{n+1}  	\|_{ L^{	2\bar{r}					}_t L^{ \frac{d\bar{r}}{(2-\ee)\bar{r}-1}	}_x}
		\big)
		+
		C_1\bar{\eta}
			\| \bar{u} 			\|_{ \BB_{p,\bar{r}}^{\frac{d}{p}-1+\ee} }.
	\end{equation*}
	Moreover, thanks to Lemma \ref{Lemma3} and Theorem \ref{Maximal_regularity_theorem}, we get
	\begin{equation*}
	\begin{aligned}
		\| F_{n+1}^2 +F_{n+1}^3   				\|_{ L^{ 	2\bar{r}					}_t L^{ \frac{d\bar{r}}{\bar{r}(1-\ee)-1}	}_x} + 
		\| \nabla F_{n+1}^2 + \nabla F_{n+1}^3	\|_{ L^{	2\bar{r}					}_t L^{ \frac{d\bar{r}}{(2-\ee)\bar{r}-1}	}_x}  
		\leq 
		C
		\| \nu - 1 \|_{L^\infty_x} 
		\| \nabla u^n \|_{ L^{	2\bar{r}						}_t L^{ \frac{d\bar{r}}{(2-\ee)\bar{r}-1}	}_x}.
	\end{aligned}
	\end{equation*}
	Summarizing the previous estimates and absorbing the terms with indexes $n+1$ on the right side by the left-hand side, we get that there exists a positive constant $C$ 
	such that
	\begin{align*}
		\| u_{n+1}			\|_{ L^{2\bar{r}						}_t L^{ \frac{d\bar{r}}{\bar{r}(1-\ee)-1}	}_x}  + 
		\| \nabla u_{n+1} 	\|_{ L^{	2\bar{r}					}_t L^{ \frac{d\bar{r}}{(2-\ee)\bar{r}-1}	}_x} 
		\leq 
		(C(1+C_1\bar{\eta})+\bar{C}_5C_1\bar{\eta})\| \bar{u} \|_{ \BB_{p,r}^{\frac{d}{p}-1+\ee} },
	\end{align*}
	thus \eqref{induction_ee} is true for any positive integer $n$, assuming $\bar{C}_5>2C$ and $\bar{\eta}$ small enough. Now recalling that $\bar{r}=2r/(2-\ee r)$, 
	\eqref{induction_ee} can be reformulated by
	\begin{equation}\label{estimate_nee}
		\| 			u_n \|_{ L^{	\frac{4r}{2-\ee r}			}_t L^{ \frac{2dr}{(2-\ee)r-2}	}_x} +
		\| \nabla 	u_n \|_{ L^{	\frac{4r}{2-\ee r}			}_t L^{ \frac{2dr}{(4-\ee)r-2}	}_x}
		\leq 
		\bar{C}_5\| \bar{u} \|_{ \BB_{p,\bar{r}}^{\frac{d}{p}-1+\ee }}\leq 
		C_5\| \bar{u} \|_{ \BB_{p,\bar{r}}^{\frac{d}{p}-1+\ee }},
	\end{equation}
	for a suitable positive constant $C_5$.
	
	\vspace{0.1cm}
	\noindent
	Now we want to prove the existence of a positive constant $C_6$ such that
	\begin{equation}\label{induction_ee/2}
			\| 			u_n \|_{ L^{\frac{4r}{2-\ee r}		}_t L^{ \frac{dr}{r-1}			}_x}	+
			\| 			u_n \|_{ L^{2r						}_t L^{ \frac{2dr}{(2-\ee)r-2}	}_x} + 
			\| \nabla 	u_n \|_{ L^{\frac{4r}{2-\ee r}		}_t L^{ \frac{dr}{2r-1}			}_x}+
			\| \nabla 	u_n \|_{ L^{	2r					}_t L^{ \frac{2dr}{(4-\ee)r-2}	}_x} \leq 
			C_6 \| \bar{u} \|_{ \BB_{p,r}^{\frac{d}{p}-1+\frac{\ee}{2}} }.						
	\end{equation}
	Let us remark that such spaces are well defined, since $2-\ee r>0$ (from $\ee< 2/r$) and $(2-\ee)r-2>0$ (from $\ee/2<\ee<1-1/r$). 
	Proceeding exactly as for proving \eqref{induction_ee}, with $r$ instead of $\bar{r}$ and $\ee/2$ instead of $\ee$, we get
	\begin{equation}\label{estimate_ee/2_part1}
		\| 			u_n \|_{ L^{	2r					}_t L^{ \frac{2dr}{(2-\ee)r-2}	}_x} +
		\| \nabla 	u_n \|_{ L^{	2r					}_t L^{ \frac{2dr}{(4-\ee)r-2}	}_x}
		\leq 
		\bar{C}_6 \| \bar{u} \|_{ \BB_{p,\bar{r}}^{\frac{d}{p}-1+\frac{\ee}{2} }},
	\end{equation}
	for a suitable positive constant $\bar{C}_6$. Furthermore, recalling the mild formulation of $u_{n+1}$ \eqref{def_u_n+1}, Corollary 
	\ref{Cor_Characterization_of_hom_Besov_spaces}  and Theorem \ref{Theorem_embedding_Besov} implies
	\begin{equation*}
		\| 			u_L \|_{ L^{\frac{4r}{2-\ee r}	}_t L^{ \frac{dr}{r-1}			}_x}	+
		\| \nabla 	u_L \|_{ L^{\frac{4r}{2-\ee r}	}_t L^{ \frac{dr}{2r-1}			}_x}	+
		\leq 
		C \| \bar{u} \|_{ \BB_{p,r}^{\frac{d}{p}-1+\frac{\ee}{2}} },
	\end{equation*}
	for a suitable positive constant $C$. Thanks to Lemma \ref{Lemma2} and Lemma \ref{Lemma3}, we obtain
	\begin{equation*}
	\begin{aligned}
		\| 			F_{n+1}^1 		\|_{ L^{\frac{4r}{2-\ee r}	}_t L^{ \frac{dr}{r-1}			}_x} 	+
		\| \nabla 	F_{n+1}^1 		\|_{ L^{\frac{4r}{2-\ee r}	}_t L^{ \frac{dr}{2r-1}			}_x}	+ 
		\leq 
		C\| g_{n+1} \|_{L^{\frac{2r}{2-\ee r}}_t L^{\frac{dr}{3r-2}}_x}.
	\end{aligned}
	\end{equation*}
	From the definition of $g_{n+1}$ \eqref{def_gn} and the estimates \eqref{horizantal_inequality_apprx_system}, we get that
	\begin{equation*}
	\begin{aligned}
		\| g_{n+1} \|_{L^{\frac{r}{1-\ee r}}_t L^{\frac{dr}{3r-2}}_x} \leq 
		C_1\eta 
		\big(\, 
			\| u_n^d 				\|_{L^{\frac{4r}{2-\ee r}}_t L^{\frac{dr}{r-1}}_x} & + 
			\| u_{n+1}^d 			\|_{L^{\frac{4r}{2-\ee r}}_t L^{\frac{dr}{r-1}}_x}  + \\&+
			\| \nabla u_n^d		 	\|_{L^{\frac{4r}{2-\ee r}}_t L^{ \frac{dr}{2r-1}}_x}+
			\| \nabla u_{n+1}^d  	\|_{L^{\frac{4r}{2-\ee r}}_t L^{ \frac{dr}{2r-1}}_x} 
		\big),
	\end{aligned}
	\end{equation*}
	so that, by the induction hypotheses of \eqref{induction_ee/2}, we have the following bound
	\begin{equation*}
		\| g_{n+1} \|_{L^{\frac{r}{1-\ee r}}_t L^{\frac{dr}{3r-2}}_x} \leq 
		C_1\eta
		\big(\, 
			\| u_{n+1} 			\|_{L^{\frac{4r}{2-\ee r}}_t L^{\frac{dr}{r-1}}_x	}  +
			\| \nabla u_{n+1}  	\|_{L^{\frac{4r}{2-\ee r}}_t L^{ \frac{dr}{2r-1}}_x	} 
		\big)
		+
		C_1\eta
			\| \bar{u} 			\|_{ \BB_{p,r}^{\frac{d}{p}-1+\frac{\ee}{2}} }.
	\end{equation*}
	Finally, thanks to Lemma \ref{Lemma1} and \ref{Lemma2}, we get
	\begin{equation*}
	\begin{aligned}
		\| F_{n+1}^2 + F_{n+1}^3 				\|_{ L^{\frac{4r}{2-\ee r}	}_t L^{\frac{dr}{r-1}			}_x}	+
		\| \nabla F_{n+1}^2 + \nabla F_{n+1}^3	\|_{ L^{\frac{4r}{2-\ee r}	}_t L^{ \frac{dr}{2r-1}			}_x}    
		\leq 
		C
		\| \nu - 1 \|_{L^\infty} 
		\| \nabla u^n \|_{ L^{\frac{4r}{2-\ee r}	}_t L^{ \frac{dr}{2r-1}			}_x}.
	\end{aligned}
	\end{equation*}
	Summarizing the previous estimates and absorbing the terms with indexes $n+1$ on the right side by the left-hand side, we get that there exists a positive constant 
	$C$ such that
	\begin{equation}\label{estimate_ee/2_2part}
		\| u_{n+1} 			\|_{L^{\frac{4r}{2-\ee r}	}_t L^{\frac{dr}{r-1}			}_x}+
		\| \nabla u_{n+1} 	\|_{L^{\frac{4r}{2-\ee r}	}_t L^{ \frac{dr}{2r-1}			}_x} 
		\leq 
		(C(1+C_1\bar{\eta})+C_6C_1\bar{\eta})\| \bar{u} \|_{ \BB_{p,r}^{\frac{d}{p}-1+\frac{\ee}{2}} }.
	\end{equation}
	Thus, recalling \eqref{estimate_ee/2_part1} and \eqref{estimate_ee/2_2part}, we get that \eqref{induction_ee/2} is true for any $n\in\NN$, with $C_6>\bar{C}_6+2C$ 
	and $\eta$ small enough. 
	
	\vspace{0.2cm}
	\textsc{Step 3. Convergence of the Series.} 
	We denote by $	\delta u_n:= u_{n+1}-u_{n}$ by $\delta \nu_n:= \nu(\theta_{n+1})-\nu(\theta_{n})$ and by $\delta \theta_n:= \theta_{n+1}-\theta_{n}$, for every 
	positive integer $n$. Moreover, fixing $\lambda>0$, we define
	\begin{equation*}
	\begin{aligned}
	\delta U_{n,\lambda}(T):=
			\|\delta u_{n,\lambda}\|_{L^{2r}(0,T;L^{\frac{dr}{r-1}}_x)}
			&+\|\delta u_{n,\lambda}\|_{L^{2r}(0,T;L^{\frac{2dr}{(2-\ee)r-2}}_x)}\\&
			+\|\nabla \delta u_{n,\lambda}\|_{L^{2r}(0,T;L^{\frac{dr}{2r-1}}_x)}
			+\|\nabla \delta u_{n,\lambda}\|_{L^{2r}(0,T;L^{\frac{2dr}{(4-\ee)r-2}}_x)},
	\end{aligned}
	\end{equation*}
	where, recalling \eqref{def_h}, $\delta u_{n,\lambda}(t):=\delta u_n(t)h_{n,\lambda}(0,t)$. We want to prove that the series $\sum_{n\in \NN}\delta U_{n}(T)$ is finite.
	Denoting by $\delta g_{n}:=g_{n+1}-g_{n}$, $\delta \MM_n:=\MM_{n+1}-\MM_n$, then, thanks to the equality \eqref{def_u_n+1}, we can formulate 
	$\delta u_{n,\lambda }=
	f_{n,1}+f_{n,2}+f_{n,3}$, where 
	\begin{equation}\label{formulation_delta_u_n}
	\begin{aligned}
		f_{n,1}&:= h_{n,\lambda}(0,t)\int_{0}^t e^{(t-s)\Delta}\PP\delta g_{n}(s)\dd s ,\\
		f_{n,2}&:=h_{n,\lambda}(0,t)\int_{0}^t\big[  \nabla e^{(t-s)\Delta} R\cdot R\cdot \{( \nu(\theta_n)-1) \delta \MM_{n-1}\}+
					\Div\, e^{(t-s)\Delta} \{( \nu(\theta_n)-1)  \delta \MM_{n-1}\}\big](s)\dd s,\\
		f_{n,3}&:=h_{n,\lambda}(0,t)\big(\int_{0}^t \nabla e^{(t-s)\Delta} R\cdot R\cdot \{\delta \nu_{n}\MM_n\}(s)\dd s
		+h_{n,\lambda}(0,t)\int_{0}^t \Div\, e^{(t-s)\Delta} \{\delta \nu_{n}\MM_n\}(s)\dd s\big).
	\end{aligned}
	\end{equation}
	At first step let us estimate 
	\begin{equation*}
		\|f_{n,1}\|_{L^{2r}(0,T;L^{\frac{dr}{r-1}}_x)}
			+\|f_{n,1}\|_{L^{2r}(0,T;L^{\frac{2dr}{(2-\ee)r-2}}_x)}
				+\|\nabla f_{n,1}\|_{L^{2r}(0,T;L^{\frac{dr}{2r-1}}_x)}
					+\|\nabla f_{n,1}\|_{L^{2r}(0,T;L^{\frac{2dr}{(4-\ee)r-2}}_x)},
	\end{equation*}
	Observing that
	\begin{equation*}
		\delta g_{n}=
		-
		\left(\,
		\begin{matrix}
			u_n^d \partial_d \delta u_n^h + \delta u_n^d \partial_d u_n^h+
			\delta u_{n-1}^h\cdot \nabla u_n^h + u_{n-1}^h \cdot \nabla \delta u_{n-1}^h\\
			\\
			\nabla^h u_n^d\cdot\delta u_n^h+\nabla^h\delta u_n^d\cdot u_n^h-
			u_n^d \Div^h\delta u_n^h-\delta u_{n-1}^d \Div^h u_n^h\\
			\end{matrix}\,
		\right),
	\end{equation*}
	then, by Lemma \ref{Lemma_A.1} and Lemma \ref{Lemma_A.2}, we obtain
	\begin{align*}
		\|f_{n,1}&\|_{L^{2r}(0,T;L^{\frac{dr}{r-1}}_x)}+\|\nabla f_{n,1}\|_{L^{2r}(0,T;L^{\frac{dr}{2r-1}}_x)}\leq\\
		&\leq
		C\bigg\{ 
			\frac{1}{\lambda^{\frac{1}{4r}}}
			\Big(
				\|  u_n^d\|_{L^{2r}_t L^{\frac{dr}{r-1}}_x}^{\frac{1}{2}} 
				\| \partial_d \delta u_{n,\lambda}^h\|_{L^{2r}(0,T;L^{\frac{dr}{2r-1}}_x)}
				+
				\|\nabla^h u_n^d\|_{L^{2r}_t L^{\frac{dr}{2r-1}}_x}^\frac{1}{2}
				\|\delta u_{n,\lambda}^h\|_{L^{2r}(0,T; L^{\frac{dr}{r-1}}_x)}+\\
		&\quad\quad\quad+
				\|u_n^d\|_{L^{2r}_t L^{\frac{dr}{r-1}}_x}^\frac{1}{2} 
				\|\nabla^h \delta u_{n,\lambda}^h\|_{L^{2r}(0,T;L^{\frac{dr}{2r-1}}_x)}\,
			\Big)
			+
			\|\delta u_{n,\lambda}^d \|_{L^{2r}(0,T;L^{\frac{dr}{r-1}}_x)}
			\| \partial_d u_n^h\|_{L^{2r}_tL^{\frac{dr}{2r-1}}_x}+\\
		&\quad\quad\quad+	
			\|\delta u_{n-1,\lambda}^h \|_{L^{2r}(0,T;L^{\frac{dr}{r-1}}_x)}
			\|\nabla u_n^h \|_{L^{2r}_t L^{\frac{dr}{2r-1}}_x} 
			+
			\|u_{n-1}^h \|_{L^{2r}_t L^{\frac{dr}{r-1}}_x}
			\| \nabla \delta u_{n-1,\lambda}^h\|_{L^{2r}(0,T;L^{\frac{dr}{2r-1}}_x)}+\\
		&\quad\quad\quad+	
			\|\nabla^h\delta u_{n,\lambda}^d\|_{L^{2r}(0,T;L^{\frac{dr}{2r-1}}_x)} 
			\|u_n^h\|_{L^{2r}_t L^{\frac{dr}{r-1}}_x}
			+
			\|\delta u_{n-1,\lambda}^d\|_{L^{2r}(0,T;L^{\frac{dr}{r-1}}_x)} 
			\|\nabla^h u_n^h\|_{L^{2r}_t L^{\frac{dr}{2r-1}}_x}
		\bigg\}.
	\end{align*}
	which yields, by \eqref{horizantal_inequality_apprx_system} and \eqref{inequality_lambda1}
	\begin{align*}
		\|f_{n,1}&\|_{L^{2r}(0,T;L^{\frac{dr}{r-1}}_x)}+\|\nabla f_{n,1}\|_{L^{2r}(0,T;L^{\frac{dr}{2r-1}}_x)}
		\leq
		\frac{1}{4}
			\| \nabla \delta u_{n,\lambda}^h\|_{L^{2r}(0,T;L^{\frac{dr}{2r-1}}_x)}
			+
			C\bar{C}_1\tilde{\eta} \|\delta u_{n,\lambda}^d \|_{L^{2r}(0,T;L^{\frac{dr}{r-1}}_x)}+\\
		&+
			C\bar{C}_1\tilde{C}_r\eta 
			\Big(
				\|\delta u_{n-1,\lambda}^h \|_{L^{2r}(0,T;L^{\frac{dr}{r-1}}_x)}
				+
				\| \nabla \delta u_{n-1}^h\|_{L^{2r}(0,T;L^{\frac{dr}{2r-1}}_x)}
			\Big)
			+
			\frac{1}{4}\|\delta u_{n,\lambda}^h\|_{L^{2r}(0,T; L^{\frac{dr}{r-1}}_x)}
			+\\
		&+
			C\bar{C}_1\tilde{\eta}\|\nabla^h\delta u_{n,\lambda}^d\|_{L^{2r}(0,T;L^{\frac{dr}{2r-1}}_x)}+ 
			\frac{1}{4} 
			\|\nabla^h \delta u_{n,\lambda}^h\|_{L^{2r}(0,T;L^{\frac{dr}{2r-1}}_x)}
			+
			C\bar{C}_1\tilde{C}_r\eta
			\|\delta u_{n-1,\lambda}^d\|_{L^{2r}(0,T;L^{\frac{dr}{r-1}}_x)}.
	\end{align*}
	Assuming $\eta$ small enough, the previous inequality yields
	\begin{equation}\label{estimate_f_1}
	\begin{aligned}
		\|f_{n,1}\|_{L^{2r}(0,T;L^{\frac{dr}{r-1}}_x)}&+\|\nabla f_{n,1}\|_{L^{2r}(0,T;L^{\frac{dr}{2r-1}}_x)}\leq
		\frac{1}{4}
			\big\{ 
				\|\delta u_{n,\lambda}\|_{L^{2r}(0,T; L^{\frac{dr}{r-1}}_x)}+\\
		&+			
				\|\delta u_{n-1,\lambda}\|_{L^{2r}(0,T;L^{\frac{dr}{r-1}}_x)}+
				\|\nabla \delta u_{n,\lambda}\|_{L^{2r}(0,T; L^{\frac{dr}{2r-1}}_x)}+
				\|\nabla \delta u_{n-1,\lambda}\|_{L^{2r}(0,T;L^{\frac{dr}{2r-1}}_x)}
			\big\}
	\end{aligned}
	\end{equation}	
	Now, let us estimate $f_{n,1}$ and $\nabla f_{n,1}$ in $L^{2r}(0,T; L^{2dr/((2-\ee)r-2)}_x)$ and $L^{2r}(0,T; L^{2dr/((4-\ee)r-2)}_x)$ respectively. 
	Thanks to Lemma \ref{Lemma_A.1} and \ref{Lemma_A.2}, the following inequality is satisfied:
	\begin{align*}
		\|&f_{n,1}\|_{L^{2r}(0,T;L^{\frac{2dr}{(2-\ee)r-2}}_x)}+\|\nabla f_{n,1}\|_{L^{2r}(0,T;L^{\frac{2dr}{(4-\ee)r-2}}_x)}\leq\\
		&\leq
		C\bigg\{ 
			\frac{1}{\lambda^{\frac{1}{4r}}}
			\Big(
				\|  u_n^d\|_{L^{2r}_t L^{\frac{dr}{r-1}}_x}^{\frac{1}{2}} 
				\| \partial_d \delta u_{n,\lambda}^h\|_{L^{2r}(0,T;L^{\frac{2dr}{(4-\ee)r-2}}_x)}
				+
				\|\nabla^h u_n^d\|_{L^{2r}_t L^{\frac{dr}{2r-1}}_x}^\frac{1}{2}
				\|\delta u_{n,\lambda}^h\|_{L^{2r}(0,T; L^{\frac{2dr}{(2-\ee)r-2}}_x)}+\\
		&\quad\quad\quad+
				\|u_n^d\|_{L^{2r}_t L^{\frac{dr}{r-1}}_x}^\frac{1}{2} 
				\|\nabla^h \delta u_{n,\lambda}^h\|_{L^{2r}(0,T;L^{\frac{2dr}{(4-\ee)r-2}}_x)}\,
			\Big)
			+
			\|\delta u_{n,\lambda}^d \|_{L^{2r}(0,T;L^{\frac{2dr}{(2-\ee)r-2}}_x)}
			\| \partial_d u_n^h\|_{L^{2r}_tL^{\frac{dr}{2r-1}}_x}+\\
		&\quad\quad\quad+	
			\|\delta u_{n-1,\lambda}^h \|_{L^{2r}(0,T;L^{\frac{2dr}{(2-\ee)r-2}}_x)}
			\|\nabla u_n^h \|_{L^{2r}_t L^{\frac{dr}{2r-1}}_x} 
			+
			\|u_{n-1}^h \|_{L^{2r}_t L^{\frac{dr}{r-1}}_x}
			\| \nabla \delta u_{n-1,\lambda}^h\|_{L^{2r}(0,T;L^{\frac{2dr}{(4-\ee)r-2}}_x)}+\\
		&\quad\quad\quad+	
			\|\nabla^h\delta u_{n,\lambda}^d\|_{L^{2r}(0,T;L^{\frac{2dr}{(4-\ee)r-2}}_x)} 
			\|u_n^h\|_{L^{2r}_t L^{\frac{dr}{r-1}}_x}
			+
			\|\delta u_{n-1,\lambda}^d\|_{L^{2r}(0,T;L^{\frac{2dr}{(2-\ee)r-2}}_x)} 
			\|\nabla^h u_n^h\|_{L^{2r}_t L^{\frac{dr}{2r-1}}_x}
		\bigg\}.
	\end{align*}
	Hence, \eqref{horizantal_inequality_apprx_system}, \eqref{inequality_lambda1} and the smallness condition on $\eta$ imply that
	\begin{equation}\label{estimate2_f_1}
		\begin{aligned}
		\|f_{n,1}&\|_{L^{2r}(0,T;L^{\frac{2dr}{(2-\ee)r-2}}_x)}+\|\nabla f_{n,1}\|_{L^{2r}(0,T;L^{\frac{2dr}{(4-\ee)r-2}}_x)}\leq
		\frac{1}{4}
			\big\{ 
				\|\delta u_{n,\lambda}\|_{L^{2r}(0,T; L^{\frac{2dr}{(2-\ee)r-2}}_x)}+\\
		&+			
				\|\delta u_{n-1,\lambda}\|_{L^{2r}(0,T;L^{\frac{2dr}{(2-\ee)r-2}}_x)}+
				\|\nabla \delta u_{n,\lambda}\|_{L^{2r}(0,T; L^{\frac{2dr}{(4-\ee)r-2}}_x)}+
				\|\nabla \delta u_{n-1,\lambda}\|_{L^{2r}(0,T;L^{\frac{2dr}{(4-\ee)r-2}}_x)}
			\big\}
	\end{aligned}.
	\end{equation}
	Thus, summarizing \eqref{estimate_f_1} and \eqref{estimate2_f_1}, we obtain
	\begin{equation}\label{estimates_f_1}
	\begin{aligned}
			\|f_{n,1}\|_{L^{2r}(0,T;L^{\frac{dr}{r-1}}_x)}
			&+\|f_{n,1}\|_{L^{2r}(0,T;L^{\frac{2dr}{(2-\ee)r-2}}_x)}
				+\|\nabla f_{n,1}\|_{L^{2r}(0,T;L^{\frac{dr}{2r-1}}_x)}\\&
					+\|\nabla f_{n,1}\|_{L^{2r}(0,T;L^{\frac{2dr}{(4-\ee)r-2}}_x)}
			\leq
			\frac{1}{4}\delta U_{n,\lambda}(T)+\frac{1}{4}\delta U_{n-1,\lambda}(T).
	\end{aligned}
	\end{equation}
	Now, we want to estimate $f_{n,2}$ in $L^{2r}(0,T;L^{dr/(r-1)}_x)\cap L^{2r}(0,T;L^{2dr/((2-\ee)r-2)}_x)$ and moreover $\nabla f_{n,2}$ in 
	$L^{2r}(0,T;L^{dr/(2r-1)}_x)\cap L^{2r}(0,T;L^{2dr/((4-\ee)r-2)}_x)$. From Lemma \ref{Lemma2} and Theorem \ref{Maximal_regularity_theorem} we obtain
	\begin{align*}
		\|f_{n,2}\|_{L^{2r}(0,T;L^{\frac{dr}{r-1}}_x)}
			&+\|f_{n,2}\|_{L^{2r}(0,T;L^{\frac{2dr}{(2-\ee)r-2}}_x)}
				+\|\nabla f_{n,2}\|_{L^{2r}(0,T;L^{\frac{dr}{2r-1}}_x)}
					+\|\nabla f_{n,2}\|_{L^{2r}(0,T;L^{\frac{2dr}{(4-\ee)r-2}}_x)}\leq \\
			&\leq 
			C\|\nu-1\|_{\infty}
			\Big( 
				\|\nabla \delta u_{n-1}\|_{L^{2r}(0,T;L^{\frac{dr}{2r-1}}_x)}+
				\|\nabla \delta u_{n-1}\|_{L^{2r}(0,T;L^{\frac{2dr}{(4-\ee)r-2}}_x)}
			\Big)\\
		&\leq \tilde{C}_r\eta
			\Big( 
				\|\nabla \delta u_{n-1,\lambda}\|_{L^{2r}(0,T;L^{\frac{dr}{2r-1}}_x)}+
				\|\nabla \delta u_{n-1,\lambda}\|_{L^{2r}(0,T;L^{\frac{2dr}{(4-\ee)r-2}}_x)}
			\Big),
	\end{align*}
	hence, we deduce that
	\begin{equation}\label{estimates_f_2}
	\begin{aligned}
		\|f_{n,2}\|_{L^{2r}(0,T;L^{\frac{dr}{r-1}}_x)}
			+\|f_{n,2}\|_{L^{2r}(0,T;L^{\frac{2dr}{(2-\ee)r-2}}_x)}
				&+\|\nabla f_{n,2}\|_{L^{2r}(0,T;L^{\frac{dr}{2r-1}}_x)}+\\&
					+\|\nabla f_{n,2}\|_{L^{2r}(0,T;L^{\frac{2dr}{(4-\ee)r-2}}_x)}\leq 
			\bar{C}_r\eta \delta U_{n-1,\lambda}(T).
	\end{aligned}
	\end{equation}
	Now we deal with $f_{n,3}$ and $\nabla f_{n,3}$. At first, since $v\in C^{\infty}(\RR)$ and 
	$\|\theta_{n}\|_{L^\infty_{t,x}}\leq \|\bar{\theta}\|_{L^{\infty}_x}$, then there exists 
	$\tilde{c}>0$ (dependent on $\|\bar{\theta}\|_{L^{\infty}_x)}$) such that 
	$\|\delta \nu_n(t)\|_{L^{^\infty}_x}\leq \tilde{c}\|\delta \theta_n(t)\|_{L^{^\infty}_x}$, for almost every $t\in \RR_+$. 
	Moreover, by Lemma \ref{Lemma1} and Theorem \ref{Maximal_regularity_theorem}, we have
	\begin{align*}
		\|&f_{n,3}\|_{L^{2r}(0,T;L^{\frac{dr}{r-1}}_x)}
			+\|f_{n,3}\|_{L^{2r}(0,T;L^{\frac{2dr}{(2-\ee)r-2}}_x)}
				+\|\nabla f_{n,3}\|_{L^{2r}(0,T;L^{\frac{dr}{2r-1}}_x)}+\\&
					+\|\nabla f_{n,2}\|_{L^{2r}(0,T;L^{\frac{2dr}{(4-\ee)r-2}}_x)}\leq 
							C\big\{ 
								\| \delta \nu_n \MM_n \|_{L^{2r}(0,T; L^{ \frac{dr}{2r-1} }_x)}+ 
								\| \delta \nu_n \MM_n \|_{L^{2r}(0,T; L^{\frac{2dr}{(4-\ee)r-2}}_x)}
							\big\}.
	\end{align*}	
	Thus, recalling \eqref{estimate_nee} and \eqref{induction_ee/2}, we finally obtain
	\begin{equation}\label{estimate_f_n3_partA}
	\begin{aligned}
		\|f_{n,3}\|_{L^{2r}(0,T;L^{\frac{dr}{r-1}}_x)}
			+\|f_{n,3}\|_{L^{2r}(0,T;L^{\frac{2dr}{(2-\ee)r-2}}_x)}
				+\|\nabla f_{n,3}\|_{L^{2r}(0,T;L^{\frac{dr}{2r-1}}_x)}
					+\|\nabla f_{n,3}\|_{L^{2r}(0,T;L^{\frac{2dr}{(4-\ee)r-2}}_x)} \\
							\leq
			2C\tilde{c}\| \delta \theta_n\|_{L^{\frac{4}{\ee}}(0,T;L^{\infty}_x)}
			\big\{ 
				\| \nabla u_n \|_{L^{\frac{4r}{2-\ee r}}_t L^{ \frac{dr}{2r-1} }_x}+ 
				\| \nabla u_n \|_{L^{\frac{4r}{2-\ee r}}_t L^{ \frac{2dr}{(4-\ee)r-2}  }_x}
			\big\} \leq \hat{C}_1(\bar{u})\| \delta \theta_n\|_{L^{\frac{4}{\ee}}(0,T;L^{\infty}_x)},							
	\end{aligned}
	\end{equation}
	where $\hat{C}_1(\bar{u}):=2C \tilde{c} (C_5 \| \bar{u} \|_{\BB_{p,r}^{d/p-1+\ee}} + C_6\| \bar{u} \|_{\BB_{p,r}^{d/p-1+\ee/2}} )$. Now, let us 
	observe that $\delta \theta_n$ is the weak solution of
		\begin{equation*}
		\begin{cases}
			\partial_t \delta \theta_n -\ee \Delta \delta \theta_n  =
			-\Div (\,\delta \theta_n u_n \,) -\Div (\,\delta u_{n-1} \theta_n \,)
																				& \RR_+ \times\RR^d,\\
			\delta \theta_{n\,|t=0} =0											& \;\;\,\quad \quad \RR^d,\\
		\end{cases}
	\end{equation*}
	which implies
	\begin{equation}\label{delta_theta_n}
		\delta \theta_n(t)=	-\int_0^t \Div\, e^{\ee(t-s)\Delta}\delta \theta_n(s)u_n(s)\dd s
							-\int_0^t \Div\, e^{\ee(t-s)\Delta}\delta u_{n-1}(s)\theta_n(s)\dd s.
	\end{equation}
	By Remark \ref{remark2.1} we deduce then
	\begin{equation*}
		\| \delta \theta_n(t)\|_{L^\infty_x} \leq 
			\int_0^t 
				\frac{\|\delta \theta_n(s)u_n(s)\|_{L^{\frac{2dr}{(2-\ee )r -2}}_x}}{|\ee(t-s)|^{1-\frac{1}{2r}-\frac{\ee}{4}}}
			\dd s
			+
			\int_0^t 
				\frac{\|\delta u_{n-1}(s)\theta_n(s)\|_{L^{\frac{2dr}{(2-\ee )r -2}}_x}}{|\ee(t-s)|^{1-\frac{1}{2r}-\frac{\ee}{4}}}
			\dd s,
	\end{equation*}
	hence, defining $\alpha := (1-1/(2r) -\ee/4 )(2r)'<1$, $\| \delta \theta_n(t)\|_{L^\infty_x}^{2r}$ is bounded by
	\begin{align*}
		2^{2r-1}\bigg(
			\int_0^t
				\frac{1}{|\ee(t-s)|^\alpha}\dd s		
		\bigg)^{2r-1}
		&\big\{
			\int_0^t 
				\|\delta \theta_n(s)\|_{L^\infty_x}^{2r} \| u_n(s)\|_{L^{q*}_x}^{2r}
			\dd s+
			\int_0^t 
				\|\bar{\theta}\|_{L^\infty_x}^{2r} \|\delta u_{n-1}(s)\|_{L^{q*}_x}^{2r}
			\dd s
		\big\}.
	\end{align*}
	Then, using the Gronwall inequality, we have
	\begin{align*}
		\| \delta \theta_n(t)\|_{L^\infty_x}^{2r} \leq		
		\big( 2\frac{(1-\alpha)t^{1-\alpha}}{\ee^\alpha} \big)^{2r-1}\|\bar{\theta}\|_{L^\infty_x}^{2r}
		\int_0^t 
				 \|\delta u_{n-1}(s)\|_{L^{\frac{2dr}{(2-\ee)r-2}}_x}^{2r}
			\dd s
		\exp
		\big\{
			\int_0^t 
				 \| u_n(s)\|_{L^{\frac{2dr}{(2-\ee)r-2}}_x}^{2r}\dd s
		\big\},
	\end{align*}
	which yields 
	$
		\|\delta \theta_n(t)\|_{L^\infty_x}\leq \chi (t)\delta U_{n-1}(t),
	$
	where $\chi$ is an increasing function defined by
	\begin{equation*}
		\chi(t):=\big( 2\frac{(1-\alpha)t^{1-\alpha}}{\ee^\alpha} \big)^{1-\frac{1}{2r}}
		\exp
		\Big\{ 
			\frac{1}{2r}
			C_6\|\bar{u}\|_{\dot{B}_{p,r}^{\frac{d}{p}-1+\frac{\ee}{2}}}
		\Big\}.
	\end{equation*}
	Hence, Recalling \eqref{estimate_f_n3_partA}, we deduce that
	\begin{equation}\label{estimates_f_3}
	\begin{aligned}
		\|f_{n,3}\|_{L^{2r}(0,T;L^{\frac{dr}{r-1}}_x)}	+	
			\|f_{n,3}\|_{L^{2r}(0,T;L^{\frac{2dr}{(2-\ee)r-2}}_x)}&+
				\|\nabla f_{n,3}\|_{L^{2r}(0,T ;L_x^{ \frac{2dr}{(4-\ee)r-2}})}\\
				&+\|\nabla f_{n,3}\|_{L^{2r}(0,T ;L_x^{ \frac{dr}{2r-1} })}
							\leq \hat{C}_1(\bar{u})\chi(T)\|\delta U_{n-1} \|_{L^{\frac{4}{\ee}}(0,T)}.							
	\end{aligned}
	\end{equation}
	Summarizing \eqref{estimates_f_1}, \eqref{estimates_f_2} and \eqref{estimates_f_3} we finally deduce that
	\begin{equation}\label{est_deltaUn}
		\delta U_{n,\lambda}(T) \leq  \Big(\frac{1}{3}+\frac{4}{3}\tilde{C}_r \eta \Big)\delta U_{n-1,\lambda}(T)
		+ \frac{4}{3}\hat{C}_1(\bar{u})\chi(T)\|\delta U_{n-1} \|_{L^{\frac{4}{\ee}}(0,T)},
	\end{equation}
	Supposing $\eta$ small enough, we can assume $	\mu:=(1/3+4\tilde{C}_r\eta/3  )<1$. Thus, fixing $T>0$ and denoting by $C_T$ the 
	constant $4\bar{C}_1(\bar{u})\chi(T)	\exp\{ \lambda ( \bar{C}_2\|\bar{u}^d\|_{\BB_{p,r}^{\frac{d}{p}-1}}+\bar{C}_3)\}/3$,
	then we have
	\begin{equation*}
		\delta U_{n,\lambda}(t) \leq  \mu\,\delta U_{n-1,\lambda}(t)
		+ C_T\|\delta U_{n-1,\lambda} \|_{L^{\frac{4}{\ee}}(0,t)},
	\end{equation*}
	for all $t\in [0,T]$, where we have used that $\chi$ is an increasing function.
	Now, let us prove by induction that there exists $C=C(T)>0$ and $K=K(T)>0$ such that
	\begin{equation}\label{induction_delta_U}
		\delta U_{n,\lambda}(t)\leq C\mu^{\frac{n}{2}}\exp\big\{ K\frac{t}{\sqrt{\mu}}\big\},
	\end{equation}
	for all $t\in [0,T]$ and for all $n\in\NN$. The base case is trivial, since it is sufficient to find $C=C(T)>0$ such that
	$
		\delta U_{0,\lambda}(t)\leq C
	$, for all $t\in [0,T]$. 
	Then
	$
		\delta U_{0,\lambda}(t)\leq C\exp\{ Kt/\mu\}
	$, for all $K>0$. Passing to the induction,
	\begin{align*}
		\delta U_{n+1,\lambda}(t)
			&\leq\mu\delta U_{n-1,\lambda}(t)+ C_T\|\delta U_{n-1,\lambda} \|_{L^{\frac{4}{\ee}}(0,t)}
			\leq \sqrt{\mu}C \mu^{\frac{n+1}{2}}+
					C_TC\mu^{\frac{n}{2}}
					\Big(
						\int_0^t \exp\big\{\frac{4}{\ee} K\frac{s}{\sqrt{\bar{\eta}}} \big\}\dd s
					\Big)^\frac{\ee}{4}\\
			&\leq 	\big( 
						\sqrt{\mu} + \big( \frac{\ee}{4K}\big)^{\frac{4}{\ee}} \mu^{\frac{\ee}{8}-\frac{1}{2}} C_T
					\big)C\mu^{\frac{n+1}{2}}\exp\big\{ K\frac{t}{\sqrt{\bar{\eta}}}\big\}. 
	\end{align*}
	Chosen $K>0$ big enough, we finally obtain that \eqref{induction_delta_U} is true for any positive integer $n$. Hence, the series 
	$\sum_{n\in\NN} \delta u_{n,\lambda}(T)$ is convergent, for any $T\in\RR_+$. This yields that 
	\begin{equation*}
		\sum_{n\in\NN} \delta U_n(T)\leq 
		\exp\Big\{ \lambda \big( \bar{C}_2\|\bar{u}^d\|_{\dot{B}_{p,r}^{-1+\frac{d}{p}}}+\bar{C}_3 \big)^{2r}\Big\}
		\sum_{n\in\NN} \delta U_{n,\lambda}(T)<\infty,
	\end{equation*}
	so that $(u_n)_\NN$ and $(\nabla u_n)_\NN$  are Cauchy sequences in $L^{2r}(0,T;L^{dr/(r-1)}_x)$ and
	$L^{2r}(0,T;L^{\frac{dr}{2r-1}}_x)$ respectively. Furthermore, $(\theta_n)_\NN$ is a Cauchy sequence in 	$L^{\infty}(\, (0,T)\times \RR^d)$, 
	since $\|\delta \theta_n\|_{L^{\infty}(\, (0,T)\times \RR^d)}$ is bounded by $\chi(T)\delta U_{n-1}(T)$. 
	Recalling also the definition of $\delta g_n$ \eqref{def_gn}, we get
	\begin{equation*}
		\sum_{n\in\NN}\|\delta g_n \|_{L^{r}(0,T;L^{dr/(3r-2)}_x)}<\infty,	
	\end{equation*}
	for all $T>0$. Thus $(g_n)_\NN$ is a Cauchy sequence in $ L^{r}(0,T;L^{dr/(3r-2)}_x)$ and 
	$ (\,{(\sqrt{-\Delta})^{-1}}g_n)_\NN$ is a Cauchy sequence in $L^{r}(0,T;L^{dr/(2r-2)}_x)$, thanks to 
	Corollary \ref{Hardy-Litlewood-Sobolev-Inequality}. Recalling the Mild formulation \eqref{formulation_delta_u_n}, by Lemma \ref{Lemma1} and Theorem 
	\ref{Maximal_regularity_theorem}, there exist $C>0$ such that
	\begin{align*}
		\|  \nabla \delta u_n \|_{L^{r}(0,T;L^{\frac{dr}{2(r-1)}}_x)}\leq C
		\Big\{
			\|\delta g_n\|_{L^{r}(0,T;L^{\frac{dr}{3r-2}}_x)}
			&+\|\nu-1\|_\infty 
			\|\nabla \delta u_{n-1} \|_{L^{r}(0,T;L^{\frac{dr}{2(r-1)}}_x)}+\\
			&+
			\|\delta \nu_n\|_{L^{\infty}(\,(0,T)\times\RR^d )}\|\nabla u_n \|_{L^{r}_t L^{\frac{dr}{2(r-1)}}_x)}
		\Big\}
	\end{align*}
	for all $n\in\NN$. Hence the series 
	$
		\sum_{n\in\NN}\|\nabla\delta u_n \|_{L^{r}(0,T;L^{dr/(2r-2)}_x)}
	$
	is finite, which implies that $(\nabla u_n)_\NN$ is a Cauchy sequence in 
	$
		L^{r}(0,T;L^{dr/(2r-2)}_x).
	$
	Finally $(\Pi_n)_\NN$ is a Cauchy sequence in $L^{r}(0,T;L^{dr/(2r-2)}_x)$, by \eqref{def_Pi_n} and this concludes the proof of the Proposition.
\end{proof}
\noindent
Now, let us prove that system \eqref{Navier_Stokes_system} admits a weak solution, adding some regularity to the initial data.
\begin{theorem}\label{Theorem_solutions_smooth_dates}
	Let $1<r<\infty$ and $p\in (1,dr/(2r-1))$. Suppose that $\bar{\theta}$ belongs to $L^\infty_x\cap L^2_x$ and $\bar{u}$ belongs to 
	$\BB_{p,r}^{d/p-1}\cap \BB_{p,r}^{d/p-1+\ee} $ with $\ee<\min\{1/(2r), 1-1/r, 2(d/p -2 + 1/r)\}$. 
	If \eqref{smallness_condition} holds, then there exists a global weak solution $(\theta, u, \Pi)$  
	of \eqref{Navier_Stokes_system_eps} which satisfies the properties of Theorem \ref{Main_Theorem}.
\end{theorem}
\begin{proof}
	By Proposition \ref{Proposition_solutions_smooth_data}, there exist $u_\ee$ in $L^{2r}_t L^{dr/(r-1)}_x$ with $\nabla u_\ee$ in 
	$L^{2r}_t L^{dr/(2r-1)}_x\cap L^{r}_t L^{dr/(2r-2)}_x
	$, and also $\theta_\ee \in L^{\infty}(\RR_+\times \RR^d),\quad \Pi_\ee\in L^{r}_t L^{\frac{dr}{2(r-1)}}_x$, 
	such that $(\theta_\ee,\,u_\ee,\,\Pi_\ee)$ is weak solution of \eqref{Navier_Stokes_system_eps}. Moreover, thanks to 
	\eqref{inequalities_statement_prop_smooth_dates}, we have the following weakly convergences:
	\begin{equation*}
	\begin{array}{lll}
				u_{\ee_n} 				\rightharpoonup  		 	u	\quad	w-L^{2r}_t L^{\frac{dr}{r-1}	}_x,	&
		\nabla 	u_{\ee_n} 				\rightharpoonup 	\nabla 	u 	\quad 	w-L^{2r}_t L^{\frac{dr}{2r-1}	}_x,	&
		\nabla 	u_{\ee_n} 				\rightharpoonup 	\nabla 	u	\quad	w-L^{2r}_t L^{\frac{dr}{2(r-1)}}_x,		\\
		\theta_{\ee_n} 	\overset{*}{	\rightharpoonup} 	\theta		\quad	w*-L^{\infty}_{t,x},					&
		\Pi_{\ee_n}						\rightharpoonup 	\Pi			\quad	w-L^{r}_t L^{\frac{dr}{2(r-1)}}_x,
	\end{array}
	\end{equation*}
	for a  positive decreasing sequence $(\ee_n)_\NN$ which is convergent to $0$. We want to prove that $(\theta, u, \Pi)$ is a weak solution of
	\eqref{Navier_Stokes_system}.
	First let us observe that $\{ u_\ee\,|\,\ee>0\}$ is a compact set on 
	$C([0,T]; \dot{W}_x^{-1,dr/(2r-2)})$, for all $T>0$. Indeed, recalling the momentum equation of 
	\eqref{Navier_Stokes_system_eps}, $\partial_t {(\sqrt{-\Delta})^{-1}}u_\ee $ is uniformly bounded in $L^r(0,T;L^{dr/(2r-2)}_x)$. This yields that 
	$\{(\sqrt{-\Delta})^{-1}u_\ee\,|\,\ee>0\}$ is an equicontinuous and bounded family on 
	$C([0,T], L_x^{dr/(2r-2)})$. Hence we can assume that $(\sqrt{-\Delta})^{-1} u_{\ee_n}$ strongly converges to 
	$(\sqrt{-\Delta})^{-1} u$ in $L^{\infty}(0,T;L^{dr/(2r-2)}_x)$, namely  $ u_{\ee_n}$ strongly converges to $u$ in $L^{\infty}(0,T; \dot{W}_x^{-1,dr/(2r-2)} )$. 
	We recall that $(\nabla u_{\ee_n})_\NN$ is a bounded sequence on $L^{r}_t L^{dr/(2r-2)}_x$, so that $(u_{\ee_n})_\NN$ is a bounded sequence on 
	$L^{r}_t \dot{W}^{1,dr/(2r-2)}_x$. Thus, passing through the following real interpolation
	 \begin{equation*}
	 	\Big[  \dot{W}_x^{-1,\frac{dr}{2(r-1)}},  \dot{W}_x^{+1,\frac{dr}{2(r-1)}}\Big]_{\frac{1}{2r},1}=
	 	\dot{B}_{\frac{dr}{2(r-1)},1}^{1-\frac{1}{r}},
	\end{equation*}
	(see \cite{MR0482275}, Theorem $6.3.1$), and since 
	$	\dot{B}_{dr/(2r-2),1}^{1-\frac{1}{r}}\hookrightarrow L^{dr/(r-1)}_x $	
	(see \cite{MR2768550}, Theorem $2.39$), we deduce that,
	\begin{align*}
		\|u_{\ee_n}-u\|_{L^{2r}(0,T;L^{\frac{dr}{r-1}}_x)}
		&\leq 
		C	\Big\| 
				\|u_{\ee_n}-u\|_{\dot{W}_x^{-1,\frac{dr}{2(r-1)}}}^{1-\frac{1}{2r}} 
				\|u_{\ee_n}-u\|_{\dot{W}_x^{1,\frac{dr}{2(r-1)}}}^{\frac{1}{2r}}
			\Big\|_{L^{2r}(0,T)}\\
		&\leq
		C	\|u_{\ee_n}-u\|_{L^{\infty}(0,T;\dot{W}^{-1,\frac{dr}{2(r-1)}}_x)}^{1-\frac{1}{2r}}
			\|u_{\ee_n}-u\|_{L^{1}(0,T;\dot{W}^{1,\frac{dr}{2(r-1)}}_x)}^{\frac{1}{2r}},
	\end{align*}
	for all $T>0$. This implies that $u_{\ee_n}$ strongly converges to $u$ in $L^{2r}_{loc}(\RR_+;L^{\frac{dr}{r-1}}_x)$, for all $T>0$, and moreover that 
	$u_{\ee_n}\theta_{\ee_n}$ and $u_{\ee_n}\cdot \nabla u_{\ee_n}$ converge to $u\, \theta$ and $u\cdot \nabla u$, respectively, in the distributional sense.
	We deduce that $\theta$ is a weak solution of
	\begin{equation}\label{prop_smooth_data_transport_eqution}
		\partial_t\theta + \Div (\theta u)=0\quad\text{in}\quad \RR_+ \times\RR^d,\quad\quad	
		\theta_{|t=0} = \bar{\theta}	\quad\text{in }\quad\RR^d.
	\end{equation}
	Now, we claim that $\theta_{\ee_n}\rightarrow\theta$ almost everywhere on $\RR_+\times \RR^d$, up to a subsequence. Multiplying the first equation of 
	\eqref{Navier_Stokes_system_eps} by $\theta/2$ and integrating in $[0,t)\times \RR^d$ we get
	\begin{equation*}
		\| \theta_{\ee_n}(t) \|_{L^2_x}^2 + 
		\ee_n \int_0^t \| \nabla \theta_\ee (s) \|_{L^2_x}^2 \dd s =  
		\| \bar{\theta} \|_{L^2_x},
	\end{equation*}
	which yields $\| \theta_{\ee_n} \|_{L^2((0,T)\times \RR^d)}\leq T^{1/2}\| \bar{\theta} \|_{L^2_x}$ for any $T>0$. Moreover, multiplying 
	\eqref{prop_smooth_data_transport_eqution} by 	$\theta$ and integrating in  $[0,t)\times \RR^d$, we achieve $\| \theta (t) \|_{L^2_x} = \| \bar{\theta} \|_{L^2_x}$ 
	for any $t\in (0,T)$, hence 
	\begin{equation*}
		\limsup_{n\rightarrow \infty} \| \theta_{\ee_n} \|_{L^\infty(0,T;L^2_x)} \leq 
		T^\frac{1}{2}\| \bar{\theta} \|_{L^2_x} =
		\| \theta \|_{L^2(0,T;L^2_x)}.
	\end{equation*}	
	Thus we can extract a subsequence (which we still call it $\theta_{\ee_n}$) such that $\theta_{\ee_n}$ strongly converges to $\theta$ in $L^2_{loc}(\RR_+\times\RR^d)$.
	We deduce that $\theta_{\ee_n}$ converges almost everywhere to $\theta$, up to a subsequence, and $\nu(\theta_{\ee_n})$ strongly converges to 
	$\nu(\theta)$ in $L^m_{loc}(\RR_+\times \RR^d)$, for every $1\leq m<\infty$, thanks to the Dominated Convergence Theorem. 
	Then $\nu(\theta_{\ee_n})\MM_{\ee_n}$ converges to $\nu(\theta)\MM$ in the distributional sense. 
	
	\noindent
	Summarizing all the previous considerations we finally conclude that $(\theta,\,u,\,\Pi)$ is a weak solution of 
	\eqref{Navier_Stokes_system} and it satisfies the inequalities given by \eqref{inequalities_statement_prop_smooth_dates}.
\end{proof}

\section{Proof of Theorem \ref{Main_Theorem}}

In this section we present the proof of Theorem \eqref{Main_Theorem}. Because of the low regularity of the initial temperature, by the dyadic partition we approximate our initial data and by Theorem \ref{Theorem_solutions_smooth_dates} we construct a sequence of approximate solutions. A step one, still using the mentioned Theorem, we observe that such solutions fulfill inequalities which are dependent only on the initial data. Therefore, using a compactness argument, we establish that the approximate solutions converge, up to a subsequence, and that the limit is the solution we are looking for. 

\begin{proof}[Proof of Theorem \ref{Main_Theorem}]
Recalling the Besov embedding $L^\infty_x\hookrightarrow \dot{B}_{\infty,\infty}^0$, we define 
\begin{equation*}
	\bar{\theta}_n :=\chi_n \sum_{|j|\leq n}\dot{\Delta}_j\bar{\theta}\quad\text{and}\quad 
	\bar{u}_n:=\sum_{|j|\leq n}\dot{\Delta}_j\bar{u}, \quad\text{for every}\quad n\in\NN,
\end{equation*}
where $\chi_n\leq 1$ is a cut-off function which has support on the ball $B(0,n)\subset \RR^d$. Thus $\bar{\theta}_n\in L^\infty_x\cap L^2_x$ and $\bar{u}_n	\in		\dot{B}_{p,r}^{d/p}	\cap	\dot{B}_{p,r}^{d/p-1+\ee}$, with $\ee<\min\{1/(2r), 1-1/r, 2(d/p -2 + 1/r)\}$.  
Then, by Theorem \ref{Theorem_solutions_smooth_dates}, there exists $(\theta_n, u_n, \Pi_n)$ weak solution of
\begin{equation*}
	\begin{cases}
		\partial_t\theta_n + \Div (\theta_n u_n)=0									& \RR_+ \times\RR^d,\\
		\partial_t u_n + u_n\cdot \nabla u_n -\Div (\nu(\theta_n)\nabla u_n) +\nabla\Pi_n=0	& \RR_+ \times\RR^d,\\
		\Div\, u_n = 0																& \RR_+ \times\RR^d,\\
		(\theta_n,\,u_n)_{t=0} = (\bar{\theta}_n,\,\bar{u}_n)						& \;\;\quad \quad\RR^d,
	\end{cases}
\end{equation*}
such that $\theta_n\in L^{\infty}(\RR_+\times\RR^d)$, $u_n		\in	L^{2r}_tL^{dr/(r-1)}_x$, $\nabla u_n \in L^{2r}_tL^{dr/(2r-1)}_x\cap L^{r}_tL^{dr/(2r-2)}_x	$
and moreover $\Pi_n\in L^{r}_t L^{\frac{dr}{2(r-1)}}_x$.
Furthermore the following inequalities are satisfied:	
\begin{align*}
	& 	\| \nabla u_n^h \|_{L^{2r}_t L^{ \frac{dr}{2(r-1)}}_x}+
		\| \nabla u_n^h \|_{L^{2r}_t L^{ \frac{dr}{2r-1}}_x}+ 
		\| u_n^h\|_{L^{2r}_t L^{\frac{dr}{r-1}}_x} 
				\leq C_1\eta\text{,}\\
	&	\| \nabla u_n^d \|_{L^{2r}_t L^{ \frac{dr}{2(r-1)}}_x}+
		\|\nabla u_n^d\|_{L^{2r}_t L^{\frac{dr}{2r-1}}_x}  + 
		\| u_n^d \|_{L^{2r}_t L^{\frac{dr}{r-1}}_x} 
				\leq C_2\| \bar{u}^d \|_{\dot{B}_{p,r}^{-1+\frac{d}{p}}} + C_3,\\
	&	\|\Pi_n\|_{L^{r}_t L^{\frac{dr}{2(r-1)}}_x} \leq C_4\eta,\quad
		\|\theta_n\|_{L^\infty(\RR_+\times\RR^d)}\leq C\|\bar{\theta}\|_{L^{\infty}_x}			
\end{align*}
for all $n\in\NN$ and for some positive constants $C_1$, $C_2$, $C_3$, $C_4$, $C_5$ and $C$.
Then there exists a subsequence (which we still denote by $( \,(\theta_n, u_n, \Pi_n)\,)_\NN$ ) and $(\theta,\, u,\,\Pi)$ in the same space of 
$(\theta_n, u_n, \Pi_n)$, such that
\begin{equation*}
\begin{array}{lll}
	 u_{n} 			\rightharpoonup  		 u	\quad	w-L^{2r}_t L^{\frac{dr}{r-1}}_x,
	&\nabla u_{n} 	\rightharpoonup  \nabla u 	\quad 	w-L^{2r}_t L^{\frac{dr}{2r-1}}_x, 
	&\nabla u_{n} 	\rightharpoonup  \nabla u	\quad	w-L^{2r}_t L^{\frac{dr}{2(r-1)}}_x,\\
	 \theta_{n} 	\overset{*}{\rightharpoonup} 	\theta	\quad	w^*-L^{\infty}_{t,x},
	&\Pi_{n}		\rightharpoonup 				\Pi		\quad	w-L^{r}_t L^{\frac{dr}{2(r-1)}}_x.&
\end{array}
\end{equation*}
Moreover, proceeding as in Theorem \ref{Theorem_solutions_smooth_dates}, $u_n$ strongly converges to $u$ in $L^{2r}_{loc, t} L^{dr/(r-1)}_x$, so that $\theta$ is weak 
solution of
\begin{equation}\label{transport_equation}
	\partial_t\theta + \Div (\theta u)=0	\quad\text{in}\quad	\RR_+ \times\RR^d		\quad\text{and}\quad 	
	\theta_{|t=0} = \bar{\theta}\quad\text{in} \quad\RR^d.
\end{equation}
Now, we claim that $\theta^2_n \overset{*}{\rightharpoonup} \theta^2$ in $L^\infty (\RR_+ \times\RR^d)$. Observing that $\|\theta^2\|_{L^\infty(\RR_+ \times\RR^d)}\leq 
C^2\|\bar{\theta}\|^2_{L^\infty_x}$, there exists $\omega\in L^\infty_{t,x}$ such that $\theta^2_n \overset{*}{\rightharpoonup} \omega$ 
in $L^\infty _{t,x}$, up to a subsequence.
Now, let us remark that $\theta_n^2$ is weak solution of
\begin{equation*}
		\partial_t\theta_n^2 + \Div (\theta_n^2 u_n)=0		\quad\text{in}\quad \RR_+ \times\RR^d\quad\text{and}\quad
		\theta^2_{n|t=0} = \bar{\theta}^2					\quad\text{in}\quad	 \RR^d,
\end{equation*}
then, passing through the limit as $n$ goes to $\infty$, we deduce that $\omega$ is weak solution of
\begin{equation*}
		\partial_t \omega + \Div (\omega u)=0				\quad\text{in}\quad \RR_+ \times\RR^d\quad\text{and}\quad
		\omega_{|t=0} = \bar{\theta}^2						\quad\text{in}\quad	 \RR^d.
\end{equation*}
Moreover, multiplying \eqref{transport_equation} by $\theta$, we get
\begin{equation*}
	\partial_t\theta^2 + \Div (\theta^2 u)=0			\quad\text{in}\quad \RR_+ \times\RR^d\quad\text{and}\quad
	\theta^2_{|t=0} = \bar{\theta}^2					\quad\text{in}\quad	 \RR^d,
\end{equation*}
which yields $\omega=\theta^2$, from the uniqueness of the transport equation.
Summarizing the previous considerations, we deduce that $\theta_n \rightarrow \theta$ $s-L^2_{loc}(\RR_+ \times\RR^d)$, so that $\theta_n$ converges to $\theta$ almost everywhere in $\RR_+ \times\RR^d$  up to a subsequence, thus $\nu(\theta_n)$ converges to $\nu(\theta)$ almost everywhere in $\RR_+ \times\RR^d$. We conclude that and $\nu(\theta_n)$ strongly converges to $\nu(\theta_n)$ in $L^m_{loc}(\RR_+ \times\RR^d)$, for every $m\in [1,\infty)$, thanks to the Dominated Convergence Theorem. Therefore, passing through the limit as $n$ goes to $\infty$, we deduce that
\begin{equation*}
	\Div (\nu(\theta_n)\nabla u_n) \rightarrow \Div (\nu(\theta)\nabla u),
\end{equation*}
in the distributional sense, which allows to conclude that $(\theta, u, \Pi)$ is a weak solution of \eqref{Navier_Stokes_system}.
\end{proof}

\begin{remark}
	If we replace the two first equations of system \eqref{Navier_Stokes_system} by
	\begin{equation*}
		\partial_t\theta + \Div\, (\theta u)+a\theta=0									\quad \text{in}\quad	 	\RR_+ \times\RR^d\quad \text{and}\quad
		\partial_t u + u\cdot \nabla u -\Div\, (\nu(\theta)\MM) +\nabla\Pi=a \theta e_d	 \quad \text{in}\quad		\RR_+ \times\RR^d,
	\end{equation*}
	where $e_d=\,^t(0,\dots,1)\in\RR^d$ and $a$ is a positive real constant, then  we can adapt our strategy in order to establish the existence of weak solutions for 
	such new system. In the case of the original system, a term  as $\theta e_d$ can be assumed only to be bounded both in time and space, hence it does not provide a time 
	integrability, which is necessary in order to achieve the existence result. However, adding the damping term $a \theta$ to the classical transport equation, and supposing 
	$\bar{\theta}$ to belongs to $L^{2d/(3r-2)}_x$, then 
	\begin{equation*}
		\|\theta(t)\|_{L^{\frac{dr}{3r-2}}_x}\leq 
		\|\bar{\theta}\|_{L^{\frac{dr}{3r-2}}_x}
		\exp\big\{-a\,t\big\}, 
	\end{equation*}
	for every $t\in\RR_+$. Thus $\theta$ belongs to $L^r_tL^{dr/(3r-2)}_x$ and we can proceed as in the previous proofs, obtaining a global weak solution 
	$(\theta,\,u,\, \Pi)$ which belongs to the space defined by Theorem $\ref{Main_Theorem}$. Moreover, increasing $\eta$ by
	\begin{equation*}
		\eta_2:=\Big( 
				\| \nu - 1 \|_\infty + \|\bar{u}^h\|_{\dot{B}_{p,r}^{-1+\frac{d}{p}}}
				+a\|\bar{\theta}\|_{L^{\frac{dr}{3r-2}}}
			\Big)
			\exp\Big\{ c_r \|\bar{u}^d\|_{\dot{B}_{p,r}^{-1+\frac{d}{p}}}^{4r} \Big\}, 
	\end{equation*}
	the solution $(\theta,\,u,\, \Pi)$ fulfills
	\begin{align*}
		& \| \nabla u^h \|_{L^{2r}_t L^{ \frac{dr}{2r-1}}_x}+ 
			\|\nabla u^h\|_{L^{r}_t L^{\frac{dr}{2(r-1)}}_x} +
			\| u^h \|_{L^{2r}_t L^{\frac{dr}{r-1}}_x} 
					\leq C_1\eta_2\text{,}\\
		&	\|\nabla u^d\|_{L^{2r}_t L^{\frac{dr}{2r-1}}_x}  + 
			\|\nabla u^d\|_{L^{r}_t L^{\frac{dr}{2(r-1)}}_x} 
			\| u^d \|_{L^{2r}_t L^{\frac{dr}{r-1}}_x} 
					\leq C_2\Big(\| \bar{u}^d \|_{\dot{B}_{p,r}^{-1+\frac{d}{p}}} 
					+a\|\bar{\theta}\|_{L^{\frac{dr}{3r-2}}}\Big) + C_3,\\
		&	\|\Pi\|_{L^{r}_t L^{\frac{dr}{2(r-1)}}_x} 	\leq 
					C_4\eta_2,
	\end{align*}
	for some positive constants $C_1$, $C_2$, $C_3$ and $C_4$.
\end{remark}

\section{The general case: smooth initial data}
As preliminary, before starting the proof of the main Theorem, we enunciate three fundamental Lemma concerning the regularizing effects of the
heat kernel, which will be useful. We recall that $\Bb$ and $\Cc$ are defined by
\begin{equation*}
	\Bb  f(t):= \int_0^t \nabla e^{(t-s)\Delta} f(s)\dd s,\quad
	\Cc  f(t):= \int_0^t 		 e^{(t-s)\Delta} f(s)\dd s.
\end{equation*}
\begin{lemma}\label{Lemma4}
	Let us assume that $p$, $p_3$, $r$, $\alpha$, $\gamma_1$, $\gamma_2$ fulfill the hypotheses of Theorem \ref{Main_Theorem2} and let $\ee$ 
	be a non-negative constant bounded by $\min\{ 1/r, 1-1/r, d/p-1\}$. If $t^\alpha f(t)$ belongs to $L^{2r/(1-\ee r)}(0,T;L^p_x)$ then 
	$t^{\gamma_1}\Cc f(t)$ belongs to  $L^{2r/(1-\ee r)}(0,T;L^{p_3}_x)$ and there exists a positive constant $C$ such that
	\begin{equation*}
		\|t^{\gamma_1}	\Cc f(t)\|_{L^{\frac{2r}{1-\ee r}}(0,T; L^{p_3}_x)}
		\leq	C
		\|t^\alpha			f(t)\|_{L^{\frac{2r}{1-\ee r}}(0,T; L^{p  }_x)}.
	\end{equation*}
	Moreover, if $\ee$ is null then $t^{\gamma_2}\Cc f(t)$ belongs to  $L^{\infty}(0,T;L^{p_3}_x)$ and
	\begin{equation*}
		\|t^{\gamma_2}	\Cc f(t)\|_{L^{\infty				}(0,T; L^{p_3}_x)}
		\leq	C
		\|t^\alpha			f(t)\|_{L^{\frac{2r}{1-\ee r}	}(0,T; L^{p  }_x)}.
	\end{equation*}
\end{lemma}

\begin{lemma}\label{Lemma5}
	Let us assume that $p$, $p_2$, $r$, $\alpha$, $\beta$ fulfill the hypotheses of Theorem \ref{Main_Theorem2} and let $\ee$ be a non-negative 
	constant bounded by $\min\{ 1/r, 1-1/r, d/p-1\}$. If $t^\alpha f(t)$ belongs to $L^{2r/(1-\ee r)}(0,T;L^p_x)$ then $t^\beta\Bb f(t)$ 
	belongs to  $L^{2r/(1-\ee r)}(0,T;L^{p_2}_x)$ and there exists a positive constant $C$ such that
	\begin{equation*}
		\|t^\beta	\Bb f(t)\|_{L^{\frac{2r}{1-\ee r}}(0,T; L^{p_2}_x)}
		\leq	C
		\|t^\alpha		f(t)\|_{L^{\frac{2r}{1-\ee r}}(0,T; L^{p  }_x)}.
	\end{equation*}
\end{lemma}
\begin{lemma}\label{Lemma6}
	Let us assume that $p$, $p_2$, $r$, $\alpha$, $\beta$, $\gamma_1$, $\gamma_2$ fulfill the hypotheses of Theorem \ref{Main_Theorem2} and 
	let $\ee$ be a non-negative constant bounded by $\min\{ 1/r, 1-1/r, d/p-1\}$. 
	If $t^{\beta}f$ belongs to $L^{2r/(1-\ee r)}(0,T; L^{p_2}_x)$ then $t^{\gamma_1}\Bb f(t)$ belongs to  $L^{2r/(1-\ee r)}(0,T;L^{p_3}_x)$
	\begin{equation}\label{Lemma6_inequality}
		\|t^{\gamma_1	}\Bb 	f	\|_{L^{\frac{2r}{1-\ee r}	}(0,T;L^{p_3}_x)}
		\leq C
		\|t^{\beta		}		f	\|_{L^{\frac{2r}{1-\ee r}	}(0,T;L^{p_2}_x)}.
	\end{equation}
	Furthermore, if $\ee=0$ then there exists a positive $C$ such that
	\begin{equation}\label{Lemma6_inequality2}
		\|t^{\gamma_2	}\Bb 	f	\|_{L^{\infty	}(0,T;L^{p_3}_x)}
		\leq C
		\|t^{\beta		}		f	\|_{L^{\frac{2r}{1-\ee r}	}(0,T;L^{p_2}_x)}.
	\end{equation}
\end{lemma}
The proofs of these lemmas are a direct consequence of Remark \ref{remark2.1}. We perform the one of Lemma \ref{Lemma6}, while the others can 
be achieved thanks to a similar procedure.
\begin{proof}[Proof of Lemma \ref{Lemma6}]
	We begin controlling the $L^{2r/(1-\ee r)}(0,T; L^{p_3}_x)$-norm. First Remark \eqref{remark2.1} yields
	\begin{equation*}
		\|t^{\gamma_1	}\Bb 	f(t)	\|_{L^{p_3}_x}
		\leq  C
		\int_{0}^t 
		\frac{t^{\gamma_1}}{| t-s |^{\frac{d}{2}\big(	\frac{1}{p_2}-\frac{1}{p_3}	\big)+\frac{1}{2}}}\|f(s)	\|_{L^{p_2}_x}
		\dd s= C
		\int_{0}^1 
		\frac{t^{\gamma_1-\frac{d}{2}\big(	\frac{1}{p_2}-\frac{1}{p_3}	\big)+\frac{1}{2}-\beta}}
		{| 1-\tau	|^{\frac{d}{2}\big(	\frac{1}{p_2}-\frac{1}{p_3}	\big)+\frac{1}{2}}\tau^{\beta}}
		F(t \tau)\dd \tau,
	\end{equation*}
	where $F(s):= s^{\beta}\|f(s)	\|_{L^{p_2}_x} $. Now, since 
	$\gamma_1-d(1/p_2-1/p_3)/2+1/2-\beta$ is null, we have
	\begin{align*}
		\|t^{\gamma_1	}\Bb 	f	\|_{L^{\frac{2r}{1-\ee r}	}(0,T;L^{p_3}_x)}&\leq C
		\int_{0}^1 
		\frac{1}{| 1-\tau	|^{\frac{d}{2}\big(	\frac{1}{p_2}-\frac{1}{p_3}	\big)+\frac{1}{2}}\tau^{\beta}}
		\|F(t \tau)\|_{L^{\frac{2r}{1-\ee r}	}_t(0,T;L^{p_2}_x)}\dd \tau\\&\leq C
		\int_{0}^1 
		\frac{1}{| 1-\tau	|^{\frac{d}{2}\big(	\frac{1}{p_2}-\frac{1}{p_3}	\big)+\frac{1}{2}}\tau^{\beta+\frac{1}{2r}-\frac{\ee}{2}}}
		\dd \tau
		\|F\|_{L^{\frac{2r}{1-\ee r}	}(0,T;L^{p_2}_x)},
	\end{align*}
	thanks to the Minkowski inequality. Thus \eqref{Lemma6_inequality} is true, since $\beta+1/(2r)-\ee/2<1$ and moreover 
	$d(1/p_2-1/p_3)/2+1/2=2/3-d/(6p)+1/2<1-1/(2r)<1$. Finally, observing that
	\begin{align*}
		\|t^{\gamma_2	}\Bb 	f(t)	\|_{L^{p_3}_x}
		&\leq C
		\int_{0}^t 
		\frac{t^{\gamma_2}}
		{| t-s |^{\frac{d}{2}\big(	\frac{1}{p_2}-\frac{1}{p_3}	\big)+\frac{1}{2}}}\|f(s)	\|_{L^{p_2}_x}
		\dd s\\
		&\leq C
		\Big(
		\int_{0}^t 
		\Big|
		\frac{t^{\gamma_2}}
		{| t-s |^{\frac{d}{2}\big(	\frac{1}{p_2}-\frac{1}{p_3}	\big)+\frac{1}{2}}s^{\beta}}
		\Big|^{(2r)'}
		\dd s
		\Big)^{1-\frac{1}{2r}}
		\|F\|_{L^{2r	}(0,T;L^{p_2}_x)}
	\end{align*}
	we obtain
	\begin{equation*}
		\|t^{\gamma_2	}\Bb 	f(t)	\|_{L^{p_3}_x}
		\leq C
		\Big(
		\int_{0}^1 
		\Big|
		\frac{1}{| 1-\tau	|^{\frac{d}{2}\big(	\frac{1}{p_2}-\frac{1}{p_3}	\big)+\frac{1}{2}}\tau^{\beta}}
		\Big|^{(2r)'}
		\dd \tau
		\Big)^{1-\frac{1}{2r}}
		\|F\|_{L^{2r	}(0,T;L^{p_2}_x)}
	\end{equation*}
	by the change of variable $s=t\tau$, since $(2r)'\{\gamma_2-d(1/p_2-1/p_3)/2-1/2-\beta\}+1$ is null. Hence 
	\eqref{Lemma6_inequality2} turns out from $\{d(	1/{p_2}-1/{p_3})/2+1/2\}(2r)'<1$ and $\beta(2r)'<1$.
\end{proof}
We present the statement of a modified version of the Maximal Regularity Theorem, whose proof can be found in \cite{MR3056619}.
\begin{theorem}\label{Maximal_regularity_Thm_weight_time}
	Let $T\in ]0,\infty]$, $1<\bar{r},q<\infty$ and  $\alpha\in (0,1-1/\bar{r})$. Let the operator 
	$\Aa$ be defined as in Theorem \ref{Maximal_regularity_theorem}. Suppose that 
	$t^\alpha f(t)$ 
	belongs to $L^{\bar{r}}(0,T;L^q_x)$. Then $t^\alpha \Aa f(t)$ belongs to $L^{\bar{r}}(0,T;L^q_x)$ and there 
	exists $C>0$ such that
	\begin{equation*}
		\|t^\alpha \Aa f(t)\|_{L^{\bar{r}}(0,T;L^q_x)}\leq C \| t^\alpha f(t)\|_{L^{\bar{r}}(0,T;L^q_x)}.
	\end{equation*}	 
\end{theorem}
As last part of this preliminaries, we have the following corollary, which will be useful in order to control the pressure $\Pi$.
\begin{cor}\label{cor_maxthm_weight_time}
	Let $p\in (1,d)$, $\bar{r}\in(1,\infty)$ and $\alpha\in (0,1-1/\bar{r})$. If $t^{\alpha}f$ belongs to $L^{\bar{r}}(0,T;L^p_x)$ then $t^{\alpha}\Bb f$ belongs to 
	$L^{\bar{r}}(0,T; L^{p^*}_x$ and there exists a positive constant $C$ (not dependent by $f$) such that 
	\begin{equation*}
		\|t^{\alpha} \Bb f(t)\|_{L^{2r}(0,T; L^{p^*}_x)}
		\leq C
		\|t^{\alpha}     f(t)\|_{L^{2r}(0,T; L^{p  }_x)}.
	\end{equation*}	 
\end{cor}
\begin{proof}
	It is sufficient to observe that $\Bb f(t)$ reads as follows:
	\begin{equation*}
		\Bb f(t) =
		-(\sqrt{-\Delta})^{-1} R \int_0^t \Delta e^{(t-s)\Delta }f(s) \dd s =
		-(\sqrt{-\Delta})^{-1} R \Aa f(t).
	\end{equation*}
	Recalling that $R$ is a bounded operator from $L^q_x$ to itself for any $q\in (1,\infty)$ and $(\sqrt{-\Delta})^{-1}$ from $L^p_x$ into 
	$L^{p^*}_x$, the lemma is a direct consequence of Theorem \ref{Maximal_regularity_Thm_weight_time}.
\end{proof}

\begin{prop}\label{Proposition_solutions_smooth_data_general_case}
	Let $p,\,r,\,p_2,\,p_3$ be as in Theorem \ref{Main_Theorem2}. 
	Suppose that $\bar{\theta}$ belongs to $L^\infty_x$ and $\bar{u}$ belongs to $\BB_{p,r}^{d/p-1}$. If the smallness condition \eqref{smallness_condition_general_case} holds, 
	then there exists a global weak solution $(\theta,\,u,\,\Pi)$ of \eqref{Navier_Stokes_system_eps} such that it belongs to the functional framework defined by 
	Theorem \ref{Main_Theorem2} and moreover it 
	satisfies
	\begin{equation}\label{inequalities_statement_prop_smooth_dates_general_data}
	\begin{aligned}
		 &	\|t^{\beta}		\nabla u^h\|_{L^{2r}_t L^{p_2}_x}+
		 	\|t^{\alpha}	\nabla u^h\|_{L^{2r}_t L^{p^*}_x}+
			\|t^{\gamma_1} u^h \|_{L^{2r}_t L^{p_3}_x}+ 
			\|t^{\gamma_2} u^h \|_{L^{\infty}_t L^{p_3}_x}			
			\leq C_1\eta,\\&
			\|t^{\beta}		\nabla u^d\|_{L^{2r}_t L^{p_2}_x}+ 
			\|t^{\alpha}	\nabla u^d\|_{L^{2r}_t L^{p^*}_x}+ 
			\|t^{\gamma_1} u^d \|_{L^{2r}_t L^{p_3}_x}+ 
			\|t^{\gamma_2} u^d \|_{L^{\infty}_t L^{p_3}_x}		
			\leq C_2\| \bar{u}^d \|_{\BB_{p,r}^{\frac{d}{p}-1}} + C_3\\&
			\|t^\alpha\Pi\|_{L^{2r}_t L^{p^*}_x} 
					\leq 
					C_4\eta,\quad 
			\| \theta \|_{L^\infty_{t,x}}\leq \| \bar{\theta} \|_{L^\infty_x}.
	\end{aligned}
	\end{equation}
	for some positive constants $C_1$, $C_2$ and $C_3$.
\end{prop}
\begin{proof}
	We proceed as in the proof of Proposition \ref{Proposition_solutions_smooth_data}, considering the sequence of solutions for systems 
	\eqref{Transport_equation_navier_stokes_approximate} and \eqref{Navier_Stokes_system_approximate}. We claim that such solutions belong to 
	the same space defined in Theorem \ref{Main_Theorem2} and moreover that:
	\begin{equation}
	\begin{aligned}
		  &	\|	t^{\beta	}	\nabla 	u^{h}_n		\|_{L^{2r		}_t L^{p_2}_x	}+
			\|	t^{\gamma_1	} 			u^{h}_n 	\|_{L^{2r		}_t L^{p_3}_x	}+ 
			\|	t^{\gamma_2	} 			u^{h}_n		\|_{L^{\infty	}_t L^{p_3}_x	}			
			\leq C_1\eta,\\&
			\|	t^{\beta	}	\nabla 	u^{d}_n		\|_{L^{2r		}_t L^{p_2}_x	}+ 
			\|	t^{\gamma_1	} 			u^{d}_n 	\|_{L^{2r		}_t L^{p_3}_x	}+ 
			\|	t^{\gamma_2	} 			u^{d}_n 	\|_{L^{\infty	}_t L^{p_3}_x	}		
			\leq C_2\| \bar{u}^d \|_{\BB_{p,r}^{\frac{d}{p}-1}} + C_3,
	\end{aligned}
	\end{equation}
	for some suitable positive constants $C_1$, $C_2$ and $C_3$, and for any positive integer $n$. 
	
	\vspace{0.2cm}
	\textsc{Step 1: Estimates.} First,  the maximal principle for parabolic equation implies that
	$\|\theta_n\|_{L^\infty_{t,x}}$ is bounded by $\|\bar{\theta}\|_{L^\infty_x}$. Now, we want to prove by 
	induction that 
	\begin{equation}\label{prop_smooth-general_induction}
	\begin{aligned}
		&	\|t^{\beta}\nabla u^{h}_{n}\|_{L^{2r}_t L^{p_2}_x}+
			\|t^{\gamma_1} u^{h}_{n} \|_{L^{2r}_t L^{p_3}_x}+ 
			\|t^{\gamma_2} u^{h}_{n} \|_{L^{\infty}_t L^{p_3}_x}			
			\leq \frac{C_1}{2}\tilde{\eta}
			\leq \frac{C_1}{2}\eta,\\&
			\|t^{\beta}\nabla u^{d}_{n}\|_{L^{2r}_t L^{p_2}_x}+ 
			\|t^{\gamma_1} u^{d}_{n} \|_{L^{2r}_t L^{p_3}_x}+ 
			\|t^{\gamma_2} u^{d}_{n} \|_{L^{\infty}_t L^{p_3}_x}		
			\leq \frac{C_2}{2}\| \bar{u}^d \|_{\BB_{p,r}^{\frac{d}{p}-1}} + \frac{C_3}{2},
	\end{aligned}
	\end{equation}
	for some positive constant $C_1$, $C_2$ and $C_3$, where $\tilde{\eta}$ is defined by 
	\begin{equation*}
		\tilde{\eta}:= (\|\bar{u}^h\|_{\BB_{p,r}^{\frac{d}{p}-1}}+\|\bar{\theta}\|_{L^\infty_x}+\|	\nu	-1	\|_{\infty})
		\exp\big\{\frac{c_r}{2}\| \bar{u}^d \|_{\BB_{p,r}^{\frac{d}{p}-1}}^{2r}\big\}<\eta.
	\end{equation*}
	We begin with the horizontal component $u_n^{h}$. 
	Let $\lambda$ be a positive real number, and let $u_{n+1,\lambda}$, $\nabla u_{n+1, \lambda}$ and $\Pi_{n+1, \lambda}$ be defined by
	\begin{equation}\label{def_ulambda2}
		(u_{n+1, \lambda},\,\nabla u_{n+1, \lambda},\,\Pi_{n+1, \lambda} )(t):= h_{n,\lambda }(0,t)(u_{n+1},\,\nabla u_{n+1},\,\Pi_{n+1} )(t),
	\end{equation}
	where, for all $0 \leq s<t<\infty$, 
	\begin{equation}\label{def_h2}
		h_{n, \lambda}(s,t):=  
		\exp
		\big\{ 
			-\lambda \int_s^t t^{2r\gamma_1}\|u_n^d(\tau)\|^{2r}_{L_x^{p_3}}\dd \tau\,
			-\lambda \int_s^t t^{2r\beta	}\|\nabla u_n^d(\tau)\|^{2r}_{L_x^{p_2}}\dd \tau\,
		\big\}.
	\end{equation}
	We decompose $u_{n+1,\lambda}$ as in \eqref{def_un+1lambda}, 
	$u_{n+1,\lambda}=u_L + F^{1}_{n+1,\lambda}+F^{2}_{n+1,\lambda}+F^{3}_{n+1,\lambda} $, the first estimate is given by Theorem 
	\ref{Characterization_of_hom_Besov_spaces} and Theorem \ref{Theorem_embedding_Besov}:
	\begin{equation}\label{Prop_smooth_data_est1}
	\begin{aligned}
		\|	t^{\beta	 }	\nabla 	u^{h}_{L,\lambda}		\|_{L^{2r	 }_t 	L^{p_2}_x}+
		\|	t^{\gamma_1} 			u^{h}_{L,\lambda} 		\|_{L^{2r	 }_t 	L^{p_3}_x}&+ 
		\|	t^{\gamma_2} 			u^{h}_{L,\lambda} 		\|_{L^{\infty}_t 	L^{p_3}_x}
		\leq 
		\|	t^{\beta	 }	\nabla 	u^{h}_{L,\lambda}		\|_{L^{2r	 }_t 	L^{p_2}_x}+\\&+
		\|	t^{\gamma_1} 			u^{h}_{L,\lambda} 		\|_{L^{\infty}_t	L^{p_3}_x}+ 
		\|	t^{\gamma_2} 			u^{h}_{L,\lambda} 		\|_{L^{\infty}_t 	L^{p_3}_x}
		\leq 
		C\| \bar{u}^h \|_{\BB_{p,r}^{\frac{d}{p}-1}},
	\end{aligned}
	\end{equation}
	for a positive constant $C$. Moreover, recalling the definition \eqref{def_gn} of $g_{n+1}$, we get
	\begin{equation}\label{Prop_smooth_data_est2}
	\begin{aligned}
		\|	t^{\beta	 }	\nabla 	F^{1,h}_{n+1,\lambda}	 \|_{L^{2r	 }_t	L^{p_2}_x} +
		\|	t^{\gamma_1} 			F^{1,h}_{n+1,\lambda}	&\|_{L^{2r	 }_t	L^{p_3}_x} + 
		\|	t^{\gamma_2} 	   		F^{1,h}_{n+1,\lambda}    \|_{L^{\infty}_t 	L^{p_3}_x}\leq 
		C\big\{
			\frac{1}{	\lambda^{\frac{1}{2r}}	} 
			\|	t^{\beta	 }	\nabla 	u_{n+1,\lambda}^h	\|_{L^{2r	 }_t 	L^{p_2}_x} +\\&+
			\|	t^{\gamma_2} 			u_{n		  }^h	\|_{L^{\infty}_t 	L^{p_3}_x}
			\|	t^{\beta	 }	\nabla 	u_{n          }^h	\|_{L^{2r	 }_t 	L^{p_2}_x} +
			\frac{1}{\,\lambda^{\frac{1}{2r}}} 	
			\|	t^{\gamma_1}			u_{n+1,\lambda}^h	\|_{L^{2r	 }_t 	L^{p_3}_x}
		\big\}.
	\end{aligned}
	\end{equation}
	thanks to Lemma \ref{Lemma4}, Lemma \ref{Lemma5}, Lemma \ref{Lemma_A.3} and Lemma \ref{Lemma_A.4}. Moreover,
	\begin{equation}\label{Prop_smooth_data_est3}
	\begin{aligned}
		\|	t^{\gamma_1} 			F^{2,h}_{n+1,\lambda}	 \|_{L^{2r	 }_t	L^{p_3}_x} + 
		\|	t^{\gamma_2} 	   		F^{2,h}_{n+1,\lambda}   \|_{L^{\infty}_t 	L^{p_3}_x} +
		\|	t^{\gamma_1} 			F^{3,h}_{n+1,\lambda}	 \|_{L^{2r	 }_t	L^{p_3}_x} + 
		\|	t^{\gamma_2} 	   		F^{3,h}_{n+1,\lambda}   \|_{L^{\infty}_t 	L^{p_3}_x}\leq \\
		\leq
		C
		\|	t^{\beta   }(\nu(\theta_{n+1})-1)	\MM_n		\|_{L^{2r	 }_t	L^{p_2}_x}\leq
		C
		\|	\nu	-1	\|_{\infty}
		\|	t^{\beta   }				\nabla 	u_n			\|_{L^{2r	 }_t	L^{p_2}_x}
	\end{aligned}
	\end{equation}		
	by Lemma \ref{Lemma5} and Lemma. Finally, Theorem \ref{Maximal_regularity_Thm_weight_time} 
	yields
	\begin{equation}\label{Prop_smooth_data_est4}
	\begin{aligned}
		\|	t^{\beta} 	\nabla		F^{2,h}_{n+1,\lambda}	 \|_{L^{2r	 }_t	L^{p_2}_x} + 
		\|	t^{\beta} 	\nabla 		F^{3,h}_{n+1,\lambda}   \|_{L^{2r	 }_t 	L^{p_2}_x}\leq 
		C
		\|	\nu	-1	\|_{\infty}
		\|	t^{\beta   }				\nabla 	u_n			\|_{L^{2r	 }_t	L^{p_2}_x}.
	\end{aligned}
	\end{equation}
	Summarizing \eqref{Prop_smooth_data_est1}, \eqref{Prop_smooth_data_est2}, \eqref{Prop_smooth_data_est3} and 
	\eqref{Prop_smooth_data_est4}, we deduce that
	\begin{equation}\label{prop_smooth-general_case_horiz_est1}
	\begin{aligned}
			\|t^{\beta}\nabla u^{h}_{n+1,\,\lambda}\|_{L^{2r}_t L^{p_2}_x}+
			\|t^{\gamma_1} u^{h}_{n+1,\,\lambda}\|_{L^{2r}_t L^{p_3}_x}+ 
			\|t^{\gamma_2} u^{h}_{n+1,\,\lambda} \|_{L^{\infty}_t L^{p_3}_x}			
			&\leq \\
			\leq
			C
			\Big\{
			\| \bar{u}^h \|_{\BB_{p,r}^{\frac{d}{p}-1}}+
			\frac{1}{	\lambda^{\frac{1}{2r}}	} 
			\|	t^{\beta	 }	\nabla 	u_{n+1,\lambda}^h	\|_{L^{2r	 }_t 	L^{p_2}_x} +
			\|	t^{\gamma_2} 			u_{n		  }^h	\|_{L^{\infty}_t 	L^{p_3}_x}&
			\|	t^{\beta	 }	\nabla 	u_{n          }^h	\|_{L^{2r	 }_t 	L^{p_2}_x} +\\+
			\frac{1}{\,\lambda^{\frac{1}{2r}}} 	
			\|	t^{\gamma_1}			u_{n+1,\lambda}^h	\|_{L^{2r	 }_t 	L^{p_3}_x} &+
			\|	\nu	-1	\|_{\infty}
			\|	t^{\beta   }				\nabla 	u_n			\|_{L^{2r	 }_t	L^{p_2}_x}
			\Big\}
	\end{aligned}
	\end{equation}
	for a suitable positive constant $C$. Setting $\lambda := (2C)^{2r} $, we can absorb the terms with index $n+1$ on the right-hand side 
	by the the left-hand side, hence there exists a positive constant $\tilde{C}$ such that  
	\begin{align*}
			\|t^{\beta}\nabla u^{h}_{n+1,\,\lambda}\|_{L^{2r}_t L^{p_2}_x}+&
			\|t^{\gamma_1} u^{h}_{n+1,\,\lambda}\|_{L^{2r}_t L^{p_3}_x}+ 
			\|t^{\gamma_2} u^{h}_{n+1,\,\lambda} \|_{L^{\infty}_t L^{p_3}_x}			
			\leq \\
			&\leq
			\tilde{C}
			\Big\{
			\| \bar{u}^h \|_{\BB_{p,r}^{\frac{d}{p}-1}}+
			\frac{C_1^2}{4}\tilde{\eta}^2+
			\|	\nu	-1	\|_{\infty}
			(\frac{C_1}{2}\tilde{\eta} + 
				\frac{C_2}{2}
				\| \bar{u}^d \|_{\BB_{p,r}^{\frac{d}{p}-1}}
				+\frac{C_3}{2})
			\Big\}.
	\end{align*}
	Then we deduce that
	\begin{align*}
			&\|t^{\beta}\nabla u^{h}_{n+1}\|_{L^{2r}_t L^{p_2}_x}+
			\|t^{\gamma_1} u^{h}_{n+1}\|_{L^{2r}_t L^{p_3}_x}+ 
			\|t^{\gamma_2} u^{h}_{n+1} \|_{L^{\infty}_t L^{p_3}_x}			
			\leq \\
			&\leq
			\tilde{C}
			\sup_{t\in (0,\infty)}h_{n,\lambda}(0,t)^{-1}
			\Big\{
			\| \bar{u}^h \|_{\BB_{p,r}^{\frac{d}{p}-1}}+
			\frac{C_1^2}{4}\tilde{\eta}^2+
			\|	\nu	-1	\|_{\infty}
			(\bar{C}_1\tilde{\eta} + 
				\frac{C_2}{2}
				\| \bar{u}^d \|_{\BB_{p,r}^{\frac{d}{p}-1}}
				+\frac{C_3}{2})
			\Big\}\\
			&\leq 
			\tilde{C}
			\exp\big\{
				(2C)^{2r}(
				\frac{C_2}{2}
				\| \bar{u}^d \|_{\BB_{p,r}^{\frac{d}{p}-1}}
				+\frac{C_3}{2})^{2r}
			\big\}
			\Big\{
				1+ (\frac{C_1^2}{4} +\frac{C_1}{2})\tilde{\eta} + \frac{C_2}{2} +\frac{C_3}{2}
			\Big\}\tilde{\eta}.
	\end{align*}
	Imposing $C_1$ big enough and $\tilde{\eta}$ small enough in order to have
	\begin{equation*}
		\tilde{C}
			\exp\big\{
				(2C)^{2r}(
				\frac{C_2}{2}
				\| \bar{u}^d \|_{\BB_{p,r}^{\frac{d}{p}-1}}
				+\frac{C_3}{2})^{2r}
			\big\}
			\Big\{
				1+ (\frac{C_1^2}{4} +\frac{C_1}{2})\tilde{\eta} + \frac{C_2}{2} +\frac{C_3}{2}
			\Big\}\leq \frac{C_1}{2} \tilde{\eta},
	\end{equation*}
	we finally deduce that the first inequality of \eqref{prop_smooth-general_induction} is true for any positive integer $n$.
	Now, let us handle the vertical component $u_n^d$. Proceeding as in the proof of \eqref{prop_smooth-general_case_horiz_est1}, 
	we obtain that the following inequality is satisfied:
	\begin{align*}
			\|t^{\beta}\nabla 	u^{d}_{n+1}\|_{L^{2r}_t L^{p_2}_x}&+
			\|t^{\gamma_1} 		u^{d}_{n+1}\|_{L^{2r}_t L^{p_3}_x}+ 
			\|t^{\gamma_2} 		u^{d}_{n+1} \|_{L^{\infty}_t L^{p_3}_x}			
			\leq \\
			&\leq
			C
			\Big\{
			\| \bar{u}^d \|_{\BB_{p,r}^{\frac{d}{p}-1}}+
			\|	t^{\alpha	 }	g_{n+1}		\|_{L^{2r	 }_t 	L^{p}_x}+
			\|	\nu	-1	\|_{\infty}
			\|	t^{\beta   }	\nabla 	u_n			\|_{L^{2r	 }_t	L^{p_2}_x}
			\Big\},
	\end{align*}
	for a suitable positive constant $C$, where $g_{n+1}$ is defined by \eqref{def_gn}. Recalling that 
	$\alpha=\beta+\gamma_1$ and $1/p=1/p_2+1/p_3$ we get
	\begin{align*}
			\|t^{\beta}\nabla 	u^{d}_{n+1}\|_{L^{2r}_t L^{p_2}_x}&+
			\|t^{\gamma_1} 		u^{d}_{n+1}\|_{L^{2r}_t L^{p_3}_x}+ 
			\|t^{\gamma_2} 		u^{d}_{n+1} \|_{L^{\infty}_t L^{p_3}_x}			
			\leq 
			C
			\Big\{
			\| \bar{u}^d \|_{\BB_{p,r}^{\frac{d}{p}-1}}+\\+
			\|	t^{\gamma_2	 }			u_n^h		&\|_{L^{\infty	 }_t 	L^{p_3}_x}
			\|	t^{\beta	 }\nabla	u_n^h		\|_{L^{2r		 }_t 	L^{p_2}_x}+
			\|	t^{\gamma_2	 }			u_{n+1}^h	\|_{L^{\infty	 }_t 	L^{p_3}_x}
			\|	t^{\beta	 }\nabla	u_n^d		\|_{L^{2r		 }_t 	L^{p_2}_x}+	\\&+		
			\|	t^{\gamma_2	 }			u_n^h		\|_{L^{\infty	 }_t 	L^{p_3}_x}
			\|	t^{\beta	 }\nabla	u_{n+1}^d	\|_{L^{2r		 }_t 	L^{p_2}_x}+	
			\|	\nu	-1	\|_{\infty}
			\|	t^{\beta   }	\nabla 	u_n			\|_{L^{2r	 }_t	L^{p_2}_x}
			\Big\},
	\end{align*}
	which yields that
	\begin{align*}
			\|t^{\beta}\nabla 	u^{d}_{n+1}\|_{L^{2r}_t L^{p_2}_x}&+
			\|t^{\gamma_1} 		u^{d}_{n+1}\|_{L^{2r}_t L^{p_3}_x}+ 
			\|t^{\gamma_2} 		u^{d}_{n+1} \|_{L^{\infty}_t L^{p_3}_x}	
			\leq\\		
			&\leq 
			C
			(1+\frac{C_1C_2}{4}\tilde{\eta}
			)\| \bar{u}^d \|_{\BB_{p,r}^{\frac{d}{p}-1}}+
			C
			( 
			\frac{C_1C_3}{4} + \frac{C_1^2}{4}\tilde{\eta} + 
			\|\nu-1\|_{\infty} 
			(
				\frac{C_1}{2} +\frac{C_2}{2} 
			) 
		)\tilde{\eta}.
	\end{align*}
	Hence the second inequality of \eqref{prop_smooth-general_induction} is true for any positive integer $n$ if we assume $\bar{C}_2$ big 
	enough and $\eta$ small enough in order to have
	\begin{equation*}
		C(1+\frac{C_1C_2}{4} \tilde{\eta})<\frac{C_2}{2}
		\quad
		\text{and}
		\quad
		C(\frac{C_1C_3}{2} +\frac{C_1^2}{4}\eta +\eta(\frac{C_1}{2}+\frac{C_2}{2}))\eta\leq\frac{C_3}{2}.
	\end{equation*}
	Proceeding again by induction, we claim that
	\begin{equation}\label{ind_talphanabla}
		\|t^{\alpha}\nabla u^h_n\|_{L^{2r}_t L^{p^*}_x}\leq \frac{C_1}{2}\eta
		\quad\text{and}\quad
		\|t^{\alpha}\nabla u^d_n\|_{L^{2r}_t L^{p^*}_x}\leq \frac{C_2}{2}\| \bar{u}^d	\|_{\BB_{p,r}^{\frac{d}{p}-1}}+\frac{C_3}{2},
	\end{equation}	 
	for any positive integer $n$. First, we remark that $\nabla u_L$ can be rewritten as $\nabla u_L = -(\sqrt{-\Delta})^{-1}R\Delta u_L $. Hence, 
	recalling that  $(\sqrt{-\Delta})^{-1}$ is a bounded operator from $L^p_x$ into $L^{p^*}_x$ and $R$ is a bounded operator from $L^q_x$ into itself, 
	for any $q\in (1,\infty)$, there exists a positive constant $C$ such that
	\begin{equation}
		\|t^{\alpha}\nabla u_L\|_{L^{2r}_t L^{p^*}_x}\leq C
		\|t^{\alpha}\Delta u_L\|_{L^{2r}_t L^{p}_x}\leq 
		\|\bar{u}\|_{\BB_{p,r}^{\frac{d}{p}-1}},
	\end{equation}
	thanks to Theorem \ref{Characterization_of_hom_Besov_spaces}. Moreover Theorem \ref{Maximal_regularity_Thm_weight_time} and 
	Corollary \ref{cor_maxthm_weight_time} imply
	\begin{equation*}
		\|t^{\alpha}(\nabla F^{2}_{n+1}+\nabla F^{3}_{n+1})\|_{L^{2r}_t L^{p^*}_x}
		\leq 
		C\eta
		\|t^{\alpha}\nabla u_{n}\|_{L^{2r}_t L^{p^*}_x},\,
		\|t^{\alpha}\nabla F^{1}_{n+1}\|_{L^{2r}_t L^{p^*}_x}\leq C \|t^{\alpha} g_{n+1} \|_{L^{2r}_t L^{p}_x}
		\leq C\eta.
	\end{equation*}
	Assuming $\eta$ small enough we get that \eqref{ind_talphanabla} is true for any $n\in\NN$.
	Finally, recalling that $\Pi_{n+1}$ is determined by
	\begin{equation*}
		\Pi_{n+1}=(-\Delta)^{-1}R\cdot g_{n+1}-R\cdot R\cdot\{(\nu(\theta_{n+1})-1)\nabla u_n\},
	\end{equation*}
	we get
	\begin{align*}
		\|t^{\alpha}		\Pi_{n+1}	\|_{L^{2r}_tL^{p^*	}_x}
		&\leq
		C\big\{
		\|t^{\alpha}		g_{n+1}		\|_{L^{2r}_tL^{p	}_x}+
		\|\nu-1							\|_{\infty}
		\|t^{\beta}	\nabla 	u_n			\|_{L^{2r}_tL^{p_2	}_x}
		\big\}
		\leq
		C_4\eta,	
	\end{align*}
	for a suitable positive constant $C_4$ and for any positive integer $n$.
	
	\textsc{Step 2: $\ee$-Dependent Estimates.} As second step, we establish some $\ee$-dependent estimates which will be useful in order 
	to show that $(\theta_n,\,u_n,\,\Pi_n)_\NN$ is a Cauchy sequence in a suitable space. First, we claim that 
	\begin{equation}\label{ee_estimates2_induction}
		\|t^{	\gamma_1	}			u_{n,\lambda}	\|_{	L^{	\frac{2r}{1-\ee r}	}_t L^{	p_3	}_x	} +
		\|t^{	\beta		}	\nabla	u_{n,\lambda}	\|_{	L^{	\frac{2r}{1-\ee r}	}_t L^{	p_2	}_x	} 
		\leq	\bar{C}_4
		\|	\bar{u}	\|_{	\BB_{p,r}^{	\frac{d}{p}-1+\ee	}	},
	\end{equation}
	where $u_{n,\lambda}(t)=u_n(t)h(0,t)$, with $h$ is defined by \eqref{def_h2}.
	Recalling the characterization of the homogenous Besov spaces given by Theorem \ref{Characterization_of_hom_Besov_spaces} and the 
	embedding of Theorem \ref{Theorem_embedding_Besov}, we get
	\begin{equation}\label{ee_estimates2_1}
		\|t^{	\gamma_1	}			u_L	\|_{	L^{	\frac{2r}{1-\ee r}	}_t L^{	p_3	}_x	} +
		\|t^{	\beta		}	\nabla	u_L	\|_{	L^{	\frac{2r}{1-\ee r}	}_t L^{	p_2	}_x	}
		\leq 	C
		\|	\bar{u}	\|_{	\BB_{p,r}^{	\frac{d}{p}-1+\ee	}	},
	\end{equation}
	for a suitable $C>0$. Furthermore, Lemma \ref{Lemma_A.3} and Lemma \ref{Lemma_A.4} yields
	\begin{align*}
		\|&t^{	\gamma_1	}			F^1_{n+1,\lambda}	\|_{	L^{	\frac{2r}{1-\ee r}	}_t L^{	p_3	}_x	} +
		\| t^{	\beta		}	\nabla	F^1_{n+1,\lambda}	\|_{	L^{	\frac{2r}{1-\ee r}	}_t L^{	p_2	}_x	}	
		\leq 	\bar{C}
		\Big\{
			\frac{1}{\lambda^{\frac{1}{2r}}}		
			\|t^{	\beta		}	\nabla	u^h_{	n+1	,\lambda	}	\|_{	L^{	\frac{2r}{1-\ee r}	}_t L^{	p_2	}_x	}	+
			\|t^{	\gamma_2	}			u^h_{	n				}	\|_{	L^{	\infty				}_t L^{	p_3	}_x	}	
			{	\scriptstyle	\times}\\ &{	\scriptstyle	\times}	
			\|t^{	\beta		}	\nabla	u^h_{	n	,\lambda	}	\|_{	L^{	\frac{2r}{1-\ee r}	}_t L^{	p_2	}_x	}	+
			\|t^{	\gamma_2	}			u^h_{	n+1				}	\|_{	L^{	\infty				}_t L^{	p_3	}_x	}	
			\|t^{	\beta		}	\nabla	u^d_{	n	,\lambda	}	\|_{	L^{	\frac{2r}{1-\ee r}	}_t L^{	p_2	}_x	}	+
			\|t^{	\gamma_2	}			u^h_{	n				}	\|_{	L^{	\infty				}_t L^{	p_3	}_x	}	
			\|t^{	\beta		}	\nabla	u^d_{	n+1	,\lambda	}	\|_{	L^{	\frac{2r}{1-\ee r}	}_t L^{	p_2	}_x	}
		\Big\},
	\end{align*}
	for a positive constant $\bar{C}$. Imposing $\lambda:= (2\bar{C})^{2r}$, we deduce that	
	\begin{equation}\label{ee_estimates2_2}
	\begin{aligned}
		\|&t^{	\gamma_1	}			F^1_{n+1,\lambda}	\|_{	L^{	\frac{2r}{1-\ee r}	}_t L^{	p_3	}_x	} +
		\| t^{	\beta		}	\nabla	F^1_{n+1,\lambda}	\|_{	L^{	\frac{2r}{1-\ee r}	}_t L^{	p_2	}_x	}	
		\leq 	
		\frac{1}{2}
		\|t^{	\beta		}	\nabla	u^h_{	n+1	,\lambda	}	\|_{	L^{	\frac{2r}{1-\ee r}	}_t L^{	p_2	}_x	}	+\\&+
		\bar{C}C_1\eta		
		\|t^{	\beta		}	\nabla	u^h_{	n	,\lambda	}	\|_{	L^{	\frac{2r}{1-\ee r}	}_t L^{	p_2	}_x	}	+
		\bar{C}C_1\eta		
		\|t^{	\beta		}	\nabla	u^d_{	n	,\lambda	}	\|_{	L^{	\frac{2r}{1-\ee r}	}_t L^{	p_2	}_x	}	+
		\bar{C}C_1\eta		
		\|t^{	\beta		}	\nabla	u^d_{	n+1	,\lambda	}	\|_{	L^{	\frac{2r}{1-\ee r}	}_t L^{	p_2	}_x	}.
	\end{aligned}
	\end{equation} 
	Moreover, Theorem \ref{Maximal_regularity_Thm_weight_time} and Lemma \ref{Lemma6} imply
	\begin{equation}\label{ee_estimates2_3}
	\begin{aligned}
		\| t^{	\gamma_1	}		   (	F^2_{n+1,\lambda} &+ F^3_{n+1,\lambda} )	\|_{	L^{	\frac{2r}{1-\ee r}	}_t L^{	p_3	}_x	} +
		\| t^{	\beta		}	\nabla ( 	F^2_{n+1,\lambda}  + F^3_{n+1,\lambda} )	\|_{	L^{	\frac{2r}{1-\ee r}	}_t L^{	p_2	}_x	}
		\leq 	\\ \leq 
		\| t^{	\gamma_1	}		  &(	F^2_{n+1		} + F^3_{n+1	} )	\|_{	L^{	\frac{2r}{1-\ee r}	}_t L^{	p_3	}_x	}  +
		\| t^{	\beta		}	\nabla ( F^2_{n+1		} + F^3_{n+1	} )	\|_{	L^{	\frac{2r}{1-\ee r}	}_t L^{	p_2	}_x	}	\\
		&\leq	C
		\|	\nu	-	1	\|_{\infty}
		\|t^{	\beta		}	\nabla	u_{	n	}	\|_{	L^{	\frac{2r}{1-\ee r}	}_t L^{	p_2	}_x	}		
		\leq	
		\tilde{C}\eta
		\|t^{	\beta		}	\nabla	u_{	n,	\lambda		}	\|_{	L^{	\frac{2r}{1-\ee r}	}_t L^{	p_2	}_x	}
	\end{aligned}
	\end{equation}
	assuming $C_r$ in the definition of $\eta$ big enough. Summarizing \eqref{ee_estimates2_1}, \eqref{ee_estimates2_2} and 
	\eqref{ee_estimates2_3}, there exists a positive constant $C$ such that
	\begin{align*}
		\|t^{	\gamma_1	}			u_{n+1,\lambda}	\|_{	L^{	\frac{2r}{1-\ee r}	}_t L^{	p_3	}_x	} +
		\|t^{	\beta		}	\nabla	u_{n+1,\lambda}	\|_{	L^{	\frac{2r}{1-\ee r}	}_t L^{	p_2	}_x	}
		\leq 
		C	\bar{C}_4	\eta	\|	\bar{u}	\|_{	\BB_{p,r}^{\frac{d}{p}-1}+\ee	},		
	\end{align*}
	so that \eqref{ee_estimates2_induction} is true for any positive integer $n$. Finally, multiplying both the left and right-hand sides 
	of \eqref{ee_estimates2_induction} by $sup_{t\in\RR} h^{-1}(0,t)$, we get
	\begin{equation}\label{ee_estimates2_conclusion}
		\|t^{	\gamma_1	}			u_{n}	\|_{	L^{	\frac{2r}{1-\ee r}	}_t L^{	p_3	}_x	} +
		\|t^{	\beta		}	\nabla	u_{n}	\|_{	L^{	\frac{2r}{1-\ee r}	}_t L^{	p_2	}_x	} 
		\leq	C_5
		\|	\bar{u}	\|_{	\BB_{p,r}^{	\frac{d}{p}-1+\ee	}	}
		\exp
		\big\{	C_6
			\|	\bar{u}^d	\|_{	\BB_{p,r}^{	\frac{d}{p}-1	}	}^{2r}
		\big\},
	\end{equation}
	for two suitable positive constant $C_5$ and $C_6$.
	
	\textsc{Step 3. Convergence of the Series.} 
	We proceed as in the third step of Theorem \ref{Theorem_solutions_smooth_dates}, denoting $\delta u_n:= u_{n+1}-u_{n}$, 
	$\delta \nu_n:= \nu(\theta_{n+1})-\nu(\theta_{n})$ and $\delta \theta_n:= \theta_{n+1}-\theta_{n}$. We define
	\begin{align*}
	\delta U_{n,\lambda}(T):=
			 \|	t^{	\gamma_1	}			\delta u_{n,\lambda}	\|_{L^{	2r		}(0,T;L^{p_3}_x)}
			+\|	t^{	\gamma_2	}			\delta u_{n,\lambda}	\|_{L^{	\infty	}(0,T;L^{p_3}_x)}
			+\|	t^{	\beta		}	\nabla 	\delta u_{n,\lambda}	\|_{L^{	2r		}(0,T;L^{p_2}_x)},
	\end{align*}
	where $\delta u_{n,\lambda}(t):=\delta u_n(t)h_{n,\lambda}(0,t)$. We claim that the series $\sum_{n\in \NN}\delta U_{n}(T)$ is 
	convergent. First, we split $\delta u_n$ into $\delta u_{n,\lambda }=f_{n,1}+f_{n,2}+f_{n,3}$, where $f_{n,i}$ is defined by 
	\eqref{formulation_delta_u_n}, for $i=1,2,3$. We begin estimating $f_{n,1}$. Lemma \ref{Lemma_A.3} and Lemma \ref{Lemma_A.4} yield that 
	\begin{align*}
		\|	&t^{	\gamma_1	}			\delta f_{n,1}	\|_{L^{	2r		}(0,T;L^{p_3}_x)}+
		\|	 t^{	\gamma_2	}			\delta f_{n,1}	\|_{L^{	\infty	}(0,T;L^{p_3}_x)}+
		\|	 t^{	\beta		}	\nabla 	\delta f_{n,1}	\|_{L^{	2r		}(0,T;L^{p_2}_x)}
		\leq\\
		&\leq C
		\Big\{ 
			\frac{1}{\lambda^{\frac{1}{2r}}}
			\big(
				\|	t^{\beta	} 	\partial_d 	\delta 	u_{n,\lambda}^h		\|_{L^{2r		}(0,T;	L^{p_2}_x)	}+
				\|	t^{\gamma_1	}				\delta 	u_{n,\lambda}^h		\|_{L^{2r		}(0,T;	L^{p_3}_x)	}+
				\|	t^{\beta	}	\nabla^h 	\delta 	u_{n,\lambda}^h		\|_{L^{2r		}(0,T;	L^{p_2}_x)	}
			\big)+\\&\quad+
			\|	t^{\gamma_2		}				\delta 	u_{n,\lambda}^d 	\|_{L^{\infty	}(0,T;	L^{p_3}_x)	}
			\|	t^{\beta		} 	\partial_d 			u_n^h				\|_{L^{2r		}_t		L^{p_2}_x	}+
			\|	t^{\gamma_2		}				\delta 	u_{n-1,\lambda}^h 	\|_{L^{\infty	}(0,T;	L^{p_3}_x)	}
			\|	t^{\beta		}	\nabla 				u_n^h 				\|_{L^{2r		}_t 	L^{p_2}_x	}+\\&\quad+
			\|	t^{\gamma_2		}						u_{n-1}^h 			\|_{L^{2r		}_t		L^{p_3}_x	}
			\| 	t^{\beta		}	\nabla 		\delta 	u_{n-1,\lambda}^h	\|_{L^{2r		}(0,T;	L^{p_2}_x)	}+	
			\|	t^{\beta		}	\nabla^h	\delta 	u_{n,\lambda}^d		\|_{L^{2r		}(0,T;	L^{p_2}_x)	} 
			\|	t^{\gamma_2		}						u_n^h				\|_{L^{\infty	}_t 	L^{p_3}_x	}+\\&\quad+
			\|	t^{\gamma_2		}				\delta 	u_{n-1,\lambda}^d	\|_{L^{\infty	}(0,T;	L^{p_3}_x)	} 
			\|	t^{\beta		}	\nabla^h 			u_n^h				\|_{L^{2r		}_t 	L^{p_2}_x	}
		\Big\}.
	\end{align*}
	which yields, 
	\begin{equation}\label{estimates_f_1_general_case}
	\begin{aligned}
		 \|	 t^{	\gamma_1	}			\delta f_{n,1}	\|_{L^{	2r		}(0,T;L^{p_3}_x)}&+
		 \|	 t^{	\gamma_2	}			\delta f_{n,1}	\|_{L^{	\infty	}(0,T;L^{p_3}_x)} +\\&+
		 \|	 t^{	\beta		}	\nabla 	\delta f_{n,1}	\|_{L^{	2r		}(0,T;L^{p_2}_x)}
		\leq
			\frac{1}{4}
			\big(
				\delta U_{n,\lambda}(T)+\delta U_{n-1,\lambda}(T)
			\big),
	\end{aligned}
	\end{equation}
	assuming $\eta$ small enough. Now, we carry out the estimate of $f_{n,2}$. Lemma \ref{Lemma4} and Theorem 
	\ref{Maximal_regularity_Thm_weight_time} imply
	\begin{align*}
		\|	t^{	\gamma_1	}			\delta f_{n,2}				\|_{L^{	2r		}(0,T;L^{p_3}_x)}&+
		\|	t^{	\gamma_2	}			\delta f_{n,2}				\|_{L^{	\infty	}(0,T;L^{p_3}_x)} +
		\|	t^{	\beta		}	\nabla 	\delta f_{n,2}				\|_{L^{	2r		}(0,T;L^{p_2}_x)}
		\leq \\	&\leq 	C
		\|\nu-1\|_{\infty}	
		\|	t^{	\beta		}	\nabla \delta u_{n-1}				\|_{L^{2r		}(0,T;L^{p_2}_x)}
		\leq \tilde{C}_r\eta
		\|	t^{		\beta		}	\nabla \delta u_{n-1,\lambda}	\|_{L^{2r		}(0,T;L^{p_2}_x)},
	\end{align*}
	hence, we deduce that
	\begin{equation}\label{estimates_f_2_general_case}
		\|	t^{	\gamma_1	}			\delta f_{n,2}				\|_{L^{	2r		}(0,T;L^{p_3}_x)} +
		\|	t^{	\gamma_2	}			\delta f_{n,2}				\|_{L^{	\infty	}(0,T;L^{p_3}_x)} +
		\|	t^{	\beta		}	\nabla 	\delta f_{n,2}				\|_{L^{	2r		}(0,T;L^{p_2}_x)}\leq 
		\bar{C}_r\eta \delta U_{n-1,\lambda}(T).
	\end{equation}
	Now we deal with $f_{n,3}$. Thanks to  Lemma \ref{Lemma6} and Theorem \ref{Maximal_regularity_Thm_weight_time}, we have
	\begin{equation}
	\begin{aligned}\label{estimates_f3_general_case_partA}
		\|&t^{	\gamma_1	}			\delta f_{n,3}			\|_{L^{	2r				}(0,T;	L^{p_3		}_x)}+
		\|	t^{	\gamma_2	}			\delta f_{n,3}			\|_{L^{	\infty			}(0,T;	L^{p_3		}_x)} +
		\|	t^{	\beta		}	\nabla 	\delta f_{n,3}			\|_{L^{	2r				}(0,T;	L^{p_2		}_x)}
		 \leq	\\& \leq 	
		\|	t^{\beta		} 			\delta 	\nu_n \MM_n 	\|_{L^{2r				}(0,T; 	L^{ p_2 	}_x)} 
		\leq C
		\| 								\delta 	\nu_n			\|_{L^{\frac{2}{\ee}	}(0,T;	L^{\infty	}_x)}
		\|	t^{\beta		} 	\nabla		 			u_n 	\|_{L^{2r				}(0,T; 	L^{ p_2 	}_x)}
		\leq \hat{C}_1(\bar{u})
		\| 								\delta \theta_n			\|_{L^{\frac{2}{\ee}	}(0,T;	L^{\infty	}_x)}
	\end{aligned}
	\end{equation}	
	where $\hat{C}_1(\bar{u})$ is a positive constant which depends on $\| \bar{u} \|_{\BB_{p,r}^{d/p-1+\ee}}$. Now, recalling that 
	$\delta \theta_n$ is determined by \eqref{delta_theta_n}, we get
	\begin{equation*}
		\| \delta \theta_n(t)\|_{L^\infty_x} \leq 
		\int_0^t 
			\frac{s^{\gamma_1}\|\delta \theta_n(s)u_n(s)		\|_{L^{p_3}_x}}{s^{\gamma_1}|\ee(t-s)|^{\frac{d}{2}\frac{1}{p_3}+\frac{1}{2}}}
		\dd s
		+
		\int_0^t 
			\frac{s^{\gamma_1}\|\delta u_{n-1}(s)\theta_n(s)	\|_{L^{p_3}_x}}{s^{\gamma_1}|\ee(t-s)|^{\frac{d}{2}\frac{1}{p_3}+\frac{1}{2}}}
			\dd s,
	\end{equation*}
	hence, defining $\alpha := (d/(2p_3)+1/2)(2r)'<1$, $\| \delta \theta_n(t)\|_{L^\infty_x}^{2r}$ is bounded by
	\begin{align*}
		2^{2r-1}
		\Big(
			\int_0^t
				\frac{1}{s^{\gamma_1(2r)'}|\ee(t-s)|^\alpha}\dd s		
		\Big)^{2r-1}
		\Big\{
			\int_0^t & 
				\|\delta \theta_n(s)\|_{L^\infty_x}^{2r} 	s^{2r\gamma_1 }\| 			u_n		(s)	\|_{L^{p_3}_x}^{2r}
			\dd s+\\&+
			\int_0^t 
				\|		\bar{\theta}\|_{L^\infty_x}^{2r}	s^{2r\gamma_1 }\|	\delta 	u_{n-1}	(s)	\|_{L^{p_3}_x}^{2r}
			\dd s
		\Big\}.
	\end{align*}
	Then, using the Gronwall inequality, we have
	\begin{align*}
		\| \delta \theta_n(t)\|_{L^\infty_x}^{2r} 
		\leq		\hat{C}_2(t)
		\|\bar{\theta}\|_{L^\infty_x}^{2r}
		\int_0^t 
			s^{2r\gamma_1}	 \|\delta u_{n-1}(s)\|_{L^{p_3}_x}^{2r}
		\dd s
		\exp
		\Big\{
			\int_0^t 
				 s^{2r \gamma_1}\| u_n(s)\|_{L^{p_3}_x}^{2r}
			 \dd s
		\Big\},
	\end{align*}
	which yields 
	$
		\|\delta \theta_n(t)\|_{L^\infty_x}\leq \chi (t)\delta U_{n-1}(t),
	$
	where $\chi$ is an increasing function.
	Hence, Recalling \eqref{estimates_f3_general_case_partA}, we deduce that
	\begin{equation*}
		\|	t^{	\gamma_1	}			\delta f_{n,3}			\|_{L^{	2r				}(0,T;	L^{p_3		}_x)}+
		\|	t^{	\gamma_2	}			\delta f_{n,3}			\|_{L^{	\infty			}(0,T;	L^{p_3		}_x)} +
		\|	t^{	\beta		}	\nabla 	\delta f_{n,3}			\|_{L^{	2r				}(0,T;	L^{p_2		}_x)}
		\leq \hat{C}_1(\bar{u})\chi(T)\|\delta U_{n-1} \|_{L^{\frac{4}{\ee}}(0,T)}.							
	\end{equation*}
	Summarizing the last inequality with \eqref{estimates_f_1_general_case} and \eqref{estimates_f_2_general_case}, 
	we finally deduce that
	\begin{equation*}
		\delta U_{n,\lambda}(T) \leq  \Big(\frac{1}{3}+\frac{4}{3}\tilde{C}_r \eta \Big)\delta U_{n-1,\lambda}(T)
		+ \frac{4}{3}\hat{C}_1(\bar{u})\chi(T)\|\delta U_{n-1} \|_{L^{\frac{2}{\ee}}(0,T)},
	\end{equation*}
	which is equivalent to to \eqref{est_deltaUn}. Thus we can conclude proceeding as in the last part of Theorem 
	\ref{Theorem_solutions_smooth_dates}.
\end{proof}
Now, we want to prove that system \eqref{Navier_Stokes_system} admits a weak solution, adding some regularity to the initial data.
\begin{theorem}\label{Theorem_solutions_smooth_dates2}
	Let us assume that the hypotheses of Theorem \ref{Main_Theorem2} are fulfilled. Suppose that $\bar{\theta}$ belongs to 
	$L^2_x \cap L^\infty_x $ and $\bar{u}$ belongs to $\BB_{p,r}^{d/p-1}\cap \BB_{p,r}^{d/p-1+\ee} $ with $\ee<\min\{1/(2r), 1-1/(2r), d/p-1\}$. 
	If the smallness condition \eqref{smallness_condition_general_case} holds then there exists a global weak solution $(\theta, u, \Pi)$  
	of \eqref{Navier_Stokes_system_eps} which satisfies the properties of Theorem \ref{Main_Theorem2}.
\end{theorem}
\begin{proof}
	By Proposition \ref{Proposition_solutions_smooth_data_general_case}, there exists $(\theta_\ee,\,u_\ee,\,\Pi_\ee)$, solution of
	\eqref{Navier_Stokes_system_eps}, such that $t^{\gamma_1}u_\ee$ belongs to $L^{2r}_t L^{p_3}_x$, $t^{\gamma_2} u_\ee$ belongs to $L^{\infty}_t L^{p_3}_x$, $t^{\beta}\nabla u_\ee$ 
	lives in $L^{2r}_t L^{p_2}_x$, $t^{\alpha}\nabla u_\ee$ in $L^{2r}_t L^{p^*}_x$, $\theta_\ee$ in $L^{\infty}_{t,x}$ and $t^\alpha \Pi_\ee$ in $L^{2r}_t L^{p^*}_x$. 
	Then, thanks to inequalities \eqref{inequalities_statement_prop_smooth_dates}, there exists $(\theta,\,u,\,\Pi)$ in the same space of $(\theta_\ee,\,u_\ee,\,\Pi_\ee)$, such 
	that
	\begin{equation*}
	\begin{array}{lll}
		t^{\gamma_1} 		u_{\ee_n} 				\rightharpoonup  	t^{\gamma_1}		u	\quad	w-L^{2r		}_t L^{p_3	}_x,	&
		t^{\gamma_2} 		u_{\ee_n} 				\rightharpoonup  	t^{\gamma_2}		u	\quad	w-L^{\infty	}_t L^{p_3	}_x,	&
		t^{\beta   }\nabla 	u_{\ee_n} 				\rightharpoonup 	t^{\beta   }\nabla 	u 	\quad 	w-L^{2r		}_t L^{p_2	}_x,	\\
		t^{\alpha  }\nabla 	u_{\ee_n} 				\rightharpoonup 	t^{\alpha  }\nabla 	u	\quad	w-L^{2r		}_t L^{p^*	}_x,	&
		\theta_{\ee_n} 				\overset{*}{	\rightharpoonup} 	\theta					\quad	w*-L^{\infty}_{t,x},			&
		t^{\alpha  }		\Pi_{\ee_n}				\rightharpoonup 	t^{\alpha  }\Pi			\quad	w-L^{2r}_t	 	L^{p^*	}_x,
	\end{array}
	\end{equation*}
	for a  positive decreasing sequence $(\ee_n)_\NN$ convergent to $0$. We claim that $(\theta, u, \Pi)$ is weak solution of \eqref{Navier_Stokes_system}.
	First, we show that $ u_{\ee_n}$ strongly converges to $u$ in $L^{\tau_3}(0,T;L^{p_3}_x)$, up to 
	a subsequence, with a suitable $\tau_3>1$. We proceed establishing that $\{ u_\ee-u_L\,|\,\ee>0\}$ is a compact set in $C([0,T]; \dot{W}_x^{-1,p^*})$, for all $T>0$. Applying 
	$(\sqrt{-\Delta})^{-1}$ to the momentum equation of \eqref{Navier_Stokes_system_eps}, we observe that $t^{\alpha}\partial_t {(\sqrt{-\Delta})^{-1}}u_\ee $ is uniformly bounded in 
	$L^{2r}(0,T;L^{p^*}_x)$. Hence, observing that $\alpha (2r)'<1$, we get
	\begin{equation*}
		\|\partial_t (\sqrt{-\Delta})^{-1} u_\ee	 			\|_{L^{1	}(0,T; L^{p^*}_x)	} \leq  
		\frac{T^{1-\alpha(2r)'}}{1-\alpha(2r)'}
		\| t^{ \alpha }\partial_t (\sqrt{-\Delta})^{-1} u_\ee	\|_{L^{ 2r  }(0,T  L^{p^*}_x)	}
	\end{equation*}	 
	Thus $\{(\sqrt{-\Delta})^{-1}(u_\ee-u_L)\,|\,\ee>0\}$ is an equicontinuous and bounded family of $C([0,T], L_x^{p^*})$, namely it is a compact family. 
	Then we can extract a subsequence (which we still denote by $u_{\ee_n}$) such that $(\sqrt{-\Delta})^{-1}(u_{\ee_n}-u_L)$ strongly converges to 
	$(\sqrt{-\Delta})^{-1}(u-u_L)$ in $L^{\infty}(0,T;L^{p^*}_x)$, that is $ u_{\ee_n}-u_L$ strongly converges to $u-u_L$ in $L^{\infty}(0,T; \dot{W}_x^{-1,p^*} )$. 
	Now, passing through the following real interpolation
	 \begin{equation*}
	 	\Big[  \dot{W}_x^{-1,p^*},  \dot{W}_x^{1,p^*}\Big]_{\mu,1}=
	 	\dot{B}_{p^*,1}^{\frac{d}{p^*}-\frac{d}{p_3}}\hookrightarrow L^{p_3}_x,
	\end{equation*}
	with $\mu:= (d/p^* -d/p_3)+1/2<1$ (see \cite{MR0482275}, Theorem $6.3.1$ and  \cite{MR2768550}, Theorem $2.39$), we deduce that
	\begin{align*}
		\| u_{\ee_n}-u \|_{L^{\tau}(0,T;L^{p_3}_x)}
		&\leq 
		C	\Big\| 
				\|u_{\ee_n}-u\|_{\dot{W}_x^{-1,p^* }}^{1-\mu} 
				\|u_{\ee_n}-u\|_{\dot{W}_x^{1,p^*}}^{\mu}
			\Big\|_{L^{\tau}(0,T)}\\
		&\leq
		C	\|u_{\ee_n}-u\|_{L^{\infty}(0,T;\dot{W}^{-1,p^*}_x)}^{1-\mu}
			\| t^{-\alpha} \|_{L^{\frac{2r\tau}{2r-\tau}}(0,T)}^\mu
			\|t^{\alpha}\nabla (u_{\ee_n}-u)\|_{L^{2r}(0,T;L^{p^*}_x)}^{\mu},
	\end{align*}
	for all $T>0$, where we have considered $\tau\in (1, 2r/(1+ 2\alpha r))$ so that $\alpha 2r \tau /(2r-\tau)<1$. Moreover, we choose $\tau$ such that there exist 
	$\tau_2$ in $(1, 2r/(1+2\beta r ))$ and $\tau_3$ in $(1, 2r/(1+2\gamma_1 r) )$ which fulfill $1/\tau_3 + 1/\tau_2 = 1/\tau_1$. Let us remark that the norms 
	\begin{align*}
		\| u_{\ee_n} \|_{L^{\tau_3}(0,T; L^{p_3}_x)} 		&\leq \| t^{\gamma_1}\|_{L^{\frac{2r\tau_3}{2r-\tau_3}}(0,T)}\|t^{\gamma_1} u_{\ee_n} \|_{L^{2r}_t L^{p_3}_x}<\infty,\\
		\|\nabla  u_{\ee_n} \|_{L^{\tau_2}(0,T; L^{p_2}_x)} &\leq \| t^{\beta}\|_{L^{\frac{2r\tau_2}{2r-\tau_2}}(0,T)}\|t^{\beta} u_{\ee_n} \|_{L^{2r}_t L^{p_2}_x}<\infty,
	\end{align*}
	that is they are uniformly bounded in $n$. Now, we consider $\tau<\sigma <\tau_3$ strictly closed to $\tau_3$ so that it still fulfills $1/\sigma + 1/\tau_2 >1$. Then the following 
	interpolation inequality
	\begin{equation*}
		\| u_{\ee_n} - u \|_{L^\sigma	(0,T;L^{p_3}_x)} \leq 
		\| u_{\ee_n} - u \|_{L^\tau		(0,T;L^{p_3}_x)}^{\frac{\tau_3 -\sigma}{\tau - \tau_3}} 
		\| u_{\ee_n} - u \|_{L^{\tau_3}	(0,T;L^{p_3}_x)}^{\frac{\sigma -\tau_3 }{\tau - \tau_3}}, 
	\end{equation*}
	which converges to $0$ as $n$ goes to $\infty$, so that $u_{\ee_n}$ strongly converges to $u$ in $L^{\sigma}_{loc}(\RR_+;L^{p_3}_x)$. This yields that 
	$u_{\ee_n}\theta_{\ee_n}$ and $u_{\ee_n}\cdot \nabla u_{\ee_n}$ converge to $u\, \theta$ and $u\cdot \nabla u$, respectively, in the distributional sense.
	We deduce that $\theta$ is weak solution of
	\begin{equation*}
		\partial_t\theta + \Div (\theta u)=0\quad\text{in}\quad \RR_+ \times\RR^d,\quad\quad	
		\theta_{|t=0} = \bar{\theta}	\quad\text{in }\quad\RR^d.
	\end{equation*}
	Arguing as in theorem \ref{Theorem_solutions_smooth_dates}, $\theta_{\ee_n}$ converges almost everywhere to $\theta$, up to a subsequence, so that 
	$\nu(\theta_{\ee_n})$ strongly converges to $\nu(\theta)$ in $L^m_{loc}(\RR_+\times \RR^d)$, for every $1\leq m<\infty$, thanks to the Dominated Convergence Theorem. 
	Then $\nu(\theta_{\ee_n})\MM_{\ee_n}$ converges to $\nu(\theta)\MM$ in the distributional sense. 
	
	\noindent
	Summarizing all the previous considerations we finally conclude that $(\theta,\,u,\,\Pi)$ is a weak solution of 
	\eqref{Navier_Stokes_system} and it satisfies \eqref{inequalities_statement_prop_smooth_dates_general_data}.
\end{proof}

\section{Proof of Theorem \ref{Main_Theorem2}}

In this section we present the proof of Theorem \eqref{Main_Theorem2}. We proceed similarly as in the proof of Theorem \ref{Main_Theorem}, approximating our initial data by 
\begin{equation*}
	\bar{\theta}_n := \chi_n\sum_{|j|\leq n}\dot{\Delta}_j\bar{\theta}\quad\text{and}\quad 
	\bar{u}_n:=\sum_{|j|\leq n}\dot{\Delta}_j\bar{u}, \quad\text{for every}\quad n\in\NN,
\end{equation*}
where $\chi_n\leq 1$ is a cut-off function which has support on the ball $B(0,n)\subset \RR^d$, so that $\bar{\theta}_n\in L^\infty_x\cap L^2_x$ and $\bar{u}	\in		\dot{B}_{p,r}^{d/p}	\cap	\dot{B}_{p,r}^{d/p-1+\ee}$, with $\ee<\min\{1/(2r), 1-1/r, 2(d/p -2 + 1/r)\}$.  
Then, by Theorem \ref{Theorem_solutions_smooth_dates2}, there exists $(\theta_n, u_n, \Pi_n)$ weak solution of
\begin{equation*}
	\begin{cases}
		\partial_t\theta_n + \Div (\theta_n u_n)=0									& \RR_+ \times\RR^d,\\
		\partial_t u_n + u_n\cdot \nabla u_n -\Div (\nu(\theta_n)\nabla u_n) +\nabla\Pi_n=0	& \RR_+ \times\RR^d,\\
		\Div\, u_n = 0																& \RR_+ \times\RR^d,\\
		(\theta_n,\,u_n)_{t=0} = (\bar{\theta}_n,\,\bar{u}_n)						& \;\;\quad \quad\RR^d,
	\end{cases}
\end{equation*}
which belongs to the functional space defined in Theorem \ref{Main_Theorem2} and it fulfills the inequalities \eqref{Main_Theorem2_inequalities}, uniformly in $n\in\NN$. Then there exists a subsequence (which we still denote by $(\theta_n, u_n, \Pi_n)_\NN$) and an element $(\theta,\,u,\,\Pi)$ in the same space of $(\theta_n, u_n, \Pi_n)$, such that 
\begin{equation*}
\begin{array}{lll}
	t^{\gamma_1} 		u_{n} 				\rightharpoonup  	t^{\gamma_1}		u	\quad	w-L^{2r		}_t L^{p_3	}_x,	&
	t^{\gamma_2} 		u_{n} 				\rightharpoonup  	t^{\gamma_2}		u	\quad	w-L^{\infty	}_t L^{p_3	}_x,	&
	t^{\beta   }\nabla 	u_{n} 				\rightharpoonup 	t^{\beta   }\nabla 	u 	\quad 	w-L^{2r		}_t L^{p_2	}_x,	\\
	t^{\alpha  }\nabla 	u_{n} 				\rightharpoonup 	t^{\alpha  }\nabla 	u	\quad	w-L^{2r		}_t L^{p^*	}_x,	&
	\theta_{\ee_n} 				\overset{*}{\rightharpoonup} 	\theta					\quad	w*-L^{\infty}_{t,x},			&
	t^{\alpha  }		\Pi_{n}				\rightharpoonup 	t^{\alpha  }\Pi			\quad	w-L^{2r}_t	 	L^{p^*	}_x.
\end{array}
\end{equation*}
In order to complete the proof, we claim that $(\theta,\,u,\,\Pi)$ is weak solution of \eqref{Navier_Stokes_system}. 
We first rewrite $u_n= t^{-\gamma_1} t^{\gamma_1}u_n$, $\nabla u= t^{-\beta} t^{\beta}\nabla u$ and $\Pi_n= t^{-\alpha} t^{\alpha}\Pi_n$, so that the H\"older inequality guarantees that 
$u_n$, $\nabla u_n$ and $\Pi_n$ are uniformly bounded in $L^{\tau_3}(0,T; L^{p_3}_x)$, $L^{\tau_2}(0,T; L^{p_2}_x)$ and $L^{\tau_1}(0,T; L^{p^*}_x)$ respectively, with $T\in (0,\infty)$ and
\begin{equation*}
	\tau_1 \in \big( 1, \frac{2r}{1+2\alpha r} \big), \quad \tau_2 \in \big( 1, \frac{2r}{1+2\beta r} \big),\quad \tau_3 \in \big( 1, \frac{2r}{1+2\gamma_1 r} \big),
	\quad\text{such that}\quad
	\frac{1}{\tau_1} = \frac{1}{\tau_2} + \frac{1}{\tau_3}.
\end{equation*}
The same properties are preserved by $(\theta,\,u,\,\Pi)$. Moreover, arguing as in Theorem \ref{Theorem_solutions_smooth_dates2}, $u_n$ strongly converges to $u$ in $L^{\sigma}_{loc}(\RR_+;L^{p_3}_x)$, with $\sigma\in (\tau_1,\tau_3)$ strictly closed to $\tau_3$ so that $1/\sigma + 1/\tau_2 >1$. This yields that $u_n\cdot \nabla u_n$ and $u_n\theta_n$ converge to 
$u\cdot \nabla u$ and $u\,\theta$ respectively, in the distributional sense. Moreover, proceeding as in theorem \ref{Theorem_solutions_smooth_dates}, $\theta_{n}$ converges almost everywhere to $\theta$, up to a subsequence, so that $\nu(\theta_{n})$ strongly converges to $\nu(\theta)$ in $L^m_{loc}(\RR_+\times \RR^d)$, for every $1\leq m<\infty$, thanks to the Dominated Convergence Theorem. Then $\nu(\theta_{n})\MM_{n}$ converges to $\nu(\theta)\MM$ in the distributional sense and this allows us to conclude that $(\theta,\,u,\,\Pi)$ is weak solution of \eqref{Navier_Stokes_system}. Finally, passing through the limit as $n$ goes to $\infty$, $(\theta,\,u,\,\Pi)$ still fulfills inequalities \eqref{Main_Theorem2_inequalities} and this concludes the proof of the Theorem.

\appendix
\section{Inequalities}
In this section we improve Lemma \ref{Lemma2} and Lemma \ref{Lemma3} for a particular choice of the function $f$ and also with a perturbation of the operators, which is dependent on a parameter $\lambda>0$. This Lemmas are useful for the Theorem of section $3$, more precisely during the proof of the inequalities, since, for an opportune choice of $\lambda$, they permit to ``absorb'' some uncontrolled terms. Here the statements and the proofs.

\begin{lemma}\label{Lemma_A.1}
	Let $ 1< r < \infty$ and $q_1,\, q_2\in (1,\infty]$ such that $1/q=1/q_1+1/q_2\in ((2r-1)/dr, 1)$.	Let	$v \in L^{2r}_t L_x^{q_1}$ and
	for all $\lambda>0 $ let $h=h_{\lambda} $ be defined by
	\begin{equation*}
		h(s,t):= \exp\big\{ -\lambda\int_s^t 
		\|v\|_{L_x^{q_1}}^{2r}\big\}\text{,}
	\end{equation*}
	for all $0\leq s\leq t<\infty $ and consider $\mathcal{C}_\lambda$, the operator defined by
	\begin{equation*}
		\mathcal{C}_\lambda (f)(t) := 
 		\int_0^t h(s,t)e^{(t-s)\Delta} f(s)\dd s.
	\end{equation*}
	Then there exists a positive constant $C_r$, such that
		\begin{equation*}
		\|\mathcal{C}_\lambda(v\omega) \|_{L^{2r}_t L_x^{q_3}}
		\leq C_r\frac{1}{\,\lambda^{\frac{1}{4r}}}\|v \|_{L^{2r}_t L_x^{q_1}}^\frac{1}{2}
		\|\omega \|_{L^{2r}_t L_x^{q_2}},
	\end{equation*}
	where $q_3$ is defined by $1/q_3= 1/q-(2r-1)/dr$.
\end{lemma}
\begin{proof}
	Notice that
	\begin{align*}
		\|\int_0^t h(s,t) K(t-s)*v\omega (s)\dd s \,\|_{L^{q_3}_x}
		&\leq
			\int_0^th(s,t) \|K(t-s)*v\omega (s)\|_{L^{q_3}_x}\dd s\\
		&\leq
			\int_0^th(s,t)\|K(t-s)\|_{L^{\tilde{q}}_x}\|v \omega (s)\|_{L^q_x}\dd s,
	\end{align*}
	where $	1/\tilde{q}'=1-1/\tilde{q}=1/q-1/q_3=(2r-1)/(dr)$.
	By Remark \ref{remark2.1} and Holder inequality, we obtain
	\begin{equation}\label{inequatliy_A1}
	\begin{aligned}
		\|\mathcal{C}(v \omega)(t)\|_{L^{q}_x}
		&\leq
			\int_0^th(s,t)\|v (s)\|_{L^{p_1}}^\frac{1}{2}
			\frac{1}{\quad| t-s |^{\frac{2r-1}{2r}}}
			\|v (s)\|_{L^{p_1}}^\frac{1}{2}
			\|\omega (s)\|_{L^{p_2}_x}\dd s\\
		&\leq
			\bigg(
				\int_0^th(s,t)^{4r}\| v(s)\|^{2r}_{L^{p_1}_x}\dd s
			\bigg)^\frac{1}{4r}
			\bigg(\,
				\int_{\RR_+}
				\frac{ (\,
				\|v (s)\|_{L^{q_1}}^{\frac{1}{2}}
				\|\omega (s)\|_{L^{q_2}_x} )^{\frac{4r}{4r-1}}
				}
				{\quad| t-s |^{\frac{2r-1}{2r}\frac{4r}{4r-1}}}			
			\dd s
			\bigg)^{1-\frac{1}{4r}}.
	\end{aligned}
	\end{equation}
	Since
	\begin{equation*}
		g:=
		\Big(\,
			\|v (\cdot)\|_{L^{q_1}}^{\frac{1}{2}}
			\|\omega (\cdot)\|_{L^{q_2}_x}
		\Big)^{\frac{4r}{4r-1}}
		\in
		L_t^{\frac{4r-1}{3}},				
	\end{equation*}
	by Hardy-Littlewood-Sobolev inequality,
	\begin{equation*}
		|\cdot|^{-\frac{4r-2}{4r-1}}*g \in L_t^\frac{4r-1}{2}, 
	\end{equation*}
	and then
	\begin{equation*}
		\left(	
			|\cdot|^{-\frac{4r-2}{4r-1}}*g 
		\right)^{1-\frac{1}{4r}}\in L_t^{2r}.
	\end{equation*}
	Moreover there exists $C>$ such that
	\begin{align*}
	\|
		(
			|\cdot|^{-\frac{4r-2}{4r-1}}*g 
		)^{1-\frac{1}{4r}}
	\|_{L_t^{2r}}&=
	\|
			|\cdot|^{-\frac{4r-2)}{4r-1}}*g
	\|_{L_t^{\frac{4r-1}{2}}}^{1-\frac{1}{4r}}
	\leq
	C\|g\|_{L_t^{\frac{4r-1}{3}}}^{1-\frac{1}{4r}}\leq
	\bigg(
	\int_{\RR_+}
		(\,
			\|v (t)\|_{L^{q_1}_x}^{\frac{1}{2}}
			\|\omega (t)\|_{L^{q_2}_x} 
		\,)^{\frac{4}{3}r}\dd t
	\bigg)^{\frac{3}{4r}}\\
	&\leq 
	C\|\,\|v\|_{L^{q_1}_x}^\frac{1}{2} \|_{L^{4r}_t}
	\|\omega \|_{L^{2r}_t L^{q_2}_x}\leq
	C\|v\|_{L_t^{2r} L^{q_1}_x}^\frac{1}{2}
	\|\omega \|_{L^{2r}_t L^{q_2}_x}.
	\end{align*}
	Observing that
	\begin{equation*}
		\Big(
			\int_0^th(s,t)^{4r}\| v(s)\|^{2r}_{L^{q_1}_x}\dd s
		\Big)^\frac{1}{4r}
		\leq
		\Big(\frac{1}{4r\lambda}\Big)^{\frac{1}{4r}},
	\end{equation*}
	the Lemma is proved.	
\end{proof}

\begin{lemma}\label{Lemma_A.2}
	Let $ 1< r < \infty$ , $q_1\in [1,\frac{dr}{r-1}]$ and $v \in L^{2r}_t L_x^{q_1}$. 
	For all $\lambda>0 $ let $h=h_{\lambda} $ be defined as in Lemma \ref{Lemma_A.1} and let $\Bb_\lambda$ the operator defined by
	\begin{equation*}
		\mathcal{B}_\lambda(f)(t) := 
 		\int_0^t h(s,t)\nabla e^{(t-s)\Delta}f(s)\dd s.
	\end{equation*}
	For all $q_2\in [q_1',\infty]$, there exists  a positive constant $C_r$, such that
		\begin{equation*}
		\|\Bb_\lambda(v\omega) \|_{L^{2r}_t L_x^{q}}
		\leq C_r\frac{1}{\,\lambda^{\frac{1}{4r}}}
		\|v \|_{L^{2r}_t L_x^{p_1}}^\frac{1}{2}
		\|\omega \|_{L^{2r}_t L_x^{p_2}},		
	\end{equation*}
	where $q$ is defined by $1/q:=1/q_1+1/q_2-(r-1)/dr$.
\end{lemma}

\begin{lemma}\label{Lemma_A.3}
	Let $r\in (1,\infty)$, $p_1\in (d/2,d)$, $p_3> dr/(r-1)$ and $p_2$ be given by $1/p_1+1/p_2=1/p_3$. Let 
	$t^\gamma_1 v \in L^{2r}_t L^{L^{p_3}}_x$ and $t^{\beta} \omega\in L^{2r}_t L^{p_2}_x$. Defining 
	\begin{equation*}
		h_{\lambda}(s,t) := 
		\exp
		\Big\{ 
			- \lambda\int_s^t \tau^{2r\gamma_1}\|v		(\tau)\|_{L^{p_3}_x}^{2r}\dd\tau
			- \lambda\int_s^t \tau^{2r\beta}\|\omega	(\tau)\|_{L^{p_2}_x}^{2r}\dd\tau
		\,\Big\},
	\end{equation*}
	where $\lambda$ is a positive constant, there exists a positive constant $C_r$ such that 
	\begin{align}
		\|t^{\beta_1}\Bb(v\omega)_\lambda (t)\|_{L^{2r}_t L^{p_2}_x} 
		&\leq 
		\frac{C_r}{\lambda^{\frac{1}{2r}}}\|t^{\beta}\omega_\lambda \|_{L^{2r}_t L^{p_2}_x},\label{Lemma_A.3_est1}\\
		\|t^{\beta_1}\Bb(v\omega)_\lambda (t)\|_{L^{2r}_t L^{p_2}_x} 
		&\leq 
		\frac{C_r}{\lambda^{\frac{1}{2r}}}\|t^{\gamma_1}v_\lambda \|_{L^{2r}_t L^{p_3}_x}.\label{Lemma_A.3_est2}
	\end{align}
	\begin{proof}
		Remark \ref{remark2.1} yields that there exists a positive constant $C$ such that
		\begin{equation}\label{Lemma_A.3_first_inequality}
		\begin{aligned}
			t^{\beta}\|&\Bb(v\omega)_\lambda (t)\|_{L^{p_2}_x}			
			\leq 
			C
			\int_0^t \frac{t^{\beta_1}}{|t-s |^{\frac{d}{2p_3}+\frac{1}{2}}s^{\alpha_2}}
			h_{\lambda}(s,t)s^{\gamma_1}\|v(s)\|_{L^{p_3}_x}
			s^{\beta_1}\|\omega_\lambda(s)\|_{L^{p_2}_x}\dd s\\
			&\leq
			C
			\Big(
				\int_0^t h_\lambda (s,t)^{2r}s^{2r\gamma_1}\|v(s)\|_{L^{p_3}_x}\dd s
			\Big)^{\frac{1}{2r}}
			\Big(
			\int_0^t\Big| \frac{t^{\beta_1}}{|t-s |^{\frac{d}{2p_3}+\frac{1}{2}}s^{\alpha_2}}
			F(s)\Big|^{(2r)'}\dd s
			\Big)^{\frac{1}{(2r)'}}.
		\end{aligned}
		\end{equation}
		Hence, raising to the power of $(2r)'$ both the left-hand and the right-hand sides, we get
		\begin{equation*}
		\begin{aligned}
			t^{(2r)'\beta_1}\|\Bb(v\omega)_\lambda (t)\|_{L^{p_2}_x}^{(2r)'}
			&\lesssim 
			\frac{1}{\lambda^{\frac{(2r)'}{2r}}}
			\int_0^t\Big| \frac{t^{\beta_1}}{|t-s |^{\frac{d}{2p_3}+\frac{1}{2}}s^{\alpha_2}}
			s^{\beta_1}\|\omega(s)\|_{L^{p_2}_x}\Big|^{(2r)'}\dd s\\
			&\lesssim
			\frac{1}{\lambda^{\frac{(2r)'}{2r}}}\int_0^1
			\Big| 
				\frac{t^{\beta_1-\alpha_2 - \frac{N}{2p_3}-\frac{1}{2}}}
				{ |1-\tau |^{\frac{d}{2p_3}+\frac{1}{2}}\tau^{\alpha_2}}
			F(t\tau)\Big|^{(2r)'}t\,\dd\tau,
		\end{aligned}
		\end{equation*}
		where $F(s) = s^{\beta}\|\omega_\lambda(s)\|_{L^{p_2}_x}$. Observing that 
		$	\beta-\alpha_2 - N/(2p_3)-1/2 = 1/(2r)-1 = -1/(2r)'$,
		we get
		\begin{equation}\label{estimate_lambda1}
			t^{(2r)'\beta_1}\|\Bb(v\omega)_\lambda (t)\|_{L^{p_2}_x}^{(2r)'}\lesssim 
			\frac{1}{\lambda^{\frac{(2r)'}{2r}}}
			\int_0^1
			\Big| 
				\frac{1}{ |1-\tau |^{\frac{d}{2p_3}+\frac{1}{2}}\tau^{\alpha_2}}
			F(t\tau)\Big|^{(2r)'}\,\dd\tau,
		\end{equation}
		Hence, applying the $L^{(2r)/(2r)'}_t$-norm to both the left and right-hand sides,
		\begin{equation*}
		\begin{aligned}
			\| t^{\beta_1}\Bb(v\omega)_\lambda (t)\|_{L^{2r}_tL^{p_2}_x}^{(2r)'}&\lesssim
			\frac{1}{\lambda^{\frac{(2r)'}{2r}}} 
			\int_0^1
			\Big| 
				\frac{1}{ |1-\tau |^{\frac{d}{2p_3}+\frac{1}{2}}\tau^{\alpha_2}}\Big|^{(2r)'}
			\Big(\int_0^\infty F(t\tau)^{2r}\,\dd\tau\Big)^{\frac{1}{2r-1}}\dd t\\&\lesssim
			\frac{1}{\lambda^{\frac{(2r)'}{2r}}} 
			\int_0^1
			\Big| \frac{1}{ |1-\tau |^{\frac{d}{2p_3}+\frac{1}{2}}\tau^{\alpha_1}}\Big|^{(2r)'} 
			\dd \tau\|t^{\beta}\omega_\lambda \|_{L^{2r}_t L^{p_2}_x}^{(2r)'},
		\end{aligned}
		\end{equation*}
		thanks to Minkowski inequality. Since $\alpha_1(2r)'<1$ and $(d/(2p_3)+1/2)(2r)'<1$ we finally obtain \eqref{Lemma_A.3_est1}. Now, defining 
		$F(t):=s^{\gamma_1}\|v_\lambda(s)\|_{L^{p_3}_x}$, we also have
		\begin{equation*}
			t^{\beta}\|\Bb(v\omega)_\lambda (t)\|_{L^{p_2}_x}			
			\leq 
			C
			\Big(
				\int_0^t h_\lambda (s,t)^{2r}s^{2r\beta}\|\omega(s)\|_{L^{p_2}_x}\dd s
			\Big)^{\frac{1}{2r}}
			\Big(
			\int_0^t\Big| \frac{t^{\beta}}{|t-s |^{\frac{d}{2p_3}+\frac{1}{2}}s^{\alpha_2}}
			F(s)\Big|^{(2r)'}\dd s
			\Big)^{\frac{1}{(2r)'}},
		\end{equation*}
		which is equivalent to \eqref{Lemma_A.3_first_inequality}. Thus, arguing as for proving \eqref{Lemma_A.3_est1}, we also obtain \eqref{Lemma_A.3_est2}.
	\end{proof}
\end{lemma}

\begin{lemma}\label{Lemma_A.4}
	Let $r\in (2,\infty)$, $p_1\in (dr/(2r-2),N)$ and $p_3\geq Nr/(r-2)$ such that $1/p_1+1/p_2=1/p_3$. Let 
	$h_\lambda$, $v$ and $\omega$ be defined as in the previous Lemma. Then there exists $C_r>0$ such that
	\begin{equation}\label{Lemma_A.4_est}
	\begin{aligned}
		\|t^{\gamma_1}\Cc(v\omega)_\lambda(t)\|_{L^{2r}_t L^{p_3}_x} + 
		\|t^{\gamma_2}\Cc(v\omega)_\lambda(t)\|_{L^{\infty}_t L^{p_3}_x}\leq
		\frac{C_r}{\lambda^{\frac{1}{2r}}}\|t^{\beta_1}\omega_\lambda \|_{L^{2r}_t L^{p_2}_x},\\
		\|t^{\gamma_1}\Cc(v\omega)_\lambda(t)\|_{L^{2r}_t L^{p_3}_x} + 
		\|t^{\gamma_2}\Cc(v\omega)_\lambda(t)\|_{L^{\infty}_t L^{p_3}_x}\leq 
		\frac{C_r}{\lambda^{\frac{1}{2r}}}\|t^{\beta_1}v_\lambda \|_{L^{2r}_t L^{p_3}_x}.		
	\end{aligned}
	\end{equation}
\end{lemma}
\begin{proof}
	We control the $L^{2r}_t L^{p_3}_x$ norm arguing as in previous proof. Indeed we have
	\begin{equation*}
			t^{(2r)'\gamma_1}\|\Cc(v\omega)_\lambda (t)\|_{L^{p_3}_x}^{(2r)'}
			\leq 
			C
			\frac{1}{\lambda^{\frac{(2r)'}{2r}}}
			\int_0^1
			\Big| 
				\frac{1}{ |1-\tau |^{\frac{d}{2p_2}}\tau^{\alpha_2}}
			F(t\tau)\Big|^{(2r)'}\,\dd\tau,
	\end{equation*}
	where $F(s)= s^\beta \|\omega_\lambda\|_{L^{p_2}_x}$ or $F(s)= s^{\gamma_1} \|v_\lambda\|_{L^{p_3}_x}$ 
	instead of \eqref{estimate_lambda1}. Let us take in consideration the $L^\infty_t L^{p_3}_x$ norm. With 
	a direct computation we get
	\begin{align*}
		\|t^{\gamma_2}\Cc (t)\|_{L^{p_3}_x}
		&
		\leq 
		C 
		\Big(\int_0^t \Big| \frac{t^{\gamma_2}}{|t-s|^{\frac{N}{2p_2}}s^{\alpha_2}}\Big|^{r'}\dd 
		s\Big)^{\frac{1}{r'}}
		\Big(\int_0^th(s,t)^r s^{r\gamma_1}\|v(s)\|_{L^{p_3}_x}^rs^{r\beta_1}\|\omega(s)\|_{L^{p_2}_x}^r
		\dd s\Big)^{\frac{1}{r}}\\
		&
		\leq 
		C \Big(\int_0^1 \Big| \frac{t^{\gamma_2-\alpha_2-\frac{N}{2p_2}}}
		{|1-\tau|^{\frac{N}{2p_2}}\tau^{\alpha_2}}\Big|^{r'}t\, \dd \tau\Big)^{\frac{1}{r'}}
		\Big(\int_0^th(s,t)^r s^{r\gamma_1}\|v(s)\|_{L^{p_3}_x}^rs^{r\beta_1}\|\omega(s)\|_{L^{p_2}_x}^r
		\dd s\Big)^{\frac{1}{r}}
	\end{align*}
	Thus, observing that $\gamma_2-\alpha_2-d/(2p_2)+1/r'=0$, $dr'/(2p_2)<1$ and $\alpha_2r'<1$, we conclude that 
	\begin{align*}
		\|t^{\gamma_2}\Cc (t)\|_{L^{p_3}_x}&\leq 
		\bar{C}_r
		\Big(
				\int_0^t h_\lambda (s,t)^{2r}s^{2r\gamma_1}\|v(s)\|_{L^{p_3}_x}\dd s
		\Big)^{\frac{1}{2r}}				
		\|t^{\beta}\omega_\lambda \|_{L^{2r}_t L^{p_2}_x} \quad\text{and}\\
		\|t^{\gamma_2}\Cc (t)\|_{L^{p_3}_x}&\leq 
		\bar{C}_r
		\Big(
				\int_0^t h_\lambda (s,t)^{2r}s^{2r\beta}\|\omega(s)\|_{L^{p_2}_x}\dd s
		\Big)^{\frac{1}{2r}}				
		\|t^{\gamma_1}\omega_\lambda \|_{L^{2r}_t L^{p_3}_x},
	\end{align*}
	for a suitable positive constant $\bar{C}_r$, which finally yields \eqref{Lemma_A.4_est}.
\end{proof}

\section{}
\begin{theorem}\label{Theorem_stokes_system_with_linear_perturbation}
	Let $r\in(1,\infty)$, $p\in (1,dr/(2r-1))$ and $\bar{u}\in\dot{B}_{p,r}^{d/p-1}$. Le us	suppose that
	\begin{equation*}
		f_1\in (L^{r}_t L^{\frac{dr}{3r-2}}_x)^d\cap (L^{r}_t L^{\check{p}}_x)^d,\quad
		f_2\in  (L^{2r}_tL^{\frac{dr}{2r-1}}_x)^{d\times d}\cap(L^{r}_tL^{\frac{dr}{2(r-1)}}_x)^{d\times d},
	\end{equation*}
	Let $v$ belongs to $L^{2r}_t L^{dr/(r-1)}_x$ with $\nabla v\in L^{2r}_t L^{\frac{dr}{2r-1}}_x$. Then system
	\begin{equation} \label{Stokes_system_with_linear_perturbation}
		\begin{cases}
			\partial_t u^h + v\,\partial_d u^h -\Delta u^h +\nabla^h\Pi
											=f_1^h+\Div f_2^h 					& \RR_+ \times\RR^d,\\
			\partial_t u^d + \nabla^h v \cdot u^h - v\,\Div^h u^h -\Delta u^d +\partial_d\Pi
											=f_1^d+\Div f_2^d 					& \RR_+ \times\RR^d,\\
			\Div u = 0															& \RR_+ \times\RR^d,\\
			u_{|t=0} = \bar{u}													&\;\;\quad \quad \RR^d,\\
		\end{cases}
	\end{equation}
	admits a weak solution $(u,\Pi)$, such that $u$ belongs to $L^{2r}_tL^{\frac{dr}{r-1}}_x$ with $\nabla u$ in $L^{2r}_tL^{\frac{dr}{2r-1}}_x$ and 
	$\Pi$ in $L^{r}_tL^{\frac{dr}{2(r-1)}}_x$.
\end{theorem}

\begin{proof}
	For all $u$ in $L^{2r}_t L^{dr/(r-1)}_x)^d $ with $ \nabla u \in L^{2r}_t L^{dr/(2r-1)}_x$, let $g(u)$ be defined by 
	\begin{equation}\label{def_g_apx}
		g(u):=  (- v\,\partial_d u^h , 
		-\nabla^h v\cdot u^h +v\,\Div ^h u^h )\in L^{r}_t L^{\frac{dr}{3r-2}}_x.
	\end{equation}
	Then, the momentum equations of \eqref{Stokes_system_with_linear_perturbation} reads as follows:
	\begin{equation}\label{Stokes_system_with_linear_perturbation_g} 
			\partial_t u  -\Delta u +\nabla \Pi		=g(u)+f_1+\Div f_2 	\quad\text{in}\quad					\RR_+ \times\RR^d,
	\end{equation}
	We want to prove the existence of a weak solution for this system, using the Fixed-Point Theorem. We define the functional space $Y_r$ by
	\begin{equation*}
		Y_r:= \Big\{ u\in L^{2r}_t L^{\frac{dr}{r-1}}_x
		\quad \text{such that}\quad
		\nabla u\in L^{2r}_t L^{\frac{dr}{2r-1}}_x	\Big\}, 
	\end{equation*}
	then, fixing a positive constant $\lambda$, we consider the norm $\|\cdot \|_\lambda$ on $Y_r$, defined by
	\begin{equation*}
		\|u\|_\lambda := 
		\| u(t)\,h_\lambda (0,t) \|_{L^{2r}_t L^{\frac{dr}{r-1}}_x} 
		+\| \nabla u(t) \,h_\lambda (0,t) \|_{L^{2r}_t L^{\frac{dr}{2r-1}}_x},
	\end{equation*}
	where, for all $0\leq s\leq t\leq \infty$, 
	\begin{equation}\label{definition_h}
	\begin{aligned}
		h_\lambda (s,t):= \exp
		\bigg\{
			&-\lambda\Big( \int_s^t \| v(\tau)\|_{L^{\frac{dr}{r-1}}_x}^{2r} 
			+ \int_s^t \| \nabla v(\tau)\|_{L^{\frac{dr}{2r-1}}_x}^{2r}
			+ \int_s^t \| \nabla v(\tau)\|_{L^{q}_x}^{2r}\Big)
		\Big\}\leq 1.
	\end{aligned}
	\end{equation}	 		
	Let $\Psi$ be the operator from $Y_r$ to itself, such that, for all $\omega\in Y_r$, $\Psi(\omega)$ is the velocity of the weak solution of
	\begin{equation*}
		\begin{cases}
			\partial_t u  -\Delta u +\nabla \Pi
											=g(\omega)+f_1+\Div f_2 			& \RR_+ \times\RR^d,\\
			\Div u = 0															& \RR_+ \times\RR^d,\\
			u_{|t=0} = \bar{u}													&\;\;\quad \quad \RR^d.\\
		\end{cases}
	\end{equation*}	
	Let us prove that, for a good choice of $\lambda$, $\Psi$ is a contraction on $Y_r$.
	First of all, for all $\omega_1,\,\omega_2\in Y_r$, the difference $\delta\Psi:=\Psi(\omega_1)-\Psi(\omega_2)$ is the velocity field of the weak solution of
	\begin{equation*}
		\begin{cases}
			\partial_t \delta\Psi  -\Delta \delta\Psi +\nabla \Pi
											=g(\delta \omega)		 			& \RR_+ \times\RR^d,\\
			\Div\, \delta\Psi = 0															& \RR_+ \times\RR^d,\\
			\delta \Psi_{|t=0} = 0															&\;\;\quad \quad \RR^d,\\
		\end{cases}
	\end{equation*}
	where $\delta \omega := \omega_1-\omega_2$. Since the Mild formulation yields
	\begin{equation*}
		\delta\Psi(t)= \int_{0}^t e^{(t-s)\Delta}\PP g(\delta \omega)(s)\dd s,
	\end{equation*}
	then, by the definition \eqref{def_g_apx} of $g$, Lemma \ref{Lemma_A.1} and Lemma \ref{Lemma_A.2}	the following inequality is fulfilled:
	\begin{equation*}
		\|\delta\Psi\|_\lambda\leq
		\frac{C}{\lambda^{4r}}\Big\{
			\| v \|_{L^{2r}_t L^{\frac{dr}{r-1}}_x}^\frac{1}{2}
			\| \delta\nabla\omega (t)h(0,t)\|_{L^{2r}_t L^{\frac{dr}{2r-1}}_x}+
			\| \nabla v\|_{L^{2r}_t L^{\frac{dr}{2r-1}}_x}^{\frac{1}{2}}
			\| \delta\omega(t)h(0,t)\|_{L^{2r}_t L^{\frac{dr}{r-1}}_x}\Big\}.
	\end{equation*}
	Imposing $\lambda>0$ big enough we finally obtain $\|\delta\Psi\|_\lambda\leq \|\delta\omega\|_\lambda/2$,	namely $\Psi$ is a contraction on $Y_r$.
	Then, by the Fixed-Point Theorem, there exists a function $u$ in $Y_r$ such that, $u$ is the velocity field of the weak solution $(u,\Pi)$ of 
	\eqref{Stokes_system_with_linear_perturbation_g}. Let us remark that $\nabla u$ belongs also to $L^{r}_tL^{dr/(2r-2)}_x$. Indeed $\nabla u$ is formulated by
	\begin{equation*}
	\begin{aligned}
		\nabla u(t):= e^{t \Delta}\nabla \bar{u}\,+&\int_{0}^t \nabla e^{(t-s)\Delta}\PP
		\left(f_1(s)+g(u)(s)\right)\dd s\,+\\
		&-\int_{0}^t \Delta  e^{(t-s)\Delta}R R R\cdot R\cdot f_2(s)\dd s-
		\int_{0}^t \Delta e^{(t-s)\Delta} R\cdot R\cdot f_2(s)\dd s,
	\end{aligned}
	\end{equation*}
	then the result holds thanks to Corollary \ref{Characterization_of_hom_Besov_spaces}, Lemma \ref{Lemma1} and Theorem \ref{Maximal_regularity_theorem}. 
	Finally, recalling that $\Pi$ is determined by
	\begin{equation*}
		\Pi := 
			-\left(-\Delta\right)^{-\frac{1}{2}}R\cdot \left(f_1 + g(u)\right)
			-R\cdot R\cdot f_2,
	\end{equation*}
	we deduce that $\Pi$ belongs to $L^{r}_t L^{dr/(2r-2)}_x$, by Corollary \ref{Hardy-Litlewood-Sobolev-Inequality}.
\end{proof}
\begin{remark}
	If we add a small extra regularity on $\bar{u}$ in Theorem \ref{Theorem_stokes_system_with_linear_perturbation} assuming $\bar{u}$ in 
	$\BB_{p,r}^{d/p-1+\ee}$, with $\ee<\min\{ 1/(2r), 1-1/r, 2(d/p-2+1/r)\}$, the weak solution $(u,\,\Pi)$ fulfills also
	\begin{equation*}
		u\in L^{2r}_t L^{\frac{2dr}{(2-\ee)r-2}}_x\cap L^{\frac{4r}{2-\ee r}}_t L^\frac{dr}{r-1}_x \quad\text{with}\quad 
		\nabla u\in L^{2r}_t L^{\frac{2dr}{(4-\ee)r-2}}_x\cap L^{\frac{4r}{2-\ee r}}_t L^\frac{dr}{2r-1}_x.
	\end{equation*}
\end{remark}

\pagestyle{empty}
\bibliographystyle{amsplain}

\begin{thebibliography}{10}

\bibitem{MR2290277}
H.~Abidi and T.~Hmidi, \emph{On the global well-posedness for {B}oussinesq
  system}, J. Differential Equations \textbf{233} (2007), no.~1, 199--220.
  \MR{2290277 (2007k:35365)}

\bibitem{2013arXiv1301.2371A}
H.~{Abidi} and P.~{Zhang}, \emph{{On the global well-posedness of 2-D
  density-dependent Navier-Stokes system with variable viscosity}}, ArXiv
  e-prints (2013).

\bibitem{MR2336833}
Hammadi Abidi and Marius Paicu, \emph{Existence globale pour un fluide
  inhomog\`ene}, Ann. Inst. Fourier (Grenoble) \textbf{57} (2007), no.~3,
  883--917. \MR{2336833 (2008g:35160)}

\bibitem{MR2768550}
Hajer Bahouri, Jean-Yves Chemin, and Rapha{\"e}l Danchin, \emph{Fourier
  analysis and nonlinear partial differential equations}, Grundlehren der
  Mathematischen Wissenschaften [Fundamental Principles of Mathematical
  Sciences], vol. 343, Springer, Heidelberg, 2011. \MR{2768550 (2011m:35004)}

\bibitem{MR0482275}
J{\"o}ran Bergh and J{\"o}rgen L{\"o}fstr{\"o}m, \emph{Interpolation spaces.
  {A}n introduction}, Springer-Verlag, Berlin-New York, 1976, Grundlehren der
  Mathematischen Wissenschaften, No. 223. \MR{0482275 (58 \#2349)}

\bibitem{MR2759829}
Haim Brezis, \emph{Functional analysis, {S}obolev spaces and partial
  differential equations}, Universitext, Springer, New York, 2011. \MR{2759829
  (2012a:35002)}

\bibitem{MR2227730}
Dongho Chae, \emph{Global regularity for the 2{D} {B}oussinesq equations with
  partial viscosity terms}, Adv. Math. \textbf{203} (2006), no.~2, 497--513.
  \MR{2227730 (2007e:35223)}

\bibitem{2013arXiv1304.3235D}
R.~{Danchin} and P.~{Zhang}, \emph{{Inhomogeneous Navier-Stokes equations in
  the half-space, with only bounded density}}, ArXiv e-prints (2013).

\bibitem{MR3134744}
Rapha{\"e}l Danchin and Lingbing He, \emph{The {O}berbeck-{B}oussinesq
  approximation in critical spaces}, Asymptot. Anal. \textbf{84} (2013),
  no.~1-2, 61--102. \MR{3134744}

\bibitem{MR2484939}
Rapha{\"e}l Danchin and Piotr~Bogus{\l}aw Mucha, \emph{A critical functional
  framework for the inhomogeneous {N}avier-{S}tokes equations in the
  half-space}, J. Funct. Anal. \textbf{256} (2009), no.~3, 881--927.
  \MR{2484939 (2009j:35248)}

\bibitem{MR3017294}
\bysame, \emph{Incompressible flows with piecewise constant density}, Arch.
  Ration. Mech. Anal. \textbf{207} (2013), no.~3, 991--1023. \MR{3017294}

\bibitem{MR2520505}
Rapha{\"e}l Danchin and Marius Paicu, \emph{Global well-posedness issues for
  the inviscid {B}oussinesq system with {Y}udovich's type data}, Comm. Math.
  Phys. \textbf{290} (2009), no.~1, 1--14. \MR{2520505 (2010f:35298)}

\bibitem{MR1642033}
J.~I. D{\'{\i}}az and G.~Galiano, \emph{Existence and uniqueness of solutions
  of the {B}oussinesq system with nonlinear thermal diffusion}, Topol. Methods
  Nonlinear Anal. \textbf{11} (1998), no.~1, 59--82. \MR{1642033 (99j:35174)}

\bibitem{MR1022305}
R.~J. DiPerna and P.-L. Lions, \emph{Ordinary differential equations, transport
  theory and {S}obolev spaces}, Invent. Math. \textbf{98} (1989), no.~3,
  511--547. \MR{1022305 (90j:34004)}

\bibitem{MR2763332}
T.~Hmidi, S.~Keraani, and F.~Rousset, \emph{Global well-posedness for
  {E}uler-{B}oussinesq system with critical dissipation}, Comm. Partial
  Differential Equations \textbf{36} (2011), no.~3, 420--445. \MR{2763332
  (2012a:76037)}

\bibitem{MR2305876}
Taoufik Hmidi and Sahbi Keraani, \emph{On the global well-posedness of the
  two-dimensional {B}oussinesq system with a zero diffusivity}, Adv.
  Differential Equations \textbf{12} (2007), no.~4, 461--480. \MR{2305876
  (2009c:35404)}

\bibitem{2012arXiv1212.3918H}
J.~{Huang} and M.~{Paicu}, \emph{{Decay estimates of global solutions to 2D
  incompressible inhomogeneous Navier-Stokes equations with variable
  viscosity}}, ArXiv e-prints (2012).

\bibitem{MR3056619}
Jingchi Huang, Marius Paicu, and Ping Zhang, \emph{Global well-posedness of
  incompressible inhomogeneous fluid systems with bounded density or
  non-{L}ipschitz velocity}, Arch. Ration. Mech. Anal. \textbf{209} (2013),
  no.~2, 631--682. \MR{3056619}

\bibitem{MR1938147}
P.~G. Lemari{\'e}-Rieusset, \emph{Recent developments in the {N}avier-{S}tokes
  problem}, Chapman \& Hall/CRC Research Notes in Mathematics, vol. 431,
  Chapman \& Hall/CRC, Boca Raton, FL, 2002. \MR{1938147 (2004e:35178)}

\bibitem{MR0145789}
John~M. Mihaljan, \emph{A rigorous exposition of the {B}oussinesq
  approximations applicable to a thin layer of fluid}, Astrophys. J.
  \textbf{136} (1962), 1126--1133. \MR{0145789 (26 \#3317)}

\bibitem{MR2822227}
Chao Wang and Zhifei Zhang, \emph{Global well-posedness for the 2-{D}
  {B}oussinesq system with the temperature-dependent viscosity and thermal
  diffusivity}, Adv. Math. \textbf{228} (2011), no.~1, 43--62. \MR{2822227
  (2012g:35264)}

\end{thebibliography}
\providecommand{\bysame}{\leavevmode\hbox to3em{\hrulefill}\thinspace}
\providecommand{\MR}{\relax\ifhmode\unskip\space\fi MR }
\providecommand{\MRhref}[2]{%
  \href{http://www.ams.org/mathscinet-getitem?mr=#1}{#2}
}
\providecommand{\href}[2]{#2}

\end{document}